\documentclass[reqno,11pt]{amsart}
\usepackage{amsmath,amssymb,amsfonts}

\usepackage{tikz}
\usetikzlibrary{matrix,arrows}
\usepackage[pagebackref,colorlinks,citecolor=red,linkcolor=blue]{hyperref}

\usepackage{fouriernc}

\title[QC-sheaves on a geometric stack]{A functorial formalism for quasi-coherent sheaves on a geometric stack}

\author[L. Alonso]{Leovigildo Alonso Tarr\'{\i}o}
\address[L. A. T.]{Departamento de \'Alxebra\\
Facultade de Matem\'a\-ticas\\
Universidade de Santiago de Compostela\\
E-15782  Santiago de Compostela, Spain}
\email{leo.alonso@usc.es}

\author[A. Jerem\'{\i}as]{Ana Jerem\'{\i}as L\'opez}
\address[A. J. L.]{Departamento de \'Alxebra\\
Facultade de Matem\'a\-ticas\\
Universidade de Santiago de Compostela\\
E-15782  Santiago de Compostela, Spain}
\email{ana.jeremias@usc.es}

\author[M. P\'erez]{Marta P\'erez Rodr\'{\i}guez}
\address[M. P. R.]{Departamento de Matem\'a\-ticas\\
Esc. Sup. de Enx. Inform\'atica,
Campus de Ourense\\
Universidade de Vigo\\
E-32004 Ourense, Spain}
\email{martapr@uvigo.es}

\author[M. J. Vale]{Mar\'{\i}a J. Vale Gonsalves}
\address[M. J. V.]{Departamento de \'Alxebra\\
Facultade de Matem\'a\-ticas\\
Universidade de Santiago de Compostela\\
E-15782  Santiago de Compostela, Spain}
\email{mj.vale@usc.es}

\thanks{This work has been partially supported by
Spain's MEC and E.U.'s FEDER research projects MTM2008-03465 and MTM2011-26088,
together with Xunta de Galicia's
PGIDIT10PXIB207144PR and GRC2013-045
}


\date{\today}

\theoremstyle{plain}
\newtheorem{thm}{Theorem}[section]
\newtheorem{lem}[thm]{Lemma}
\newtheorem{cor}[thm]{Corollary}
\newtheorem{prop}[thm]{Proposition}

\theoremstyle{remark}
\newtheorem*{rem}{Remark}

\theoremstyle{definition}

\newtheorem*{ack}{Acknowledgements}
\newtheorem*{Ariana}{An Ariadne's thread for quasi-coherent sheaves on stacks}
\newtheorem*{Description}{Outline of the paper}
\newtheorem{cosa}[thm]{}

\numberwithin{equation}{thm}

\newcommand{\CF}{{\mathcal F}}
\newcommand{\CG}{{\mathcal G}}

\newcommand{\CK}{{\mathcal K}}

\newcommand{\CM}{\mathcal{M}}

\newcommand{\CO}{\mathcal{O}}
\newcommand{\CP}{\mathcal{P}}

\newcommand{\SB}{\mathsf{B}}
\newcommand{\SC}{\mathsf{C}}
\newcommand{\SD}{\mathsf{D}}

\newcommand{\SG}{\mathsf{G}}

\renewcommand{\SS}{\mathsf{S}}
\newcommand{\SST}{\mathsf{T}}

\newcommand{\SU}{\mathsf{U}}

\newcommand{\SW}{\mathsf{W}}
\newcommand{\SX}{\boldsymbol{\mathsf{X}}}
\newcommand{\SY}{\boldsymbol{\mathsf{Y}}}
\newcommand{\SZ}{\boldsymbol{\mathsf{Z}}}

\newcommand{\fs}{\mathsf{f}}
\newcommand{\sr}{\mathsf{r}}
\newcommand{\md}{\text{-}\mathsf{Mod}}
\newcommand{\com}{\text{-}\mathsf{coMod}}

\newcommand{\qc}{\mathsf{qc}}

\newcommand{\op}{\mathsf{o}}
\newcommand{\ab}{\mathsf{Ab}}
\newcommand{\set}{\mathsf{Set}}

\newcommand{\ZZ}{\mathbb{Z}}

\newcommand{\dirlim}[1]{\begin{array}[t]{c} {\rm lim}\\[-7.5 pt]
 {\longrightarrow} \\[-7.5 pt] {\scriptstyle {#1}} \end{array}}
\newcommand{\invlim}[1]{\begin{array}[t]{c} {\rm lim}\\[-7.5 pt]
 {\longleftarrow} \\[-7.5 pt] {\scriptstyle {#1}} \end{array}}

\newcommand*{\longrightrightarrows}{
\ensuremath{%
\makebox[0pt][l]
{\raisebox{0.3ex}{\ensuremath{\longrightarrow}}}%
\raisebox{-0.3ex}{\ensuremath{\longrightarrow}}%
}}

\newcommand*{\longleftrightarrows}{
\ensuremath{%
\makebox[0pt][l]
{\raisebox{0.3ex}{\ensuremath{\longleftarrow}}}%
\raisebox{-0.3ex}{\ensuremath{\longrightarrow}}%
}}

\newcommand{\lto}{\longrightarrow}
\newcommand{\xto}{\xrightarrow}

\DeclareMathOperator{\iso}{\tilde{\to}}

\DeclareMathOperator{\liso}{\tilde{\lto}}

\newcommand{\imp}{\Rightarrow}

\DeclareMathOperator{\Hom}{Hom}

\DeclareMathOperator{\Img}{Im}

\DeclareMathOperator{\spec}{Spec}

\DeclareMathOperator{\id}{id}
\DeclareMathOperator{\iid}{\mathsf{id}}

\DeclareMathOperator{\qco}{\mathsf{Qco}}
\DeclareMathOperator{\fun}{\mathsf{Fun}}
\DeclareMathOperator{\cart}{\mathsf{Crt}}
\DeclareMathOperator{\car}{\mathsf{crt}}
\DeclareMathOperator{\des}{\mathsf{Desc}}
\DeclareMathOperator{\stack}{\mathsf{Stck}}

\DeclareMathOperator{\pre}{\mathsf{Pre}}
\DeclareMathOperator{\aff}{\mathsf{Aff}}
\DeclareMathOperator{\pin}{\mathsf{p}}
\DeclareMathOperator{\et}{\mathsf{\acute{e}t}}
\DeclareMathOperator{\ff}{\mathsf{fppf}}
\DeclareMathOperator{\ring}{\mathsf{Ring}}

\newcommand{\GGamma}{\boldsymbol{\Gamma}}
\newcommand{\wadj}[1]{{#1}^\mathsf{a}}

\newcommand{\ie}{{\it i.e.\/}}

\begin{document}

\begin{abstract} 
A geometric stack is a quasi-compact and semi-separated algebraic stack. We prove that the quasi-coherent sheaves on the small flat topology, Cartesian presheaves on the underlying category, and comodules over a Hopf algebroid associated to a presentation of a geometric stack are equivalent categories. As a consequence, we show that the category of quasi-coherent sheaves on a geometric stack is a Grothendieck category.

We also associate, in a 2-functorial way, to a 1-morphism of geometric stacks $f \colon \SX \to \SY$, an adjunction $f^* \dashv f_*$ for the corresponding categories of quasi-coherent sheaves that agrees with the classical one defined for schemes. This construction is described both geometrically in terms of the small flat site and algebraically in terms of comodules over the Hopf algebroid.
\end{abstract}

\maketitle
\tableofcontents

\section*{Introduction}

The theory of quasi-coherent sheaves on algebraic stacks has suffered from an unbalanced record in the published literature. Our initial aim was to develop the basic cohomological properties of quasi-coherent sheaves on algebraic stacks. The corresponding chapters of \cite{stack} were not available when we started to study these problems, and the literature presented some gaps. This led us to pursue a general approach to quasi-coherent sheaves with the cohomology formalism in mind. The result is this paper, that has thus a semi-expository nature. Some of the results here are accessible either on unpublished papers or references with different proofs or employing a language that, in our opinion, makes the connection with the classical literature difficult. The paper also contains some new approaches to the functoriality behaviors of quasi-coherent sheaves on algebraic stacks.

The setting of this paper is to define a category of quasi-coherent sheaves through a \emph{small} ringed site by mimicking the original definition given in \cite[(5.1.3)]{ega1} with a view towards an extension of \cite{ahst} to stacks. As we look toward cohomological properties, we restrict to a certain class of algebraic stacks, the geometric stacks ---we will discuss this choice later in the introduction. We will show that to such an algebraic stack $\SX$ one can associate an abelian category of $\qco(\SX)$ that is Grothendieck, \ie{} cocomplete with exact filtered direct limits and possessing a generator. Moreover this category agrees trivially with the usual one when $\SX$ is equivalent to a scheme. Further we explore the usual variances with respect to 1-morphisms and the corresponding 2-functorial properties.

What are the differences between this approach and the previous ones already in the literature? The very notion of quasi-coherent sheaf on stacks usually does not refer to a ringed site but it is often mistook with a related notion, that of Cartesian \emph{pre}sheaf. Both are equivalent for geometric stacks as we show. The sheaf approach suffers a major technical problem, namely, the lack of functoriality of the \emph{lisse-\'etale} site. 

In \cite{O}, a solution is presented for the category of quasi-coherent sheaves defined in \cite{lmb}. Unfortunately not all the results that we need are contained in this reference. We also develop a treatment in our setting of the adjoint functors $f^* \dashv f_*$ associated to a 1-morphism $f \colon \SX \to \SY$ together with its 2-functorial properties. Note the description in \cite[3.3]{O} of $f_*$ is slightly incorrect because the 1-morphism $f$ would have to be representable by schemes; it is easily repaired by considering the approach in \cite{lmb} or in \cite{six}. We provide a full detailed construction of both functors. We give also an explicit description of these functors via Hopf algebroids (Corollary \ref{adjHoveyInv}).

We should mention that in the \emph{Stacks Project} \cite{stack} a construction similar to ours is made but using \emph{big} flat sites. The drawback of this approach is that in order to the category of sheaves of modules over a big (ringed) site to possess a generator, it is necessary to make the choice of a universe. This construction makes the sites considered automatically functorial but one has to check the invariance of universe. We refrain from using Grothendieck universes by using sets and classes \emph{a la} Von Neumann-G\"odel-Bernays but we had to wrestle with the non functoriality of the small sites. This paper thus was developed mostly independently of \cite{stack}, that is why we refer to \cite{lmb} for all the basic definitions on stacks.

\begin{Description}
Let us explain first the reason why we assume always that the stacks we consider are quasi-compact and semi-separated, \ie{} with affine diagonal morphism. These are called geometric stacks in \cite{lur}. On schemes, this condition is the one that reveals itself as \emph{indispensable} to have a nice cohomological behavior. Moreover we may attach a site made of \emph{affine} schemes, much as we may use a basis of affine open subsets on quasi-compact, semi-separated schemes, so it makes the theory simpler. And a third reason is that they may be represented by a purely algebraic object, a Hopf algebroid.

This hypothesis allows us to give a shorter proof of the existence of a generator, without choosing either a big cardinal as in Gabber's proof or appealing to the finite presentation of the category. This shortcut should be regarded as a consequence of our restricted setting.

Notice that sometimes the hypothesis of semi-separateness has a reputation of being difficult to check. Fortunately, recent results of Hall (see \cite{hacoh}, especially Theorem D) ``can also be used to show that many algebraic stacks of interest have affine diagonals''. For examples of this principle at work, see, for instance, \cite{haop} and \cite{hr}.

For such a stack $\SX$ we consider a certain variant of the small flat site, that we will denote by $\SX_{\ff}$ (see \ref{site}). Its associated topos is not functorial, but still, it has enough functoriality properties to allow us to develop a useful formalism. Notice that it is finer than the classical \emph{lisse-\'etale} site, but has the advantage that a presentation yields a covering for every object in the small flat topology. This allows for simpler proofs and the category of quasi-coherent sheaves does not change.

In the affine case there are three ways of understanding a quasi-coherent sheaf, namely, 
\begin{itemize}
 \item as a sheaf of modules which locally admits a presentation,
 \item as a cartesian presheaf\footnote{This property corresponds essentially to the fact that a localization of a module at an element of the ring may be computed as a tensor product.}, and
 \item as a module over the ring of global sections.
\end{itemize}
These three aspects are, in some sense, present also in the global case. In this paper we show that they arise quite naturally when we consider a geometric stack. For the small flat site, cartesian presheaves of modules are sheaves (in fact, for any topology coarser than the \textsf{fpqc}) and moreover they are the same as quasi-coherent sheaves. This relation is described for the small flat site of an affine scheme in section \ref{sec2}, after establishing in section \ref{sec1} the general properties of ringed sites that we will need.

In section \ref{sec3} we transport all the previous discussion to the setting of geometric stacks. A geometric stack may be described by a Hopf algebroid, as the stackification of its associated affine groupoid scheme. We prove in Theorem \ref{agree} that quasi-coherent sheaves and cartesian presheaves agree on a geometric stack.

Next we develop in several steps an equivalence of categories between comodules over a Hopf algebroid and quasi-coherent sheaves on its associated geometric stack. In section \ref{sec4}, we see that quasi-coherent sheaves on a geometric stack correspond to descent data for quasi-coherent sheaves on an affine cover (Theorem \ref{p52}). Once we are in the affine setting, we see that a descent datum corresponds to certain supplementary structure on its module of global sections. 

By an algebraic procedure, in section \ref{sec5}, we show that this structure can be transformed into the structure of a comodule, establishing a further equivalence of categories. All these equivalences combined provide the equivalence between quasi-coherent sheaves on the geometric stack and comodules over the Hopf algebroid (Corollary \ref{descrt}). An analogous result is due by Hovey under a different setting, as we will discuss later. An important consequence of this equivalence is that the category of quasi-coherent sheaves on a geometric stack is Grothendieck.

The motivation for the results in the present paper is to have a convenient formalism for doing cohomology. So, it is crucial to have a good functoriality. We devote section \ref{sec6} to this issue. Specifically, we show that a map between geometric stacks induces a pair of adjoint functors between the corresponding categories of quasi-coherent sheaves. This construction is not straightforward because, as we remarked before, there is no underlying map of topoi and the definition of the direct image has extra complications. Moreover, we show that these constructions are 2-functorial. One needs to check in addition that 2-cells induce a natural transformation between their associated functors, a feature of stacks.

In section \ref{sec7}, we express the adjunction in a purely algebraic way. The benefit of this result is to put ourselves in the most convenient setting for approaching cohomological questions. It is noteworthy that for a 1-morphism of geometric stacks $f \colon \SX \to \SY$, the functor $f_*$ has a simpler geometric description while its adjoint $f^*$ is easier to describe algebraically.

In the final section \ref{sec8} we compare our functoriality formalism with the classical one. Let us recall that the category quasi-coherent sheaves on a scheme does not change if we replace the Zariski by the \'etale topology. Moreover, the \'etale topology gives a reasonable topos for a geometric Deligne-Mumford stack. So in this setting there is a couple of adjoint functors that restrict to the classical ones in the scheme case and agree with the ones just defined in the general case of a geometric stack. This is shown in Propositions \ref{p72}, \ref{p73} and Corollary \ref{final}.
\end{Description}

\begin{Ariana}
It only seems pertinent to give some inputs about the story of the development of the theory of quasi-coherent sheaves on stacks. We do not pretend to be exhaustive, and we will content ourselves with putting in perspective the references that lead to the present work while commenting the relationship with the notions employed here.

After the definition by Grothendieck (possibly inspired by Cartier) of quasi-coherent sheaf on ringed spaces in \cite[(5.1.3)]{ega1}, there did not seem to be studied in general contexts. One possible reason is that the category of quasi-coherent sheaves constitutes the closure of the category of coherent sheaves for direct limits and this is not true for analytical varieties or noetherian formal schemes.

There are however, two remarkable works in which this notion was considered in a kind of general setting. The earliest is Hakim's thesis. The notion of relative quasi-coherence with respect to a base topos is studied. Relative quasi-coherence agrees with the usual one when the base topos is the Zariski topos of a commutative ring \cite[IV, Remarque (4.13)]{hakt}. In a quite different vein, Orlov has developed a notion of quasi-coherence with respect to a topos with a view toward non-commutative geometry. His approach is, essentially, to define quasi-coherent sheaves as those sheaves whose underlying presheaf is Cartesian \cite[Definition 2.2]{orlov}.

Quite possibly, the first explicit definition of quasi-coherence on a stack is by Vistoli, \cite[(7.18)]{vit}. For a geometric stack $\SX$, his notion essentially agrees with our $\des_{\qc}(X/\SX)$ for a given presentation $X \to \SX$. The classical reference \cite{lmb} defines quasi-coherent sheaf as a mixture of Cartesian sheaf plus a local condition of quasi-coherence. Besides, it overlooks the lack of functoriality of the \emph{lisse-\'etale} site \cite[\href{http://stacks.math.columbia.edu/tag/07BF}{Section 07BF}]{stack}. The authoritative reference is Olsson \cite{O}. As we remarked previously, this reference does not contain all the results we need for the next step in our program.

Parallel to this story, another point of view has arisen with origin in Hopkin's vision on unifying stable homotopy theory with algebraic geometry to obtain information about number theoretic questions deeply rooted in complex oriented cohomology theories. The point of departure is Hopkins seminar \cite{hop}, where the idea of representing geometric stacks through Hopf algebroids arose.

An important step in this point of view is given by Hovey's papers \cite{hmh}, \cite{hov}. His definition of quasi-coherent sheaf is in fact the notion of Cartesian pre\-sheaf, as he considers mostly the discrete topology. Thus, for a geometric stack $\SX$, his notion essentially agrees with our category $\cart(\SX)$. So, there is a great deal of overlapping in the corresponding parts of this paper and Hovey's work. We will point it out along the exposition. This point of view is also embraced by Goerss in \cite{goe} who also defines quasi-coherent sheaves through the Cartesian condition.

We have to mention also Pribble's Ph.~D. thesis \cite{P} whose point of view ultimately refers to the \emph{yoga} ``coalgebras on a cotriple are descent data'' developed in \cite{BR} in a general setting. We have used several times the exact sequence associated to the monad (or triple) corresponding to a presentation, simplifying some arguments. An important work with homotopical motivation is  Naumann's \cite{N} who identifies the category of quasi-coherent sheaves as defined in \cite{lmb} with the category $\des_{\qc}(X/\SX)$, defined below (see \ref{desqc}), by invoking the classical results in \cite[Chapter 6]{blr}.

Independently of the present work, the future main reference for stacks \cite{stack} already has a presentation of the basic properties of quasi-coherent sheaves on an algebraic stack, in the setting of big sites as we said before. Moreover, they work with a notion of algebraic stack with very few restrictions on the diagonal. In this general setting a sheaf is quasi-coherent if, and only if, it is Cartesian and locally quasi-coherent. They work over big sites for which functoriality follows easy. In our opinion, this sacrifices other useful features that provide a small site, especially in the somewhat restricted setting of geometric stacks. Their choice also implies the use of Grothendieck universes, something we do not afford to use. We do not intend to replace this approach but rather to complement it.

In recent times there have been further work on the properties of the category of quasi-coherent sheaves on an algebraic stack. We will just mention the work of Hall, Neeman, Rydh and Sch\"appi, to name a few that we are aware of.

\end{Ariana}

\begin{ack}
We thank M. Olsson for the useful exchange concerning the algebraic structure of inverse images from which the consideration of Lemma \ref{Olss} arose. We thank T. Sitte for his comments on the text and conversations on the topics of the paper. We also thank the referee for comments and suggestions that allowed us to improve the paper.

\end{ack}

\section{Sheaves on ringed sites}\label{sec1}

\setcounter{thm}{-1}

\begin{cosa}
In this paper we will use as underlying axiomatics that of von Neumann, G\"odel and Bernays using sets and classes. A category $\SC$ has a class of objects and a class of arrows but the class of maps between two objets $A, B \in \SC$, denoted $\Hom_\SC(A,B)$, is always assumed to be a set.
For readers fond of Bourbaki-Grothendieck universes (and for comparison with results in \cite{sga41}) fix two universes $\mathbb{U} \in \mathbb{V}$ and identify elements of $\mathbb{U}$ with sets and elements of $\mathbb{V}$ with classes. For the basic properties of fibered categories and stacks, we have used \cite{vgt} as a reference.
\end{cosa}

\begin{cosa} \textbf{Ringed categories}.
Let $\SC$ be a small\footnote{Its class of objects is a set, not a proper class.} category. We will denote by $\pre(\SC) := \fun(\SC^\op, \ab)$ the category of contravariant functors from $\SC$ to abelian groups. We will usually refer to objects in this category as \emph{presheaves on} $\SC$. We will say that $(\SC, \CO)$ is a ringed category if $\CO$ is a ring object in $\pre(\SC)$. In other words $\CO \colon \SC^\op \to \ring$ is a presheaf, or else $\CO$ is a presheaf together with a couple of maps $\cdot \colon \CO \otimes \CO \to \CO$ and $1 \colon \ZZ_{\SC} \to \CO$, with $\ZZ_\SC$ denoting the constant presheaf with value $\ZZ$; these data make commutative the usual diagrams expressing the associativity and commutativity of the product $\cdot$ and the fact that $1$ is a unit.
\end{cosa}

\begin{cosa} \textbf{Modules}.
Let $(\SC, \CO)$ be a ringed category. Denote by $\pre(\SC,\CO)\md$ or simply by $\pre\text{-}\CO\md$, if no confusion arises, the category of presheaves $\CM \in \pre(\SC)$ equipped with an action of $\CO$, $\CO \otimes \CM \to \CM$, that makes commutative the diagrams expressing the fact that the sections of $\CM$ are modules over the sections of $\CO$ and the restriction maps are linear. We will say that $\CM$  is an $\CO$-Module.
\end{cosa}

\begin{cosa} \label{superadj} \textbf{Cartesian presheaves}.
Let $(\SC, \CO)$ be a ringed category. We will say that an $\CO$-module $\CM$ is a \emph{Cartesian} presheaf if for every morphism $f:X \to Y$ in $\SC$ the $\CO(X)$-linear map
\[
\wadj{\CM(f)}: \CO(X) \otimes_{\CO(Y)} \CM(Y) \lto \CM(X)
\]
adjoint to the $\CO(Y)$-linear map $\CM(f):\CM(Y) \to \CM(X)$ is an isomorphism. We will denote the category of Cartesian presheaves on $(\SC,\CO)$ by $\cart(\SC, \CO)$.
\end{cosa}

\begin{cosa} \textbf{Ringed sites and topoi}.
As usual we will call a small category $\SC$ a \emph{site} if it is equipped with a Grothendieck topology $\tau$ \cite[Exp. II \S1]{sga41}. If it is necessary to make it explicit, we will denote the site by $(\SC, \tau)$, otherwise we will keep the notation $\SC$. A site has an associated \emph{topos}, its category of sheaves of sets. For a site $(\SC, \tau)$ we will denote its associated topos by $\SC_\tau$. We may also consider sheaves with additional algebraic structures, in particular, a sheaf of rings $\CO$. We will call the pair $(\SC, \CO)$ a \emph{ringed site} and $(\SC_\tau, \CO)$ a \emph{ringed topos}. We will always refer with the notation $(\SC_\tau,\CO)\md$ or simply $\CO\md$ to the category of \emph{sheaves} of modules over the sheaf of rings $\CO$, and not to the biggest category of presheaves of modules.
\end{cosa}

%

\begin{cosa}\label{comma}
Let $\SC$ be a site  and $X$ and object in $\SC$. We will consider the category whose objects are morphisms $u:U \to X$, and morphisms are commutative triangles.\footnote{Sometimes called \emph{the comma category}.} We will denote such an object by $(U,u)$, or simply by $U$ if no confusion arises. The category $\SC/X$ is a site. The coverings of the corresponding topology of an object $U$
are the families of morphisms $\{f_i:U_i \to U\}$ in $\SC /X$ that are coverings in $\SC$. If $(\SC,\CO)$ is a ringed site, then $(\SC/X,\CO|_X)$ is a ringed site with the induced sheaf of rings $\CO|_X$ given by
 $\CO|_X(U,u)=\CO(U)$. If $\CM$ is a sheaf of $\CO$-Modules on the site $(\SC,\CO)$ we define the sheaf restriction of $\CM$ to $\SC/X$, denoted $\CM|_X$, as the sheaf of $\CO|_X$-Modules on the site $(\SC/X,\CO|_X)$ given by $\CM|_X(U,u)=\CM(U)$.
\end{cosa}

\begin{cosa} \textbf{Quasi-coherent sheaves}.
Let $(\SC,\CO)$ be a ringed site with a topology $\tau$. We will say that a sheaf of $\CO$-Modules $\CM$ is quasi-coherent if for every object $X \in \SC$ there is a covering $\{p_i \colon X_i \to X\}_{i \in L}$ such
that for the restriction of $\CM$ to each $\SC/X_i$, there is a
presentation
\[
\CO|_{X_i}^{(I)} \lto \CO|_{X_i}^{(J)} \lto \CM|_{X_i} \lto 0,
\]
where by $\CF^{(I)}$ we denote a coproduct of copies of a sheaf $\CF$ indexed by some set $I$. If $\SC$ is the site of open sets of a classical topology, this definition agrees with \cite[\textbf{0}, (5.1.3)]{ega1}. We will denote the category of quasi-coherent sheaves on $(\SC,\CO)$ by $\qco(\SC_\tau, \CO)$ as a full subcategory of $\CO\md$. It is clearly an additive subcategory, and cokernels agree with those of the ambient category. It may happen that kernels do not exist, due to the lack of flatness on the site maps.
\end{cosa}

%

\begin{cosa}\label{adjtop}
Let $\fs \colon (\SC, \tau) \to (\SD, \sigma)$ be a continuous functor of sites \cite[III, (1.1)]{sga41} such that $\fs$ preserves fibered products. In view of \cite[III, (1.6)]{sga41}, the continuity of $\fs$ is equivalent to the fact that it takes coverings in $(\SC, \tau)$ to coverings in $(\SD, \sigma)$.

Whenever $\SC$ possesses colimits, this morphism induces a pair of adjoint functors between the associated topoi
\[
\SD_\sigma\, \underset{\fs_*}{\overset{\fs^{-1}}{\longleftrightarrows}} \,
\SC_\tau.
\]
Let us recall first the definition of $\fs_*$. If $\CG \in \SD_\sigma$ and $V \in \SC$, then
\[
(\fs_*\CG)(V) = \CG(\fs(V));
\]
on morphisms the definition is straightforward. Let now $\CF \in \SC_\tau$, $\fs^{-1}\CF$ is defined as the sheaf associated to the presheaf 
that assigns to each $U \in \SD$ the colimit
\[
\dirlim{\mathbf{I}_U^\fs} \CF(V),
\]
where $\mathbf{I}_U^\fs$ is the category defined as follows. The objects of 
$\mathbf{I}_U^\fs$ are the pairs $(V,g)$ with $V \in \SC$ and $g \colon U \to \fs(V)$ is a morphism in $\SD$. Its morphisms correspond to commutative diagrams
\begin{center}
\begin{tikzpicture}
\matrix(m)[matrix of math nodes, row sep=1.8em, column sep=3em,
text height=1.5ex, text depth=0.25ex]{
V     &      & f(V) \\
      &   U  & \\
V'    &      & f(V')   \\}; 
\path[->,font=\scriptsize,>=angle 90]
(m-1-1) edge node[auto] {$h$}  (m-3-1)
(m-2-2) edge node[above] {$g$}  (m-1-3) 
        edge node[below] {$g'$} (m-3-3)
(m-1-3) edge node[auto] {$f(h)$} (m-3-3);
\end{tikzpicture}
\end{center}
where $ h \colon V \to V'$ is a morphism in $\SC$ such that $\fs(h)\, g = g'$, and that we denote simply by \(h \colon (V,g) \to (V',g')\)
(see \cite[I, 5.1]{sga41} and \cite[III, 1.2.1]{sga41} where $\fs^{-1}$ is denoted $\fs^s$). We extend the definition to morphisms using the same functorial construction. 

In general, $\fs^{-1}$ might not be exact. If $\fs^{-1}$ is exact, we say that the pair $(\fs_*,\fs^{-1})$ is a morphism of topoi from $\SD_\sigma$ to $\SC_\tau$ according to \cite[IV 3.1]{sga41}.
\end{cosa}

We want to extend this formalism to the case of ringed sites. To have the possibility of transporting algebraic structures through the inverse image functor, we need the following lemma, communicated to us by Olsson.

\begin{lem}\label{Olss}
Let $\mathbf{I}$ be a category with finite products and let $F_i \colon \mathbf{I}^\op \to \set$, $i \in \{1 \dots r\}$, be a finite family of functors. The natural map
 \[
\dirlim{\mathbf{I}} (F_1 \times \dots \times F_r) \lto \left( \dirlim{\mathbf{I}} F_1 \right) \times \dots \times \left( \dirlim{\mathbf{I}} F_r \right) 
 \]
is an isomorphism.
\end{lem}

\begin{proof}
By using induction, we may assume that $r = 2$. Notice that 
\[
\dirlim{\mathbf{I}} F_1 \times \dirlim{\mathbf{I}} F_2 
\]
is identical to 
\[
\dirlim{\mathbf{I} \times \mathbf{I}} F_1  \overline{\times} F_2 
\]
where $F_1  \overline{\times} F_2 \colon \mathbf{I}^\op \times \mathbf{I}^\op \to \set$ is the functor that sends a pair $(U, V) \in \mathbf{I}^\op \times \mathbf{I}^\op$ to $F_1(U) \times F_2(V)$. The diagonal functor $\Delta \colon  \mathbf{I}^\op \to \mathbf{I}^\op \times \mathbf{I}^\op$ induces a map
\[
z \colon \dirlim{\mathbf{I}}  {F_1  \times F_2}  \,\lto
\dirlim{\mathbf{I} \times \mathbf{I}}  F_1  \overline{\times} F_2 \,\,=
\dirlim{\mathbf{I}} F_1 \times \dirlim{\mathbf{I}} F_2   
\]

We will check first that $z$ is surjective using the fact that $\mathbf{I}$ has products. Indeed, for any two objects $U, V \in \mathbf{I}$ and $u \in F_1(U)$, $v \in F_2(V)$ the element 
\[
[(u,v)] \in \dirlim{\mathbf{I} \times \mathbf{I}}  F_1  \overline{\times} F_2
\]
equals the class of $(F_1(p_1)(u), F_2(p_2)(v)) \in F_1(U \times V)  \times F_2(U \times V)$ via the natural maps $p_1 \colon U \times V \to U$, $p_2 \colon U \times V \to V$, so any element lies in the image of $z$.

Now, we prove that $z$ is injective.
Consider the functor $\pi \colon \mathbf{I}^\op \times \mathbf{I}^\op \to \mathbf{I}^\op$ that sends $(U, V)$ to $U \times V$ (notice that the product is taken in $\mathbf{I}$). It gives the following natural transformation
\[
F_1  \overline{\times} F_2 \lto (F_1  \times F_2) \circ \pi
\]
defined as
\[
F_1(p_1) \times F_2(p_2) \colon F_1(U)  \times F_2(V) \lto 
F_1(U \times V) \times F_2(U \times V)
\]
on $(U,V) \in \mathbf{I}^\op \times \mathbf{I}^\op$. The functor $\pi$ induces a map 
\[
z' \colon 
\dirlim{\mathbf{I} \times \mathbf{I}} F_1  \overline{\times} F_2 
\,\lto
\dirlim{\mathbf{I}} {F_1  \times F_2} .
\]
To see that $z$ is injective is enough to check that $z' z = \id$. The map $z' z$ sends the class of $(u_1, u_2) \in F_1(U)  \times F_2(U)$ to the class of $(F_1(p_1)(u_1), F_2(p_2)(u_2)) \in F_1(U \times U)  \times F_2(U \times U)$. Let $\delta \colon U \to U \times U$ be the diagonal morphism. As we claimed, the class of $(F_1(p_1)(u_1), F_2(p_2)(u_2))$ equals, applying $F_i(\delta)$ (for $i \in \{1, 2\}$), the class of the element \(
(F_1(p_1 \delta)(u_1), F_2(p_2 \delta)(u_2)) = (u_1, u_2).
\)
\end{proof}

\begin{rem}
Of course, the thesis of the previous lemma holds clearly if we assume $\mathbf{I}$ to be cofiltering, which is the true in many cases of interest. Unfortunately this is neither the case in the topology we are interested in (see the remark after Corollary \ref{adjhaz6}) nor in the case of the \textit{lisse-\'etale} site of an algebraic stack.
\end{rem}

\begin{cosa}\label{adjring}
Let $\fs \colon (\SC, \tau) \to (\SD, \sigma)$ be a continuous functor of sites as in \ref{adjtop}. Let $\CO$ be a sheaf of rings on $(\SC, \tau)$ and $\CP$ a sheaf of rings on $(\SD, \sigma)$. Let us be given a  morphism of sheaves of rings $\fs^\# \colon \CO \to \fs_*\CP$ or, equivalently, its adjoint map $\fs_\# \colon \fs^{-1}\CO \to \CP$.

Consider first the case in which $(\fs_*, \fs^{-1})$ is a morphism of topoi. In this case $((\fs_*,\fs^{-1}), \fs_\#)$ is a morphism of \emph{ringed} topoi \cite[IV, 13.1]{sga41} and provides an adjunction
\[
 \CP\md\, \underset{\fs_*}{\overset{\fs^*}{\longleftrightarrows}}\, \CO\md.
\]
The functor $\fs^* \colon  \CO\md \to  \CP\md$ is defined by
\[
\fs^*\CF := \CP\otimes_{\fs^{-1}\CO}\fs^{-1}\CF
\]
for an $\CO$-module $\CF$. Observe that for a $\CP$-module $\CG$, $\fs_*\CG$ inherits a structure of $\CO$-module via $\fs^\#$ and we have the claimed adjunction $\fs^* \dashv \fs_*$.

In general $\fs^{-1}$ may not be exact, and the previous discussion does not apply. However, if we assume that the category $\mathbf{I}_U^\fs$ has finite products, then by Lemma \ref{Olss}, the functor that defines the inverse image for presheaves commutes with finite products. The sheafification functor commutes also with finite products since it is exact. It follows that $\fs^{-1}$ also commutes with finite products. Thus, $\fs^{-1}\CO$ has canonically the structure of a sheaf of rings and also $\fs_\#$ is a homomorphism. By the same reason, if $\CF$ is an $\CO$-module then $\fs^{-1}\CF$ is an $\fs^{-1}\CO$-module. The
functor $\fs^* \colon  \CO\md \to  \CP\md$ given by
\[
\fs^*\CF := \CP\otimes_{\fs^{-1}\CO}\fs^{-1}\CF
\]
provides an adjunction $\fs^* \dashv \fs_*$ as in the previous situation.

Notice that we use here $\fs^{-1}$ for general sheaves and reserve $\fs^*$ for modules over ringed topoi, contradicting the usage of \cite[IV]{sga41}.
\end{cosa}

\begin{rem}
The perspicuous reader will notice that Lemma \ref{Olss} is also implicitly used in \cite[Def. 3.11]{O}.
\end{rem}

\section{Cartesian presheaves on affine schemes}\label{sec2}

\begin{cosa}
Let $X = \spec(A)$ be an affine scheme. We will work with the
category  $\aff/X$ of flat finitely presented affine schemes over
$X$, in other words the opposite category to the category of flat
finite presentation $A$-algebras. Note that this is an essentially
small category. We will endow it with the topology whose coverings
are finite families $\{p_i \colon U_i \to U\}_{i \in L}$  such
that $p_i$ are flat finite presentation maps such that $\cup_{i \in L}
\Img(p_i) = U$. We will denote this site by $\aff_{\ff}/X$ and by
$X_{\ff}$ its associated topos. Due to the fact that we limit the
possible objects to those flat of finite presentation, this site
should not be regarded as a big site and sometimes we will refer
to it informally as the \emph{small flat site}.

Let $\widetilde{A}$ or $\CO_X$ denote the ring valued sheaf defined by
$\widetilde{A}(\spec(B)) = B$. Note that $\widetilde{A}$ is a sheaf on $\aff_{\ff}/X$ as follows from faithfully flat descent \cite[VII, 2c]{sga42}, see also \cite[VIII, 1.6]{sga1} and \cite[IV, 6.3.1]{sga3}. We will denote the corresponding ringed topos as $(X_{\ff}, \widetilde{A})$ or $(X_{\ff}, \CO_X)$.
\end{cosa}

Let $\cart(X) := \cart(\aff/X,\widetilde{A})$. We recall the somewhat classical

\begin{prop}\label{cartaff}
With the previous notation, there is an equivalence of categories
\[
 \cart(X) \cong A\md.
\]
\end{prop}

\begin{proof}
Consider the functors $\Gamma(X,-) \colon \cart(X) \to A\md$ and ${(-)_{\car}} \colon  A\md \to \cart(X)$ defined by $\Gamma(X,\CM) :=\CM(X,\id_X)$ and $M_{\car}(\spec(R)) := R \otimes_A M$. It is clear that they are mutually quasi-inverse.
\end{proof}

Since each Cartesian presheaf in $\cart(X)$ is isomorphic to a presheaf of the form $M_{\car}$, it  is a sheaf \cite[VII, 2c]{sga42}. The category $\cart(X)$ is a full subcategory of the category of sheaves of $\widetilde{A}$-Modules closed under the formation of kernels, cokernels and coproducts, and it is an abelian category.

\begin{cosa}
Let $f \colon X \to Y$ be a morphism of affine schemes with $X=\spec(B$) and $Y=\spec(A)$. The category $\aff_{\ff} \to \aff/\spec{\ZZ}$ is a fibered category. Upon choice of a cleavage \cite[Exp. VI, \S 7]{sga1}, we may define a continuous functor
\[
{f}^{\aff} \colon \aff_{\ff}/Y \lto \aff_{\ff}/X
\]
given on objects by ${f}^{\aff}(V,v) := (V \times_Y X, p_2)$ with $p_2$ denoting the projection and on maps by the functoriality of pull-backs provided by the cleavage. It is clear that ${f}^{\aff}$ preserves fibered products, therefore by \ref{adjtop}, induces a pair of adjoint functors
\[
X_{\ff}\, \underset{f_*}{\overset{f^{-1}}{\longleftrightarrows}}
Y_{\ff},
\]
where we are taking the usual shortcuts $f^{-1} := ({f}^{\aff})^{-1}$ and $f_* := ({f}^{\aff})_* $.

Moreover, we have the morphism of sheaves of rings
\[
{f}^{\aff \,\#}\colon \widetilde{A} \lto f_*\widetilde{B}
\]
given by ${f}^{\aff \,\#}(V,v)\colon R\to R\otimes_A B$, ${f}^{\aff \,\#}(V,v)(r)=r\otimes 1$, for $V=\spec(R)$. Since the category $\mathbf{I}_{(U,u)}^{f^{\aff}}$ has finite products, by \ref{adjring}
we have the adjunction
\[
 \widetilde{B}\md\, \underset{f_*}{\overset{f^*}{\longleftrightarrows}} \widetilde{A}\md
\]
where $f^*:= ({f}^{\aff})^*$.
Notice that if $f$ is a flat finitely presented morphism, then $f^*\!\CF(U,u)=\CF(U,fu)$ for each $\CF \in \widetilde{A}\md$.
\end{cosa}

\begin{rem}
 We will abuse notation slightly and write $\CG(U)$ for $\CG(U,u)$ for any $(U,u) \in \aff/X$ and $\CG\in X_{\ff}$, 
if the morphism $u$ is understood.
\end{rem}

\begin{prop}\label{adjaf}
It holds that
\begin{enumerate}
 \item If $\CF \in \cart(Y)$, then $f^*\CF \in \cart(X)$.
 \item If $\CG \in \cart(X)$, then $f_*\CG \in \cart(Y)$.
\end{enumerate}

\end{prop}
\begin{proof}
(\emph{i}) Let $(U,u) \in \aff_{\ff}/X$ with $U = \spec(S)$. Denote the category $\mathbf{I}_{(U,u)}^{f^{\aff}}$ (see \ref{adjtop}) simply by $\mathbf{I}_{U}$ and let
\[
A_U := \dirlim{\,\mathbf{I}_{U}} \widetilde{A}(V).
\]
The sheaf $f^*\CF$ is the sheaf associated to the presheaf $P$ given by
\[
P(U,u) = S \otimes_{A_U}\!\!\dirlim{\mathbf{I}_{U}} \CF(V).
\]
We have for every $S$-module $N$ the following isomorphisms
\begin{align*}
\Hom_S(S\otimes_{A_U} \!\!\dirlim{\mathbf{I}_{U}} \CF(V), N)
&\simeq \Hom_{A_U}(\!\!\dirlim{\mathbf{I}_{U}} \CF(V), N)\\
&\simeq \!\!\invlim{\mathbf{I}_{U}} \Hom_{\widetilde{A}(V)}(\CF(V), N)\\
&\simeq \!\!\invlim{\mathbf{I}_{U}} \Hom_S(S\otimes_{\widetilde{A}(V)}\CF(V), N) \\
&
\simeq \Hom_S(\!\!\dirlim{\mathbf{I}_{U}}(S\otimes_{\widetilde{A}(V)}\CF(V)), N).
\end{align*}
By Yoneda's lemma, it follows that $f^*\CF$ is isomorphic to the sheaf associated to the presheaf $P'$ such that
\[
P'(U,u)=\dirlim{\mathbf{I}_{U}}S\otimes_{\widetilde{A}(V)}\CF(V).
\]
Let $h \colon ((V,v),g) \to ((V',v'),g')$ be a
morphism in $\mathbf{I}_{U}$, with $V=\spec(R)$,
$V'=\spec(R')$ and $h=\spec(\varphi)$. We have the commutative diagram

\begin{center}
\begin{tikzpicture}
\matrix(m)[matrix of math nodes, row sep=2.6em, column sep=3em,
text height=1.5ex, text depth=0.25ex]
{S\otimes_{R'}\CF(V') & S\otimes_R\CF(V)                \\
S\otimes_{R'}R'\otimes_A\CF(Y) & S\otimes_R R \otimes _A \CF(Y)\\
S\otimes_A\CF(Y) & S\otimes_A\CF(Y) \\};
\path[->,font=\scriptsize,>=angle 90]
(m-1-1) edge node[auto] {$\id \otimes\CF(h)$} (m-1-2)
(m-2-1) edge node[left] {$\id \otimes \wadj{\CF(v')}$} (m-1-1)
(m-2-2) edge node[right] {$\id \otimes \wadj{\CF(v)}$} (m-1-2)
(m-2-1) edge node[auto] {$\id \otimes \varphi \otimes \id$} (m-2-2)
(m-3-1) edge[-, double distance=2pt] (m-3-2)
(m-2-1) edge[-, double distance=2pt] (m-3-1)
(m-2-2) edge[-, double distance=2pt] (m-3-2);
\end{tikzpicture}
\end{center}
where the top horizontal map is the transition morphism and the vertical maps are isomorphisms by the definition of Cartesian presheaf and the properties of the tensor product. Thus,  $P'(U,u)=S\otimes_A\CF(Y)$ and
therefore  $f^*\CF$ is isomorphic to $(B\otimes_A \CF(Y))_{\car}$.

(\emph{ii}) Let $h \colon (V,v) \to
(V',v')$ be a morphism in $\aff_{\ff}/\,Y$, $V=\spec(R)$, $V'=\spec(R')$ and $h=\spec(\varphi)$. We have to prove that the morphism
\[
\wadj{f_*\CG(h)}
\colon R \otimes_{R'} f_*\CG(V') \lto f_*\CG(V)
\]
is an isomorphism (see \ref{superadj}).
We have that $f_*\CG(V)=\CG(V\times_Y X)$ and similarly for  $f_*\CG(V')$. Notice that $f_*\CG(h)=\CG(h\times_Y \id_X)$, thus the morphism
\[
\wadj{\CG(h\times_Y \id_X)} \colon
(R\otimes_A B) \otimes_{R' \otimes_A B}\CG(V' \times_Y X)
\lto \CG(V \times_Y X)
\]
is an isomorphism because $\CG$ is Cartesian. The result follows because $\wadj{f_*\CG(h)}$ is identified with the composition of the following isomorphisms
\[
R\otimes_{R'}\CG(V' \times_Y X)
\xto{\sim} (R \otimes_A B) \otimes _{R'\otimes_A B} \CG(V' \times_Y X)
\xto{\wadj{\CG(h\times \id)}}\CG(V\times_YX).
\]
In particular, $f_*(\CG)$ is isomorphic to $(_A\CG(X))_{\car}$, where  $_A\CG(X)$ means  $\CG(X)$ considered as an $A$-module.
\end{proof}

\begin{cor}
A morphism of affine schemes $f \colon X \to Y$ induces a pair of adjoint functors
\[
\cart(X) \underset{f_{*}}{\overset{f^*}{\longleftrightarrows}} \cart(Y).
\]
\end{cor}

\begin{rem}
We leave as an exercise for the interested reader to recognize that the previous adjunction agrees with the classical one
\[
\qco(X_{\ff}, \CO_X) \underset{f_{*}}{\overset{f^*}{\longleftrightarrows}} \qco(Y_{\ff}, \CO_Y),
\]
via the fact that in both cases the (pre)sheaves can be represented by its module of global sections. Notice that we consider the $\ff$ topology instead of the usual Zariski topology but this does not change essentially the category of quasi-coherent sheaves. We will see later (in Theorem \ref{agree}) the agreement of quasi-coherent sheaves with Cartesian presheaves in a much more general setting that includes all quasi-compact and semi-separated schemes.
\end{rem}

\section{Quasi-coherent sheaves on geometric stacks}\label{sec3}

\begin{cosa}\label{31}
We will follow the conventions of \cite{lmb}, specially for the definition of \emph{algebraic} stack. But for our purposes we will restrict to what we call \emph{geometric stacks} after Lurie (see \cite{lur}). So, from now on, a geometric stack $\SX$ will be a stack on the \'etale topology\footnote{\label{art}In fact, it can be shown that the $\ff$ topology yields the same category of stacks, this follows from \cite[Theorem 6.1]{ar74}, see also \cite[Chap. 10]{lmb}.} of (affine) schemes over a base scheme $S$ (that will not play any role in what follows, and will be omitted almost always in the notations) such that:
\begin{enumerate}
 \item the stack $\SX$ is semi-separated, i.e. the diagonal morphism $\delta_{\SX} \colon \SX \to \SX \times \SX$ is representable by \emph{affine} schemes;
 \item the stack $\SX$ is algebraic and quasi-compact, this amounts to the existence of a smooth and surjective morphism $p \colon X \to \SX$ from an \emph{affine} scheme $X$.
\end{enumerate}
We will refer to the map $p \colon X \to \SX$ as a \emph{presentation} of $\SX$. By \cite[Th\'eor\`eme (10.1)]{lmb}, it is enough to assume that there exists a presentation map $p$ which is faithfully flat of finite presentation. This fact allows us to use these presentations instead of just smooth ones.

A map representable by affine schemes is usually called an affine morphism \cite[3.10]{lmb}. Observe that every morphism $u \colon U \to \SX$ from an affine scheme $U$ is affine. Indeed, if $v \colon V \to \SX$ is another morphism with $V$ an affine scheme, we have that $U\times _{\SX }V \simeq \SX \times_{\SX \times \SX}(U \times V)$ is 1-isomorphic
to an affine scheme and this 1-isomorphism is an isomorphism because $U\times_{\SX}V$ is in fact a category fibered in sets. We will identify it with the spectrum of its global sections. We recall that the pair $(X, X \times_{\SX} X)$ with the obvious structure morphisms is a scheme in groupoids\footnote{Following Grothendieck's yoga of considering schemes as functors, we use the terminology \emph{scheme in groupoids} rather than \emph{groupoid scheme}.} whose associated stack is $\SX$, \ie
\[\SX = [(X, X \times_{\SX} X)]\]
 with the notation as in \cite[(3.4.3)]{lmb}.
\end{cosa}

\begin{cosa}
As examples of geometric stacks, let $G$ be a smooth affine algebraic group acting on a separated scheme $X$, then the quotient stack $[X/G]$ is geometric. In particular, the classifying stack for $G$, $\SB G$ is geometric. For further examples of geometric stacks, we send the reader to \cite{hr}.
\end{cosa}

\begin{lem}\label{relssep}
Let $f \colon \SX \to \SY$ be a morphism of geometric stacks. The relative diagonal $\delta_{\SX/\SY} \colon \SX \to \SX \times_{\SY} \SX$ is an affine morphism.
\end{lem}

\begin{proof}
Factor $\delta_{\SX}$ as
\[
\SX \xto{\delta_{\SX/\SY}} \SX \times_{\SY} \SX \xto{h} \SX \times \SX.
\]
Notice that $h$ is affine because it is the pull-back of $\delta_{\SY}$, and also its diagonal. Consider the following pull-back diagram
\begin{equation*}
\begin{tikzpicture}[baseline=(current  bounding  box.center)]
\matrix(m)[matrix of math nodes, row sep=2.6em, column sep=3.6em,
text height=1.5ex, text depth=0.25ex]{
  \SX                   &  \SX\: {_{\delta_{\SX}}\!\times_{h}} (\SX \times_{\SY} \SX) \\
  \SX \times_{\SY} \SX  & (\SX \times_{\SY} \SX) \times_{\SX^2} (\SX \times_{\SY} \SX) \\};
\path[->,font=\scriptsize,>=angle 90] 
(m-1-1) edge node[auto] {$\langle \id, \delta_{\SX/\SY}\rangle$}
(m-1-2) edge node[left] {$\delta_{\SX/\SY}$} (m-2-1)
(m-1-2) edge node[auto] {$\delta_{\SX/\SY} \times \id$} (m-2-2)
(m-2-1) edge node[below] {$\delta_{h}$} (m-2-2);
\draw [shorten >=0.3cm,shorten <=0.2cm,->,double] (m-2-1) -- (m-1-2); 
\end{tikzpicture}
\end{equation*}
with $\SX^2 :=  \SX \times \SX$. Notice that $\delta_{h}$ is affine, thus so is $\langle \id, \delta_{\SX/\SY}\rangle$. Now, $\delta_{\SX}$ is affine and therefore also the second projection $p_2 \colon \SX\: {_{\delta_{\SX}}\!\times_{h}} (\SX \times_{\SY} \SX) \to \SX \times_{\SY} \SX$. It follows that $\delta_{\SX/\SY} = p_2 \langle \id, \delta_{\SX/\SY}\rangle$ is affine as wanted.
\end{proof}

\begin{prop}\label{2cartsq}
Let
\begin{equation*}
\begin{tikzpicture}[baseline=(current  bounding  box.center)]
\matrix(m)[matrix of math nodes, row sep=2.6em, column sep=2.8em,
text height=1.5ex, text depth=0.25ex]{
  \SX' & \SY' \\
  \SX  & \SY \\};
\path[->,font=\scriptsize,>=angle 90] 
(m-1-1) edge node[auto] {$f'$}
(m-1-2) edge node[left] {$g'$} (m-2-1)
(m-1-2) edge node[auto] {$g$} (m-2-2)
(m-2-1) edge node[auto] {$f$} (m-2-2);
\draw [shorten >=0.2cm,shorten <=0.2cm,->,double] (m-2-1) -- (m-1-2) node[auto, midway,font=\scriptsize]{$\phi$};
\end{tikzpicture}
\end{equation*}
be a 2-Cartesian square of algebraic stacks. If $\SX$, $\SY$ and $\SY'$ are geometric stacks, then so is $\SX'$. 
\end{prop}

\begin{proof}
First, $\SX'$ is an algebraic stack by \cite[(4.5)(i)]{lmb}.

Let us check that the diagonal morphism $\delta_{\SX'} \colon \SX' \to \SX' \times \SX'$ is affine. Factor it as
 \[
 \SX' \xto{\,\delta_{\SX'/\SX}\,} \SX' \times_{\SX} \SX'  
 \overset{h}\lto \SX' \times \SX'
 \]
where $h$ is the morphism in the following 2-Cartesian square
\begin{equation*}
\begin{tikzpicture}[baseline=(current  bounding  box.center)]
\matrix(m)[matrix of math nodes, row sep=2.6em, column sep=2.8em,
text height=1.5ex, text depth=0.25ex]{
  \SX' \times_{\SX} \SX' & \SX' \times \SX' \\
  \SX                    & \SX \times \SX \\};
\path[->,font=\scriptsize,>=angle 90] 
(m-1-1) edge node[auto] {$h$}
(m-1-2) edge node[left] {$g''$} (m-2-1)
(m-1-2) edge node[auto] {$g' \times g'$} (m-2-2)
(m-2-1) edge node[auto] {$\delta_{\SX}$} (m-2-2);
\draw [shorten >=0.2cm,shorten <=0.2cm,->,double] (m-2-1) -- (m-1-2) node[auto, midway,font=\scriptsize]{$\phi'$};
\end{tikzpicture}
\end{equation*} 
The morphism $h$ is affine because it is obtained by base change from $\delta_{\SX}$. The morphism $\delta_{\SX'/\SX} \simeq \delta_{\SY'/\SY} \times_{\SX} \SX'$, so by Lemma \ref{relssep}, it is affine. It follows that $\delta_{\SX'}$ is affine as wanted.

Finally, given presentations $p \colon X \to \SX$ and $q \colon Y' \to \SY'$ we have an induced morphism
\[
p \times_{\SY} q \colon X \times_{\SY} Y' \lto \SX' 
\]
Notice that $X \times_{\SY} Y'$ is affine over $X$ because is the base change of the map $q$ that is affine, therefore $X \times_{\SY} Y'$ is an affine scheme. The morphism $p \times_{\SY} q$ is smooth and surjective because both conditions are stable by base change applying \cite[Lemma~(3.11)]{lmb} in view of \cite[\textbf{0}, (1.3.9)]{ega1}. This makes of $p \times_{\SY} q$ a presentation of $\SX'$ and the proof is complete.
\end{proof}

\begin{cosa}
For $\SX$ a geometric stack, we define the category $\aff/\SX$ as follows. Objects are pairs $(U,u)$ with $U$ an affine scheme and $u \colon U \to \SX$ a flat finitely presented 1-morphism of stacks. Morphisms $(f , \alpha)
  \colon (U,u) \to (V,v)$ are commutative diagrams of 1-morphisms of stacks. Let us spell this out. To give such an $(f,\alpha)$ amounts to give a diagram
\[
 \begin{tikzpicture}
      \draw[white] (0cm,2cm) -- +(0: \linewidth)
      node (G) [black, pos = 0.35] {$U$}
      node (H) [black, pos = 0.65] {$V$};
      \draw[white] (0cm,1.8cm) -- +(0: \linewidth)
      node (I) [black, pos = 0.6] {};
      \draw[white] (0cm,1.1cm) -- +(0: \linewidth)
      node (C) [black, pos = 0.45] {};
      \draw[white] (0cm,0.5cm) -- +(0: \linewidth)
      node (E) [black, pos = 0.5] {$\SX$};
      \draw [->] (G) -- (H) node[above, midway, sloped, scale=0.75]{$f$};
      \draw [->] (G) -- (E) node[auto, swap, midway, scale=0.75]{$u$};
      \draw [->] (H) -- (E) node[auto, midway, scale=0.75]{$v$};
      \draw [shorten >=0.2cm,shorten <=0.2cm,->,double] (C) -- (I) node[auto, midway, scale=0.75]{$\alpha$};  \end{tikzpicture}
\]
where $f \colon U \to V$ is a morphism of affine schemes making the diagram 2-commutative, through the 2-cell $\alpha \colon u \imp v f$.  The composition of morphisms is given by
\[
(g, \beta)\circ  (f, \alpha)=(g f, \beta f \circ \alpha).
\]

Notice that $\aff/\SX$ is an essentially small category. In fact, let $(U,u) \in \aff/\SX$ with $U = \spec(B)$, by the semi-separateness of $\SX$, the scheme $U \times_{\SX} X$ is affine. This affine scheme is isomorphic to the spectrum of a finitely presented algebra $B'$ over the ring of global sections of the structure sheaf of $X$. By faithful flatness $B$ is isomorphic to a subring of $B'$. This makes sense of what follows.
\end{cosa}

\begin{cosa}\label{site}
The category $\aff/\SX$ is ringed by the presheaf $\CO \colon \aff/\SX \to \ring$, defined by $\CO(U,u) = B$ when $U = \spec(B)$. We define Cartesian presheaves on a geometric stack $\SX$ as Cartesian presheaves over this ringed category,
\[
\cart(\SX) := \cart(\aff/\SX, \CO).
\]

Now, on $\aff/\SX$ we may define a topology by declaring as
covering the finite families $\{(f_i, \alpha_i) \colon (U_i,u_i)
\to (V,v)\}_{i \in I}$, with every $f_i$ a flat finitely presented
 map, that are jointly surjective. It becomes a site that
we denote $\aff_{\ff}/\SX$. The associated topos, \ie the category
of sheaves of sets on $\aff_{\ff}/\SX$ is denoted simply
$\SX_{\ff}$. The site $\aff_{\ff}/\SX$ is a ringed site by the
sheaf of rings associated to the presheaf $\CO$ that we keep
denoting the same.\footnote{The next lemma will show that there is no ambiguity.} We denote
\[
\qco(\SX) := \qco(\SX_{\ff}, \CO).
\]
\end{cosa}

\begin{rem}
\noindent
\begin{enumerate}
 \item Every smooth morphism is flat of finite presentation, therefore our site is finer than the usual \emph{lisse-\'etale} topology, see \cite[(12.1)]{lmb}.
 \item The topos $\SX_{\ff}$ has enough points. The argument is the same as for the lisse-\'etale topos, see \cite[(12.2.2)]{lmb}.
\end{enumerate}
\end{rem}

In what follows we will abuse notation slightly and write $\CF(U)$ for $\CF(U,u)$ and analogously for any object in $\aff/\SX$ if the morphism $u\colon U \to \SX$ is obvious.

\begin{lem} \label{eshaz}
Any $\CF \in \cart(\SX)$ is actually a \emph{sheaf} of $\CO$-Modu\-les.
\end{lem}

\begin{proof}
Let $(V,v) \in \aff/\SX$ with $V = \spec(C)$. Let $\{ (f_i, \alpha_i) \colon (U_i,u_i) \to (V,v)\}_{i \in I}$ be a covering with $U_i = \spec(B_i)$.  Let  $p_1^{ij}$ and $p_2^{ij}$ be the two canonical projections from $U_i \times_V U_j$ to $U_i$ and $U_j$, respectively. Consider $(U_i \times_V U_j,vf_ip_1^{ij}) \in \aff/\SX$ (notice that $vf_ip_1^{ij}= vf_jp_2^{ij}$). We have the diagram
\[
\begin{tikzpicture}
\matrix (m) [matrix of math nodes, row sep=3em, column sep=4em]{
\CF(V)
&   \prod_{i \in I}\CF(U_i)
&   \prod_{i, j \in I} \CF(U_i \times_V U_j)\\
}; 
\draw [transform canvas={yshift= 0.3ex},font=\scriptsize,->]
(m-1-2) -- node[above]{$\rho_2$}(m-1-3);
\draw [transform canvas={yshift=-0.3ex},font=\scriptsize,->]
(m-1-2) -- node[below]{$\rho_1$}(m-1-3);
\path[->,font=\scriptsize,>=angle 90]
    (m-1-1) edge node[auto] {$l$} (m-1-2);
\end{tikzpicture}
\]
where $l = (\CF(f_i,\alpha_i))_{i\in I}$, for $(x_i)\in \prod_{i\in I}\CF(U_i)$, we denote by $\rho_1(x_i)$ the element whose component in $\CF(U_i\times_V U_j)$ is $\CF(p_1^{ij},\alpha_i^{-1}p_1^{ij})(x_i)$ and by $\rho_2(x_i)$  the element whose component in  $\CF(U_i\times_VU_j)$ is $\CF(p_2^{ij},\alpha_j^{-1}p_2^{ij})(x_j)$. 
Let us check that this diagram is an equalizer. Being $\CF$ Cartesian, this diagram is isomorphic to the diagram
\[
\begin{tikzpicture}
\matrix (m) [matrix of math nodes, row sep=3em, column sep=4em]{
\CF(V)
&  \prod_{i \in I}B_i \otimes_C \CF(V)
& \prod_{i, j \in I}(B_i \otimes_C B_j)
 \otimes_C \CF(V)
\\
}; \draw [transform canvas={yshift= 0.3ex},font=\scriptsize,->]
(m-1-2) -- node[above]{$\rho'_2$}(m-1-3);
\draw [transform canvas={yshift=-0.3ex},font=\scriptsize,->]
(m-1-2) -- node[below]{$\rho'_1$}(m-1-3);
 \path[->,font=\scriptsize,>=angle 90]
    (m-1-1) edge node[auto] {$l'$}(m-1-2);
\end{tikzpicture}
\]
where $l'(x)=(1\otimes x)_{i \in I}$, for $x\in \CF(V)$, and for $(b_i\otimes y_i)_{i \in I} \in \prod B_i\otimes_C\CF(V)$, the elements $\rho'_1((b_i\otimes y_i)_{i \in I})$ and $\rho'_2((b_i\otimes y_i)_{i \in I})$ are those whose components in $(B_i\otimes_CB_j)\otimes_C\CF(V)$ are $b_i\otimes 1 \otimes y_i$, and  $1\otimes b_j \otimes y_j$, respectively.
 Let $B:= \prod_{i \in I} B_i$ and the morphism  $\varphi \colon C \to B$ induced by the covering $\{f_i \colon U_i \to V \}$. Since $I$ is finite, it suffices to prove that the diagram
 \[
\begin{tikzpicture}
\matrix (m) [matrix of math nodes, row sep=3em, column sep=4em]{
\CF(V) & B \otimes_C \CF(V) & (B\otimes_CB)\otimes_C \CF(V) \\
}; 
\draw [transform canvas={yshift= 0.3ex},font=\scriptsize,->]
(m-1-2) -- node[above]{$j_2\otimes \id$}(m-1-3);
\draw [transform canvas={yshift=-0.3ex},font=\scriptsize,->]
(m-1-2) -- node[below]{$j_1\otimes \id$}(m-1-3);
\path[->,font=\scriptsize,>=angle 90]
    (m-1-1) edge node[auto] {$l''$}(m-1-2);
\end{tikzpicture}
\]
is an equalizer, where $l''(x) = 1\otimes x$, $j_1(b) = b\otimes 1$ and $j_2(b)=1\otimes b$, for $x \in \CF(V)$ and $b\in B$.
As $B$ is faithfully flat over $C$, the diagram
\[
C \; {\overset{\varphi}{\longrightarrow}} \; B \;
\underset{j_1}{\overset{j_2}{\longrightrightarrows}}\;
\;B\otimes_CB
\]
is an equalizer and stays so after tensoring by $\CF(V)$ \cite[VIII 1.5]{sga1}. The result follows.
\end{proof}

\begin{rem}
\noindent
\begin{enumerate}
\item The proof works for any topology coarser than \textsf{fpqc}, we have chosen the $\ff$ topology because it gives a reasonable small site.
 \item The category $\cart(\SX)$ is a priori a full subcategory of $\pre\text{-}\CO_{\SX}\md$. The previous lemma implies that it is a full subcategory of $\CO_{\SX}\md$, as $\qco(\SX)$ is.
\end{enumerate}
\end{rem}

\begin{prop}\label{cartab}
The category $\cart(\SX)$ is a full subcategory of $\CO_{\SX}\md$ closed under the formation of kernel, cokernels and coproducts, and therefore it is an abelian category.
\end{prop}
\begin{proof}

Let $\CK$ be the kernel of a morphism $\CM \to
\CM'$ with $\CM,\CM' \in \cart(\SX)$ and $(f,
\alpha)\colon (U,u)\to (V, v)$ be a morphism in $\aff/\SX$ with
$U=\spec(B)$ and $V=\spec(C)$. If $f$ is flat, we have the
following commutative diagram
\begin{center}
\begin{tikzpicture}
\matrix(m)[matrix of math nodes, row sep=2.6em, column sep=2.8em,
text height=1.5ex, text depth=0.25ex]{
0 & B \otimes_C\CK(V) & B \otimes_C \CM(V) & B \otimes_C \CM^{\prime}(V)\\
0 & \CK(U)            & \CM(U)             & \CM^{\prime}(U)\\
};
\path[->,font=\scriptsize,>=angle 90] (m-1-1)
(m-1-1) edge node[auto] {} (m-1-2)
(m-1-2) edge node[auto] {} (m-1-3)
(m-1-3) edge node[auto] {} (m-1-4)
(m-1-2) edge node[left] {} (m-2-2)
(m-1-3) edge node[left] {$\wr$} (m-2-3)
(m-1-4) edge node[left] {$\wr$}(m-2-4)
(m-2-1) edge node[auto] {} (m-2-2)
(m-2-2) edge node[auto] {} (m-2-3)
(m-2-3) edge node[auto] {} (m-2-4);
\end{tikzpicture}
\end{center}
where the rows are exact and the midle and right vertical arrows are
isomorphisms because $\CM$ and $\CM'$ are Cartesian sheaves.
It follows that the remaining vertical arrow
\(
\wadj{\CK(f, \alpha)} \colon B\otimes_C \CK(V) \lto \CK(U)
\)
is an isomorphism. 

If $f$ is not flat, let $p\colon X \to \SX$ be an affine presentation with $X=\spec(A_0)$. If $p_1$ and $p'_1$, respectively, $p_2$ and $p'_2$ are the projections from $U\times_{\SX}X$ and $V\times_{\SX}X$ to $U$ and $V$, and $X$, respectively, and $\beta \colon u p_1 \imp p p_2$ and $\beta' \colon v p'_1 \imp p p'_2$  are the canonical 2-morphisms, we have a
commutative diagram of affine schemes
\begin{center}
\begin{tikzpicture}
\matrix(m)[matrix of math nodes, row sep=2.6em, column sep=2.8em,
text height=1.5ex, text depth=0.25ex]
{ & U\times_{\SX}X & U \\
X & V\times_{\SX}X & V \\}; \path[->,font=\scriptsize,>=angle 90]
(m-1-2) edge node[auto] {$p_1$} (m-1-3)
(m-1-2) edge node[left] {$p_2$} (m-2-1)
(m-1-2) edge node[left] {$f'$} (m-2-2)
(m-1-3) edge node[auto] {$f$ } (m-2-3)
(m-2-2) edge node[auto] {$p'_2$} (m-2-1)
(m-2-2) edge node[below] {$p'_1$} (m-2-3);
\end{tikzpicture}
\end{center}
where the projections $p_1$, $p_2$, $p'_1$ and $p'_2$ are flat maps and  $\beta' f' = \alpha p_1 \circ \beta$. Put $\spec
(B') = U \times_{\SX} X$ and $\spec (C') = V \times_{\SX} X$. Since the projection $p_1$ is faithfully flat, to prove that the morphism
$
\wadj{\CK(f, \alpha)} 
$
is an isomorphism it suffices to verify that the morphism
\[
\id_{B'}\otimes_B\wadj{\CK(f, \alpha)} \colon
B'\otimes_B (B\otimes_C\CK(V)) \lto B'\otimes_B\CK(U)
\]
is an isomorphism. But this morphism is the composition of the following chain of isomorphisms
\begin{align*}
B'\otimes_B(B\otimes_C \CK(V)) 
    & \simeq B'\otimes_{C'} C'\otimes_C\CK(V) \\
    & \simeq B'\otimes_{C'} \CK(V \times_{\SX}X, pp'_2)
              \tag{via $\CK(p'_1,(\beta')^{-1})$} \\
    & \simeq B'\otimes_{C'} (C' \otimes_{A_0}\CK(X))
              \tag{via $\CK(p'_2, \id)$} \\
    & \simeq B'\otimes_{A_0}\CK(X) \\
    & \simeq \CK(U\times_{\SX}X, p p_2)
              \tag{via $\CK(p_2, \id)$} \\
    & \simeq  B'\otimes_B\CK(U)
              \tag{via $\CK(p_1, \beta^{-1})$}
\end{align*}

The presheaf cokernel of $\CM \to \CM'$ is a Cartesian presheaf and then it is the cokernel in the category of sheaves of $\CO$-modules. Similarly, the coproduct of a set of Cartesian presheaves is a Cartesian presheaf and then it is the coproduct in the category of sheaves of $\CO$-modules.
\end{proof}

\begin{rem}
Some readers may find this proof too long. Notice that there might be non-flat maps between objects in the site. Moreover, there is no final object in the $\ff$ site unless $\SX$  is (isomorphic to) an affine scheme. This problem have nothing to do with stacks and one have to deal with them already on the case of non affine schemes.

\end{rem}

Let $\SX$ be a geometric stack and $p \colon X \to\SX$ a presentation with $X$ an affine scheme. Then $p$ induces a continuous morphism of sites
\[
{p}^{\aff} \colon \aff_{\ff}/\SX \lto \aff_{\ff}/X
\]
that is given on objects by ${p}^{\aff}(U,u) := (U \times_{\SX} X, p_2)$, with $p_2 \colon U \times_{\SX} X \to X$  the projection, and on maps by the functoriality of pull-backs.

As discussed in \ref{adjtop} this morphism induces a pair of adjoint functors
\[
X_{\ff} \underset{p_*}{\overset{p^{-1}}{\longleftrightarrows}}
\SX_{\ff}.
\]
As $p$ is a flat finite type presentation 1-morphism, $(p^{-1}\CF)(V,v) = \CF(V,pv)$ for $\CF \in \SX_{\ff}$. In particular, $p^{-1}\CO_{\SX} = \CO_{X}$, 
therefore $p^{-1}$ induces a functor between the categories of sheaves of $\CO$-Mod\-ules, \ie{} $p^{-1}$ agrees with the functor usually denoted $p^*$. So we have an adjunction as in \ref{adjring}
\begin{equation}\label{adjdesc}
\CO_{X}\md \underset{p_*}{\overset{p^*}{\longleftrightarrows}}
\CO_{\SX}\md,
\end{equation}
with a particularly simple description of the functors involved, due to the fact that $p$ induces a restriction functor between the corresponding sites.

\begin{prop}\label{desc1}
It holds that
\begin{enumerate}
 \item If $\CF \in \cart(\SX)$, then $p^*\CF \in \cart(X)$.
 \item If $\CG \in \cart(X)$, then $p_*\CG \in \cart(\SX)$.
\end{enumerate}
\end{prop}

\begin{proof}
We prove (\emph{i}). Let $h \colon (U,u) \to (V,v)$ be a morphism in $\aff_{\ff}/X$ with $U = \spec(B)$ and $V = \spec(C)$. We have to check that the morphism
\[
\wadj{(p^*\!\CF (h))} \colon B \otimes_C p^*\!\CF(V,v) \lto p^*\!\CF(U,u)
\]
is an isomorphism. But $p^*\!\CF(V,v) = \CF(V,pv)$ and analogously for $(U,u)$. Therefore the previous morphism becomes
\[
\wadj{\CF (h,\iid_{pu})}  \colon B \otimes_C \CF(V,pv) \lto \CF(U,pu)
\]
where $(h, \iid_{pu}) \colon (U,pu) \to (V,pv)$ is a morphism in $\aff_{\ff}/\SX$. But this map is an isomorphism because $\CF$ is Cartesian.

The proof of (\emph{ii}) is similar to the proof given in Proposition \ref{adjaf} (\emph{ii}).
\end{proof}

\begin{cor}
In this setting, the adjunction (\ref{adjdesc}) restricts to
\[
\cart(X) \,\underset{p_{*}}{\overset{p^*}{\longleftrightarrows}}\, \cart(\SX).
\]
\end{cor}
Let us now prove a converse to Proposition \ref{desc1} (\emph{i}).

\begin{lem}\label{descart}
Let $\CF$ be an $\CO_{\SX}$-module. If $p^*\!\CF\in \cart(X)$, then
$\CF\in \cart(\SX)$.
\end{lem}

\begin{proof}
Set $\SST := p_* p^*$. Since $\CF$ is a sheaf, we have the equalizer of $\CO_{\SX}$-modules
\[
\CF {\overset{\eta_{\CF}}{\longrightarrow}} \SST \CF
\, \underset{\SST\eta_{\CF}}{\overset{\eta_{\SST\CF}}{\longrightrightarrows}} \, \SST^2 \CF
\]
where $\eta$ is the unit of the adjunction $p^* \dashv p_*$. As a consequence of Proposition~\ref{desc1} we have $\SST \CF$, $\SST^2 \CF \in \cart(\SX)$ and it follows that $\CF \in \cart(\SX)$.
\end{proof}

\begin{thm}\label{agree}
 Let $\SX$ be a geometric stack. The subcategories $\cart(\SX)$ and $\qco(\SX)$ of $\CO_{\SX}\md$ \emph{agree}.
\end{thm}

\begin{proof}
Fix a presentation $p \colon X \to \SX$ as before and let $X = \spec(A_0)$.


Let $\CG \in \cart(\SX)$. First, notice that the Cartesian sheaf $p^*\!\CG$ sits in an exact sequence
\begin{equation*}
 \CO_X^{(I)} \lto \CO_X^{(J)} \lto p^*\!\CG \lto 0
\end{equation*}
that comes from a presentation of $M := \Gamma(X, p ^*\CG)$ as $A_0$-module  applying the functor $(-)_{\car}$ of Proposition \ref{cartaff}. Let $(U,u) \in \aff_{\ff}/\SX$, and put $\,V=U \times_{\SX}X$. Denote by $p_2 \colon V \to X$ the canonical projection. We obtain the exact sequence of sheaves of ${\CO}_V$-modules 
\[
 \CO_V^{(I)} \lto \CO_V^{(J)} \lto p_2^*p^*\!\CG \lto 0
\]
thus, $\CG \in \qco(\SX)$.


Let now $\CF \in \qco(\SX)$. There exists a covering $\{f_i \colon U_i \to X)\}_{i \in I}$ such that for the restriction of $p^*\CF$ to each $(\aff_{\ff}/U_i)$ there is a presentation
\[
 \CO_{U_i}^{(I)} \lto \CO_{U_i}^{(J)} \lto f_i^*p^*\!\CF \lto 0.
\]
Since  $\CO_{U_i}\in \cart(U_i)$, then $ f_i^*p^*\!\CF\in \cart(U_i)$.

Because $p^*\!\CF$ is a sheaf and $\{f_i \colon U_i \to X\}_{i \in I}$ is a covering in $\aff_{\ff}/X$ we obtain the equalizer of $\CO_X$-modules
\[
\begin{tikzpicture}
\matrix (m) [matrix of math nodes, row sep=3em, column sep=3em]{
p^*\!\CF & \prod_{i \in I}f_{i*}f_i^*p^*\!\CF & 
\prod_{i, j \in I}f_{ij*}f_{ij}^*p^*\!\CF,\\
}; 
\draw [transform canvas={yshift= 0.3ex},font=\scriptsize,->] 
(m-1-2) -- (m-1-3);
\draw [transform canvas={yshift= -0.3ex},font=\scriptsize,->]
(m-1-2) -- (m-1-3);
\path[->,font=\scriptsize,>=angle 90] (m-1-1) edge (m-1-2);
\end{tikzpicture}
\]
where $f_{ij}\colon U_i \times_X U_j \to  X$ is the composition of the first projection with $f_i$, or, what amounts to the same, the composition of the second projection with $f_j$. Since $f_{i*}f_i^*p^*\!\CF, f_{ij*}f_{ij}^*p^*\!\CF \in \cart(X)$, $p^*\!\CF \in \cart(X)$. The result follows from Lemma \ref{descart}.
\end{proof}

\begin{rem}
The consequence of the previous theorem is that not only all quasi-coherent sheaves are Cartesian (presheaves) but that this fact characterizes them on a geometric stack. From now on, we will use freely this identification when dealing with quasi-coherent sheaves and use the most convenient characterization for the issue at hand.
\end{rem}

\section{Descent of quasi-coherent sheaves on geometric stacks}\label{sec4}

We proceed in our goal of describing quasi-coherent sheaves through algebraic data on an affine groupoid scheme whose stackification is the initial geometric stack. As a first step, we will develop a descent result for quasi-coherent sheaves with respect to presentations of stacks.

\begin{cosa}\textbf{The category of descent data}.
Let $\SX$ be a geometric stack and let us fix $p\colon X \to\SX$ a presentation
with $X$ an affine scheme. Let $\CG$ be an $\CO_X$-Module. As before, for an object $(U,u) \in \aff_{\ff}/\SX$, we will abbreviate $\CG(U,u)$ by $\CG(U)$ if no confusion arises. Analogously, for a morphism $(h, \alpha)$ we will shorten $\CG(h, \alpha)$ by $\CG(h)$.

Consider the basic pull-back diagram
\begin{equation}\label{phi}
\begin{tikzpicture}[baseline=(current  bounding  box.center)]
\matrix(m)[matrix of math nodes, row sep=2.6em, column sep=2.8em,
text height=1.5ex, text depth=0.25ex]{
  X\times_{\SX}X & X \\
  X              & \SX \\};
\path[->,font=\scriptsize,>=angle 90] (m-1-1) edge
node[auto] {$p_2$}
(m-1-2) edge node[left] {$p_1$} (m-2-1)
(m-1-2) edge node[auto] {$p$} (m-2-2)
(m-2-1) edge node[auto] {$p$} (m-2-2);
\draw [shorten >=0.2cm,shorten <=0.2cm, ->, double] (m-2-1) -- (m-1-2) node[auto, midway,font=\scriptsize]{$\phi$};
\end{tikzpicture}
\end{equation}
A descent datum on $\CG$ is an isomorphism 
$t \colon p_2^*\CG \to p_1^*\CG$ of $\CO_{X\times_{\SX}X}$-Modules satisfying the cocycle condition on $\aff_{\ff}/X\times_{\SX}X\times_{\SX}X$
\[
p_{12}^*\,t\, \circ \, p_{23}^*\,t=p_{13}^*t,
\]
where as usual $p_{ij} \colon X\times_{\SX}X\times_{\SX}X \to X\times_{\SX}X$ ($i, j \in \{1, 2,  3\}$) denotes the several projections obtained omitting the factor not depicted. We define the category $\des(X/\SX)$ by taking as its objects the pairs $(\CG,t)$ with $\CG$ an $\CO_X$-Module and $t \colon p_2^*\!\CG \to p_1^*\!\CG$ a descent datum on $\CG$; and as its morphisms $h \colon (\CG,t) \to (\CG',t')$ where $h\colon\CG \to \CG'$ is a morphism of $\CO_X$-Modules such that the diagram
\begin{center}
\begin{tikzpicture}
\matrix(m)[matrix of math nodes,
row sep=2.6em, column sep=2.8em,
text height=1.5ex, text depth=0.25ex]
{p_2^*\!\CG & p_1^*\!\CG        \\
 p_2^*\CG'  & p_1^*\!\CG' \\};
\path[->,font=\scriptsize,>=angle 90]
(m-1-1) edge node[auto] {$t$} (m-1-2)
        edge node[left] {$p_2^*h$} (m-2-1)
(m-1-2) edge node[auto] {$p_1^*h$} (m-2-2)
(m-2-1) edge node[auto] {$t'$} (m-2-2);
\end{tikzpicture}
\end{center}
commutes.
\end{cosa}

\begin{cosa}\label{fund}\textbf{Descent datum associated to a sheaf of modules}.
The 2-morphism $\phi \colon p p_1 \imp p p_2$ induces an
isomorphism $\phi^* \colon p_2^* p^* \to p_1^* p^*$
given by 
\[
(\phi^*_{\CF}){(U,u)}=\CF(\id_U,\phi u),
\]
for $\CF \in \CO_{\SX}\md$ and $(U,u)\in \aff_{\ff}/X\times_{\SX}X$. Notice that for any $\CO_{\SX}$-Module $\CF$, it holds that $(p^*\!\CF,\phi^*_{\CF})\in \des(X/\SX)$.
Indeed, consider the pullback
\begin{equation}\label{phiprima}
\begin{tikzpicture}[baseline=(current  bounding  box.center)]
\matrix(m)[matrix of math nodes,
row sep=2.6em, column sep=2.8em,
text height=1.5ex, text depth=0.25ex]
{X \times_{\SX}X \times_{\SX}X & X   \\
 X \times_{\SX}X               & \SX \\};
 \path[->,font=\scriptsize,>=angle 90]
(m-1-1) edge node[auto] {$p_3$} (m-1-2)
        edge node[left] {$p_{12}$} (m-2-1)
(m-1-2) edge node[auto] {$p$} (m-2-2)
(m-2-1) edge node[auto] {$p\,p_2$} (m-2-2);
\draw [shorten >=0.2cm,shorten <=0.2cm,->,double] (m-2-1) -- (m-1-2) node[auto, midway,font=\scriptsize]{$\phi'$};
\end{tikzpicture}
\end{equation}
With the previous notations, we have that $\phi p_{13}=\phi' \circ
\phi p_{12}$ and  $\phi p_{23}= \phi'$ and thus, $\phi
p_{23} \,\circ \,\phi p_{12}=\phi p_{13}$. Now for $(U,u)\in
\aff_{\ff}/X\times_{\SX}X \times_{\SX}X$ we have
\begin{align*}
(p_{12}^*\phi^*_{\CF} \circ p_{23}^*\phi^*_{\CF}){(U,u)} &= \CF(\id_U,\phi p_{12}u)\,\CF(\id_U,\phi p_{23}u)\\
&=\CF(\id_U,\phi p_{13}u)\\
&=(p_{13}^*\phi^*_{\CF}){(U,u)}.
\end{align*}
This construction defines a functor
\[
\SD \colon \CO_{\SX}\md \lto \des(X/\SX)
\]
by $\SD(\CF) := (p^*\!\CF,\phi^*_{\CF})$.
\end{cosa}

\begin{cosa}\textbf{From descent data to sheaves of modules}.
We define a functor
\[
\SG \colon \des(X/\SX) \lto \CO_{\SX}\md
\]
assigning to $(\CG,t) \in \des(X/\SX)$ the equalizer of the pair of morphisms
\[
p_* \CG\,
\underset{b_{(\CG,t)}}{\overset{a_{\CG}}{\longrightrightarrows}} \,
p_* p_{2*} p_1^*\!\CG
\]
where
\begin{align*}
a_{\CG}     &:= ((\phi_*)_{p_1^*\!\CG}) \circ (p_* \eta_{1 \CG})\\
b_{(\CG,t)} &:= (p_*p_{2*}t) \circ (p_* \eta_{2 \CG})
\end{align*}
with $\phi_* \colon p_*p_{1*} \iso p_*p_{2*}$ denoting the functorial isomorphism and $\eta_i$ the unit of the adjunction
$p_i^* \dashv p_{i*}$ for $i \in \{1, 2\}$. We denote this kernel by $\SG(\CG,t)$. This construction is clearly functorial.
\end{cosa}

\begin{prop}\label{descadj}
The previously defined functors are adjoint, $\SD \dashv \SG$.
\end{prop}

\begin{proof}
Let $i_{\CG} \colon \SG(\CG,t) \to p_*\CG$ be the canonical morphism. For each    morphism of $\CO_{\SX}$-modules $g \colon \CF \to \SG(\CG,t)$, the adjoint to $i_{\CG} g \colon \CF \to p_*\CG$ through $p^* \dashv p_*$ provides the claimed morphism $(p^*\!\CF,\phi^*_{\CF}) \to (\CG,t)$ of descent data.

Reciprocally, if $f\colon \SD(\CF) \to (\CG,t)$ is a morphism in $\des(X/\SX)$, the morphism $\CF \to p_*\CG$ given by the adjunction $p^* \dashv p_*$ factors through $i_{\CG} \colon \SG(\CG,t) \to p_*\CG$ giving a morphism $\CF \to \SG(\CG,t)$, as desired.
\end{proof}

To establish the next main result, Theorem \ref{p52}, we will need a few technical preparations.
Let $(V,v) \in \aff_{\ff}/X$ and consider the following pull-backs

\begin{center}
\begin{tikzpicture}
\matrix(m)[matrix of math nodes,
row sep=2.6em, column sep=2.8em,
text height=1.5ex, text depth=0.25ex]{
V\times_{p_2}(X\times_{\SX}X) & X\times_{\SX}X \\
V                             & X\\};
\path[->,font=\scriptsize,>=angle 90] (m-1-1) edge
node[auto] {$v_2$} (m-1-2)
        edge node[auto] {$v_1$} (m-2-1)
(m-1-2) edge node[auto] {$p_2$} (m-2-2)
(m-2-1) edge node[auto] {$v$} (m-2-2);
\end{tikzpicture}   \hspace{6mm}      \begin{tikzpicture}
\matrix(m)[matrix of math nodes,
row sep=2.6em, column sep=2.8em,
text height=1.5ex, text depth=0.25ex]{
V \times_{\SX} X & X \\
V                & \SX \\};
\path[->,font=\scriptsize,>=angle 90]
(m-1-1) edge node[auto] {$w_2$} (m-1-2)
        edge node[auto] {$w_1$} (m-2-1)
(m-1-2) edge node[auto] {$p$} (m-2-2)
(m-2-1) edge node[auto] {$p v$} (m-2-2);
\draw [shorten >=0.2cm,shorten <=0.2cm,->,double] (m-2-1) -- (m-1-2) node[auto, midway,font=\scriptsize]{$\alpha$};
\end{tikzpicture}
\end{center}
Denote $W = V\times_{p_2}(X\times_{\SX}X)$. Set $h \colon
V\times_{\SX}X \to X\times_{\SX}X$ the morphism such that $p_1 h=w_2$, $p_2 h=v w_1$ and $\phi h= \alpha^{-1}$. The
morphism $h' \colon  V\times_{\SX}X \to W$ satisfying $v_1
h'=w_1$ and $v_2 h'=h$ is an isomorphism. Let
\[
\theta \colon p_{2*}p_{1}^* \lto p^*p_*
\]
be the isomorphism of
functors given by $\theta_{\CG}{(V,v)}=\CG(h')$ for each $\CG \in \CO_{X}\md$. Set 
\[
c_{(\CG,t)} := \theta_{\CG} \circ (p_{2*}t) \circ \eta_{2\CG}.
\]

\begin{lem}\label{igualita}
With the previous notations, the following identity holds
\[
p^*a_{\CG} \circ c_{(\CG,t)} = p^*b_{(\CG,t)} \circ c_{(\CG,t)}.
\]
\end{lem}

\begin{proof}
Consider the pull-backs
\begin{center}
\begin{tikzpicture}
\matrix(m)[matrix of math nodes,
row sep=2.6em, column sep=2.8em,
text height=1.5ex, text depth=0.25ex]{
W_1             & X\times_{\SX}X \\
V\times_{\SX}X  & X \\};
 \path[->,font=\scriptsize,>=angle 90]
(m-1-1) edge node[auto] {$\pi_2$} (m-1-2)
        edge node[left] {$\pi_1$} (m-2-1)
(m-1-2) edge node[auto] {$p_1$} (m-2-2)
(m-2-1) edge node[auto] {$w_2$} (m-2-2);
\end{tikzpicture}  \hspace{6mm}    \begin{tikzpicture}
\matrix(m)[matrix of math nodes,
row sep=2.6em, column sep=2.8em,
text height=1.5ex, text depth=0.25ex]
{W_2 & X\times_{\SX}X\\
V\times_{\SX}X  & X\\}; \path[->,font=\scriptsize,>=angle 90]
(m-1-1) edge node[auto] {$\pi'_2$} (m-1-2)
        edge node[left] {$\pi'_1$} (m-2-1)
(m-1-2) edge node[auto] {$p_2$} (m-2-2)
(m-2-1) edge node[auto] {$w_2$} (m-2-2);
\end{tikzpicture}
\end{center}
Let $l \colon W_2 \to V \times_{\SX}X$ be the morphism such that
$w_1 l = w_1 \pi'_1$, $w_2 l = p_1 \pi'_2$ and
$\alpha l= \phi^{-1}\pi'_2 \circ \alpha \pi'_1$. Put
$l' \colon W_2 \to W_1$ for the morphism satisfying that $\pi_1 l' = l$ and $\pi_2 l' = \pi'_2$. Notice that $l'$ is an isomorphism and
$p^*((\phi_{*})_{ p_{1}^*\!\CG}){(V,v)} = \CG(l')$. We have
\begin{align*}
(p^*((\phi_{*})_{p_{1}^*\CG} \circ p_* \eta_{1\CG}) \circ  c_{(\CG,t)})){(V,v)}
&= \CG(l') \CG(\pi_1) \CG(h') t{(W,v_2)} \CG(v_1) \\
&= \CG(h' l) t{(W,v_2)} \CG(v_1)\\
&=t{(W_2,h l)} \CG(h' l) \CG(v_1)\\
&= t{(W_2,h l)} \CG(w_1 l).
\end{align*}
On the other hand, we have
\begin{align*}
(p^*(p_* p_{2*} t \circ p_* \eta_{2\CG}) \circ c_{(\CG,t)}){(V,v)}
&= t{(W_2,\pi'_2)} \CG(\pi'_1) \CG(h') t{(W,v_2)} \CG(v_1) \\
&= t{(W_2,\pi'_2)} t{(W_2,h \pi'_1)} \CG(h' \pi'_1) \CG(v_1)\\
&= t{(W_2,\pi'_2)} t{(W_2,h \pi'_1)} \CG(w_1 l).
\end{align*}

Thus, it is sufficient to prove that $t{(W_2,\pi'_2)} \circ t{(W_2,h \pi'_1)} = t{(W_2,h l)}$. For that, consider the 2-pull-back

\begin{center}
\begin{tikzpicture}
\matrix(m)[matrix of math nodes,
row sep=2.8em, column sep=2.8em,
text height=1.5ex, text depth=0.25ex]{
(X\times_{\SX}X)_{p p_2}\times_{\SX}V & V\\
X\times_{\SX}X                          & \SX \\};
\path[->,font=\scriptsize,>=angle 90]
(m-1-1) edge node[auto] {$c_2$} (m-1-2)
        edge node[left] {$c_1$} (m-2-1)
(m-1-2) edge node[auto] {$pv$} (m-2-2) (m-2-1) edge node[auto]
{$p p_2$} (m-2-2);
\draw [shorten >=0.2cm,shorten <=0.2cm,->,double] (m-2-1) -- (m-1-2) node[auto, midway,font=\scriptsize]{$\gamma$};
\end{tikzpicture}
\end{center}
Let $k\colon W_2 \to (X\times_{\SX}X)_{p p_2}\times_{\SX}V$ be
the isomorphism given by $c_1 k = \pi'_2$, $c_2 k = w_1 \pi'_1$, and $\gamma k = \alpha^{-1}\pi'_1$. Consider the morphism
\[
\id\times_{\SX}v:(X\times_{\SX}X)_{p p_2}\times_{\SX}V \lto
X\times_{\SX} X \times_{\SX} X
\] satisfying
$p_{12} (\id\times_{\SX}v) = c_1$, $p_3 (\id\times_{\SX}v) = v c_2$,
and $\phi' (\id\times_{\SX}v) = \gamma$. The morphism
$g = (\id\times_{\SX}v) k$ is a finitely presented flat map.
Applying the cocycle condition of $t$ to $(W_2,g)$, we obtain that
\[
t{(W_2,p_{12} g)} \circ  t{(W_2, p_{23} g)} =
t{(W_2,p_{13} g)},
\]
and the result follows because $p_{12} g = \pi'_2$,
$p_{23} g = h \pi'_1$ and $p_{13} g = h l$.
\end{proof}

\begin{thm}\label{p52}
The functor $\SD \colon \CO_{\SX}\md \to \des(X/\SX)$  (and, therefore its adjoint $\SG$) is an equivalence of categories.
\end{thm}

\begin{proof}
We will see first that $\bar{\eta}$, the unit of the adjunction $\SD \dashv \SG$, is an isomorphism. Let $\CF \in \CO_{\SX}\md$ and consider the following commutative diagram 
\[
\begin{tikzpicture}
\matrix (m) [matrix of math nodes, row sep=3em, column sep=3em]{
\CF                    & p_*p^*\CF   & p_*p^*p_*p^*\CF \\
\SG(p^*\CF,\phi^*_\CF) & p_*p^*\CF   & p_*p_{2*}p_1^*p^*\CF \\
};
\draw [transform canvas={yshift= 0.3ex},->,font=\scriptsize] (m-1-2) -- node[above] {$\rho_1$}(m-1-3) ;
\draw [transform canvas={yshift=-0.3ex},->,font=\scriptsize] (m-1-2) -- node[below] {$\rho_2$}(m-1-3) ;
\draw [transform canvas={yshift= 0.3ex},->,font=\scriptsize] (m-2-2) -- node[above] {$a_{p^*\CF}$}(m-2-3) ;
\draw [transform canvas={yshift=-0.3ex},->,font=\scriptsize] (m-2-2) -- node[below] {$b_{(p^*{\CF},\phi^*_{\CF})}$}(m-2-3) ;
\draw [-, double distance=2pt] (m-1-2) -- (m-2-2);
\draw [->,font=\scriptsize]  (m-1-3) -- node[right] {via $\theta^{-1}$}(m-2-3);
\path[->,font=\scriptsize,>=angle 90]
    (m-1-1) edge node[left] {$\bar{\eta}_{\CF}$} (m-2-1)
    (m-1-1) edge node[auto] {$\eta_{\CF}$} (m-1-2)
    (m-2-1) edge node[auto] {$i_{p^*\CF}$} (m-2-2);
\end{tikzpicture}
\]
where $\eta_{\CF}$ is the unit of the adjunction $p^*\dashv p_*$ applied to $\CF$. The first row is an equalizer because $\CF$ is a sheaf and the claim follows.

Let us prove now that $\bar{\varepsilon}$, the counit of the adjunction $\SD \dashv \SG$, is an isomorphism. Let $(\CG,t)\in \des(X/\SX)$. We have that
$\bar{\varepsilon}_{(\CG,t)}=\varepsilon_{\CG}\circ p^*i_{\CG}$, where $\varepsilon_{\CG}$ denotes the counit of the adjunction $p^*\dashv p_*$. Since $p^*$ preserve finite limits, the horizontal row in the next diagram is exact
\[
\begin{tikzpicture}
\matrix (m) [matrix of math nodes, row sep=3em, column sep=3em]{
p^*\SG(\CG,t)  & p^* p_*\CG   & p^*p_*p_{2*}p_1^*\CG \\
\CG            \\
};
\draw [transform canvas={yshift= 0.3ex},font=\scriptsize,->] (m-1-2) --
node[above] {$p^*a_\CG$} (m-1-3);
\draw [transform canvas={yshift=-0.3ex},font=\scriptsize,->] (m-1-2) --
node[below] {$p^*b_{(\CG,t)}$} (m-1-3);
\path[->,font=\scriptsize,>=angle 90]
    (m-1-1) edge node[left] {$\bar{\varepsilon}_{(\CG,t)}$}(m-2-1)
    (m-1-1) edge node[auto] {$p^*i_{\CG}$}(m-1-2)
    (m-2-1) edge node[right] {\quad$c_{(\CG,t)}$}(m-1-2);
\end{tikzpicture}
\]
By Lemma \ref{igualita}
there exists a morphism $r \colon \CG \to p^*\SG(\CG,t)$
with $p^*i_{\CG} \circ r = c_{(\CG,t)}$. It holds that $c_{(\CG,t)} \;
\bar{\varepsilon}_{(\CG,t)} = p^*i_{\CG}$. Indeed,
\begin{align*}
c_{(\CG,t)} \circ \bar{\varepsilon}_{(\CG,t)}
&=\theta_{\CG}\circ p_{2 *}t \circ \eta_{2\CG}\circ \varepsilon_{\CG}\circ p^*i_{\CG} \\
&=\varepsilon_{p^*p_*\CG} \circ p^*p_*(\theta_{\CG} \circ p_{2*}t \circ \eta_{2\CG})\circ p^*i_{\CG} \\
&\cong\varepsilon_{p^*p_*\CG} \circ p^*p_*\theta_{\CG} \circ p^*((\phi_*)_{p_1^*\!\CG})\circ p^*p_*\eta_{1\CG}\circ p^*i_{\CG}
\end{align*}
and after an straightforward calculation we see that
\[
\varepsilon_{p^*p_*\CG} \circ p^*p_*\theta_{\CG} \circ
p^*((\phi_*)_{p_1^*\!\CG})\circ p^*p_*\eta_{1\CG} = \id_{p^*p_*\CG}.
\]

We see now that $r$ is the inverse of
$\bar{\varepsilon}_{(\CG,t)}$. Indeed,
\[
p^*i_{\CG} \circ r \circ {\bar{\varepsilon}}_{(\CG,t)} =
c_{(\CG,t)}\circ \bar{\epsilon}_{(\CG,t)}              = p^*i_{\CG}
\]
and $p^*i_{\CG}$ is a monomorphism of sheaves, therefore
$r \circ \bar{\varepsilon}_{(\CG,t)} = \id_{p^*\SG(\CG,t)}$. Since $\CG$ is a
sheaf, $\eta_{2\CG}$ is injective and thus $c_{(\CG,t)}$ is a monomorphism of sheaves.
From
\[
c_{(\CG,t)}\circ \bar{\varepsilon}_{(\CG,t)} \circ r = p^*i_{\CG} \circ r=c_{(\CG,t)}
\]
it follows that $\bar{\varepsilon}_{(\CG,t)} \circ r = \id_{\CG}$.
\end{proof}

\begin{rem}
The proof of the previous Theorem becomes simpler if we assume from the outset that we work on the category of quasi-coherent sheaves. For a nice treatment, cast in the setting of affine schemes, see \cite[Chap. 17]{w}. 
\end{rem}

\begin{cosa}\label{desqc}
 We denote by $\des_{\qc}(X/\SX)$ the full subcategory  of the
category $\des(X/\SX)$ whose objects are the pairs
$(\CG,t)$ with $\CG\in \qco(X)$.
\end{cosa}

\begin{cor}  \label{c53}
The categories $\qco(\SX)$ and $\des_{\qc}(X/\SX)$ are equivalent.
\end{cor}
\begin{proof}
The result follows because for $\CF \in \qco(\SX)$ we have $\SD(\CF) \in \des_{\qc}(X/\SX)$ and, conversely, for $(\CG,t) \in \des_{\qc}(X/\SX)$ it holds that $\SG(\CG,t) \in \qco(\SX)$ as follows from Propositions \ref{cartab}, \ref{desc1} and \ref{p52}, having in mind Theorem \ref{agree}.
\end{proof}

\begin{rem}
Our method of proof is not too distant from Olsson's \cite[\S 4]{O}. Notice, however, that he is working with Cartesian sheaves on a simplicial site while our result follows from the more general Theorem \ref{p52} that applies not only to quasi-coherent Modules.
\end{rem}

\section{Representing quasi-coherent sheaves by comodules}\label{sec5}

\begin{cosa}\label{c51}
\textbf{Hopf algebroids as affine models of geometric stacks}.
Let $\SX$ be a geometric stack and $p \colon X \to \SX$ a presentation
with $X$ affine. The fibered product $X \times_{\SX} X$ is an affine scheme because $\SX$ is geometric. As we recalled at the end of \ref{31}, the pair $(X, X \times_{\SX} X)$ is a scheme in groupoids whose associated stack is $\SX$.

Let $A_0$ and $A_1$ be the rings defined by $X = \spec(A_0)$ and $X_1 := X\times_{\SX} X = \spec(A_1)$. Denote the pair $A_\bullet := (A_0, A_1)$. The dual structure of a scheme in groupoids is called a \emph{Hopf algebroid}. Let us spell out the corresponding structure of $A_\bullet$, there are:
\begin{itemize}
 \item \emph{two} homomorphisms
       \[\eta_L, \eta_R \colon A_0 \longrightrightarrows A_1,\]
       corresponding respectively to the projections 
       $p_1, p_2 \colon X \times_{\SX}\ X \rightrightarrows X$;
 \item the \emph{counit} homomorphism
       \[\epsilon \colon A_1 \lto A_0,\]
       corresponding to the diagonal $\delta \colon X \to X \times_{\SX} X$;
 \item the \emph{conjugation} homomorphism
       \[ \kappa \colon A_1 \lto A_1,\]
        corresponding to the map interchanging the factors in 
        $X \times_{\SX} X$;
 \item and the \emph{comultiplication}
\[
\nabla \colon A_1 \lto A_1\, {_{\eta_R}\otimes_{\eta_L}} A_1,
\]
defined as: if $p_i \colon X \times_{\SX} X \to X$, $i\in \{1,2\}$ denotes the canonical projections and similarly for $p_{ij} \colon X \times_{\SX} X \times_{\SX} X \to X \times_{\SX} X$, $\nabla$ corresponds via the isomorphism  
\[
 X \times_{\SX} X \times_{\SX} X  {\overset{\sim}{\longrightarrow}} 
(X \times_{\SX} X)\: {_{p_1}\times_{p_2}} (X \times_{\SX} X)
\]
(induced by the projections $p_{23}$ and $p_{12}$) to the projection 
\[
p_{13} \colon X\times_{\SX}X \times_{\SX}X\to X\times_{\SX}X,
\]
that expresses the composition of internal arrows in the scheme in groupoids $(X, X \times_{\SX} X)$.
\end{itemize}

 These data is subject to the following identities:
\begin{enumerate}
 \item $\epsilon \eta_L = \id_{A_0} = \epsilon \eta_R$ (source and target of identity).
 \item If  $j_1, j_2 \colon A_1\to A_1\, {_{\eta_R} \otimes_{\eta_L}} A_1$ are the maps given by $j_1(b)=b \otimes 1$, $j_2(b)=1\otimes b$, then
\(
 \nabla \eta_L = j_1 \eta_L \text{ and }
 \nabla \eta_R = j_2 \eta_R
\) (source and target of a composition).
 \item $\kappa \eta_L = \eta_R$ and $\kappa \eta_R = \eta_L$ (source and target of inverse).
 \item $(\id_{A_1} \otimes \epsilon) \nabla = \id_{A_1} = (\epsilon \otimes \id_{A_1}) \nabla$ (identity).
 \item $(\id_{A_1} \otimes \nabla) \nabla = (\nabla \otimes \id_{A_1}) \nabla$ (associativity of composition).
\item If $\mu$ is the multiplication on $A_1$, then
$
\mu(\kappa \otimes \id_{A_1})\nabla=\eta_R \epsilon
$
 and \\
$
\mu(\id_{A_1} \otimes \kappa) \nabla=\eta_L \epsilon
$ (inverse).
 \item $\kappa \kappa = \id_{A_1}$ (inverting twice).
\end{enumerate}
Compare with \cite[A1.1.1.]{R} and \cite[(2.4.3)]{lmb}. For the general notion of Hopf algebroid, there is no need to impose any condition on the homomorphisms $\eta_L$, $\eta_R$. In this paper these structure maps will always be assumed to be flat because we will deal with \emph{geometric} stacks.


The geometric stack defined by the scheme in groupoids ---with \emph{flat} structure maps---
$(\spec(A_0),\spec(A_1))$  is denoted
\[\stack(A_\bullet) := [(\spec(A_0),\spec(A_1))]\]
 (with the right hand side notation as in \cite[(3.4.3)]{lmb}). So, with the notation at the beginning of this section, $\SX = \stack(A_\bullet)$.
\end{cosa}

From now on, every geometric stack we consider will be the stack associated to a Hopf algebroid. From a geometric point of view, the data of a Hopf algebroid is more rigid than the underlying algebraic stack alone. It is equivalent to the datum of an algebraic stack \emph{together with a smooth presentation} by an affine scheme. In any case, the algebroid determines the stack and we will use this fact in what follows. Note that $\kappa$ is an isomorphism (being inverse to itself) and that $\eta_L$ and $\eta_R$ are \emph{smooth} morphisms.

\begin{cosa}\label{exaction}
\textbf{Example:} \emph{An algebraic group acting over an affine scheme}.

Let $G = \spec(H)$ be an affine flat algebraic group over the base scheme $S$ that in this section we will assume affine, $S = \spec(R)$. We recall that $H$ is endowed with the structure of a Hopf algebra \cite[\S 1.4]{w} whose structure homomorphisms are
\begin{itemize}
 \item the \emph{comultiplication}
\[
\nabla \colon H \lto H \otimes_{R} H,
\]
       corresponding to the multiplication $m \colon G \times G \to G$;
 \item the \emph{counit} 
       \[\varepsilon \colon H \lto R,\]
       corresponding to the unit section $e \colon S \to G$;
 \item the \emph{antipode}
       \[ \sigma \colon H \lto H,\]
        corresponding to the map taking inverse $\iota \colon G \to G$;
\end{itemize}
Let $X = \spec(B)$ be an affine scheme. Suppose given an action of $G$ on $X$, \ie{}, a morphism
\[
a \colon G \times X \lto X
\]
satisfying the usual conditions \cite[\S 3.1]{w}. This action corresponds to a homomorphism $\alpha \colon B \to H \otimes_{R} B$ and makes part of a canonically defined Hopf algebroid. Take $A_0 := B$, $A_1 = H \otimes_{R} B$. The corresponding structure homomorphisms are:
\begin{itemize}
 \item The homomorphisms
       \[\eta_L = i_2,\quad \eta_R = \alpha \]
       where $i_2 \colon B \to H \otimes_{R} B$ 
       is the morphism defined by $i_2(b) = 1 \otimes b$;
 \item the \emph{counit} homomorphism is 
       \[\varepsilon \otimes \id \colon H \otimes_{R} B \lto B;\]
 \item the \emph{conjugation} homomorphism
       \[\kappa \colon H \otimes_{R} B \lto H \otimes_{R} B,\]
        defined by $\kappa(h \otimes b) = (\sigma(h)\otimes 1)\alpha(b)$;
 \item and the \emph{comultiplication}
\[
\bar{\nabla} \colon H \otimes_{R} B \lto (H \otimes_{R} B)\, {_{\alpha}\otimes_{i_2}} (H \otimes_{R} B),
\]
is $(i_1 \otimes i_1)\nabla$, where $i_1(h) = h \otimes 1$.
\end{itemize}
Denote this Hopf algebroid by $(B, H \otimes_{R} B)$. The morphism $i_2$ is flat because $G$ is flat over $S$. Thus, the associated scheme in groupoids has flat structure maps. Denote:
\[
[X/G] := \stack(B, H \otimes_{R} B).
\]
This stack may be considered the geometric quotient of the scheme $X$ by the action of $G$. Usually this is not equivalent to a scheme. The quotient map $p \colon X \to [X/G]$ is a presentation of this stack.
\end{cosa}

\begin{cosa}
\textbf{Comodules over a Hopf algebroid}.
Let $A_\bullet := (A_0, A_1)$ be a Hopf algebroid. A (left)
$A_\bullet$-comodule $(M, \psi_M)$ is an $A_0$-module $M$ together with an
$A_0$-linear map $\psi_M \colon M \to A_1\: {_{\eta_R} \otimes
_{A_0}}M$, the target being an $A_0$-module via $\eta_L$, and
satisfying the following identities:
\begin{enumerate}
\item $(\nabla \otimes \id_M) \psi_M = (\id_{A_1} \otimes \psi_M) \psi_M$ (coassociativity of $\psi$).
\item $(\epsilon \otimes \id_M) \psi_M = \id_M$ ($\psi_M$ is counitary).
\end{enumerate}
The map $\psi_M$ is called the comodule structure map of $M$. If $M$ and $M'$ are (left) $A_\bullet$-comodules, a map of $A_\bullet$-comodules is an $A_0$-linear map $\lambda \colon M \to M'$ such that the following square is commutative
\begin{center}
\begin{tikzpicture}
\matrix(m)[matrix of math nodes,
row sep=2.6em, column sep=2.8em,
text height=1.5ex, text depth=0.25ex] {
M  & A_1\: {_{\eta_R} \otimes _{A_0}}M \\
M' & A_1\: {_{\eta_R} \otimes _{A_0}}M' \\};
\path[->,font=\scriptsize,>=angle 90]
(m-1-1) edge node[auto] {$\psi_M$} (m-1-2)
        edge node[left] {$\lambda$} (m-2-1)
(m-1-2) edge node[auto] {$\id \otimes \lambda$} (m-2-2)
(m-2-1) edge node[auto] {$\psi_{M'}$} (m-2-2);
\end{tikzpicture}
\end{center}
We denote by $A_\bullet\com$ the category of (left) $A_\bullet$-comodules. All comodules considered in this paper will be on the left so we will speak simply of comodules.
\end{cosa}

If $M$ is an $A_0$-module one can always define a structure of $A_\bullet$-comodule on $ A_1\: {_{\eta_R} \otimes _{A_0}}M$ by defining $\psi_{ A_1\: {_{\eta_R} \otimes _{A_0}}M} := \nabla \otimes \id$. This structure is called the \emph{extended comodule} structure on $A_1\: {_{\eta_R} \otimes _{A_0}}M$.

\begin{thm} \label{grothencomod}
If $A_1$ is \emph{flat} as an $A_0$-module, then the category
$A_\bullet\com$ is a Grothendieck category.
\end{thm}

\begin{proof}
In \cite[A1.1.3.]{R} it is proved that $A_\bullet\com$ is an
abelian category if $A_1$ is a flat $A_0$-module. Therefore what is left to prove is that this category has exact direct limits and a set of generators.

Since tensor product commutes with colimits, the $A_0$-module colimit of a diagram in $A_\bullet\com$ is an $A_\bullet$-comodule. Direct limits of $A_\bullet$-comodules are exact, because direct limits are exact in the category  $A_0\md$. Since  $A_1$ is flat as an $A_0$-module, the kernel as an $A_0$-module of an $A_\bullet$-comodule morphism inherits the structure of $A_\bullet$-comodule.

Finally we show that the set $\SS$ of $A_\bullet$-subcomodules of $A_1^n$ for $n\in{\mathbb N}$ (viewed as $A_\bullet$-comodules via $\nabla$) is a set of generators of $A_\bullet\com$. Indeed, let $M$ be a left $A_\bullet$-comodule, $M\neq0$. The structure morphism $\psi = \psi_M \colon M \to A_1\: {_{\eta_R} \otimes _{A_0}}M$ is injective as a morphism of $A_\bullet$-comodules, considering the structure of extended comodule on $A_1\: {_{\eta_R} \otimes _{A_0}}M$. Let $p \colon A_0^{(I)} \to M$ be a free presentation of $M$ as an $A_0$-module.
Since the morphism $\id \otimes p \colon A_1\: {_{\eta_R} \otimes _{A_0}} A_0^{(I)} \to A_1\: {_{\eta_R} \otimes _{A_0}} M$ is a morphism of comodules and $A_1^{(I)} \cong A_1\: {_{\eta_R} \otimes _{A_0}} A_0^{(I)}$ as comodules, we have a surjective morphism of comodules $\bar{p} \colon A_1^{(I)} \to A_1\: {_{\eta_R} \otimes _{A_0}} M$.
Let $K$ be the pre-image of $M$ in $A_1^{(I)}$, it is a subcomodule because it is the kernel of a comodule map. Given $x \in M$, $x \neq 0$, there exists $I' = \{{i_1},\dots,{i_n}\} \subset I$ and $y \in A_1^{I'} \cap K$ such that $\bar{p}(y) = x$. Therefore $x \in \bar{p}(A_1^{I'} \cap K)$. This shows that $\Hom_{A_\bullet\com}(A_1^{I'} \cap K,M) \neq 0$. Notice that $A_1^{I'} \cap K \in \SS$. 
\end{proof}

\begin{rem}
 We remind the reader that if $A_\bullet$ is a Hopf algebroid arising from a geometric stack, $A_1$ is smooth over $A_0$ and, therefore, flat.
\end{rem}

\begin{cosa}\textbf{The category of descent data in $A_\bullet$}.
Let $A_\bullet$ be a Hopf algebroid as before, and $M$ an $A_0$-module. Write $A_2 = A_1\: {_{\eta_R}\!\otimes _{\eta_L}}\: A_1$. A
descent datum on $M$ is an isomorphism $ \tau:A_1\: {_{\eta_L}
\otimes _{A_0}}M \to  A_1\: {_{\eta_R} \otimes _{A_0}}M$ of
$A_1$-modules satisfying that the diagram
\begin{small}
\begin{center}
\begin{equation}\label{equdes}
\begin{tikzpicture}[baseline=(current  bounding  box.center)]
\matrix(m)[matrix of math nodes, row sep=2.6em, column sep=2em,
text height=1.5ex, text depth=0.25ex] {
A_2\: {_{j_1\eta_L}}\!\otimes _{A_0}\: M & A_2\: {_{j_1}\!\otimes _{A_1}}\: A_1 \:{_{\eta_L}\!\otimes _{A_0}}\: M & A_2\: {_{j_1}\!\otimes _{A_1}}\: A_1 \:{_{\eta_R}\!\otimes_{A_0}}\: M \\
A_2\: {_{\nabla}\!\otimes _{A_1}}\: A_1 \:{_{\eta_L}\!\otimes _{A_0}}\: M
& &  
\\
A_2\: {_{\nabla}\!\otimes _{A_1}}\: A_1 \:{_{\eta_R}\!\otimes _{A_0}}\: M
& &   A_2\: {_{j_2}\!\otimes _{A_1}}\: A_1 \:{_{\eta_L}\!\otimes
_{A_0}}\: M \\
A_2\: {_{\nabla \eta_R}}\!\otimes _{A_0}\: M && A_2\: {_{j_2}\!\otimes _{A_1}}\: A_1 \:{_{\eta_R}\!\otimes_{A_0}}\: M  \\} ;
\path[->,font=\scriptsize,>=angle 90]
(m-1-2) edge node[auto] {$\id\otimes\tau $} (m-1-3)
(m-1-3) edge node[right] {$\wr$} node[left] {$j_1\eta_R = j_2\eta_L\,$} (m-3-3)
(m-2-1) edge node[left] {$\id\otimes\tau$} (m-3-1)
(m-3-3) edge node[auto] {$\id\otimes \tau$} (m-4-3);
\path
(m-1-1) edge[-, double distance=2pt] (m-1-2)
(m-1-1) edge[-, double distance=2pt] (m-2-1)
(m-3-1) edge[-, double distance=2pt] (m-4-1)
(m-4-3) edge[-, double distance=2pt] (m-4-1);
\end{tikzpicture}
\end{equation}
\end{center}
\end{small}
commutes.

We define the category $\des(A_\bullet)$ by taking as its objects the pairs $(M,\tau)$ with $M$ an $A_0$-module and $\tau$ a descent datum on $M$, and as its morphisms $f \colon (M_1,\tau_1) \to (M_2,\tau_2)$ those homomorphism of $A_0$-modules $f \colon M_1 \to M_2$ such that the diagram 
\begin{center}
\begin{tikzpicture}
\matrix(m)[matrix of math nodes,
row sep=2.6em, column sep=2.8em,
text height=1.5ex, text depth=0.25ex]
{A_1 \:{_{\eta_L}\!\otimes _{A_0}}\: M_1  & A_1 \:{_{\eta_R}\!\otimes _{A_0}}\: M_1  \\
 A_1 \:{_{\eta_L}\!\otimes _{A_0}}\: M_2 & A_1 \:{_{\eta_R}\!\otimes _{A_0}}\: M_2 \\};
\path[->,font=\scriptsize,>=angle 90]
(m-1-1) edge node[auto] {$\tau_1$} (m-1-2)
        edge node[left] {$\id\otimes f$} (m-2-1)
(m-1-2) edge node[auto] {$\id\otimes f$} (m-2-2)
(m-2-1) edge node[auto] {$\tau_2$} (m-2-2);
\end{tikzpicture}
\end{center}
commutes.

For future reference, note that commutativity of diagram \eqref{equdes} is equivalent to the following identity in $A_1\: {_{\eta_R} \otimes _{\eta_L}}\:
 A_1\: {_{\eta_R} \otimes _{A_0}}\: M$
\begin{equation}\label{cuentas}
\sum_{i=1}^{n}a c_i \otimes b \tau(1\otimes x_i)=\sum_{i=1}^{n}(a \otimes b)\nabla(c_i)\otimes x_i
\end{equation}
where  $a, b \in A_1$ and $x\in M$, with $\tau(1\otimes x)=\sum_{i=1}^{n}c_i \otimes x_i$.
\end{cosa}

The next proposition corresponds to \cite[Theorem 2.2.]{hmh}. We give here a somewhat simplified proof. The informed reader will recognize it as an instance of \cite[Th\'eor\`eme, p.97]{BR}. This idea that relates descent with comonads on a certain adjunction\footnote{In the case of geometric stacks, the coalgebras over the adjunction comonad are the comodules over the underlying Hopf algebroid.} is what makes Pribble's \cite{P} approach work.

\begin{prop}\label{comod}
The categories $\des(A_\bullet)$ and $A_\bullet\com$  are isomorphic.
\end{prop}
\begin{proof}
Let $(M,\tau)\in \des(A_\bullet)$.  To give a structure of $A_\bullet$-comodule on $M$ one defines  $\psi = \psi_M \colon M \to A_1 \:{_{\eta_R}\!\otimes _{A_0}}\: M$ as the following composition
\[
M\xto{\eta_L \otimes \id} A_1 \:{_{\eta_L}\!\otimes _{A_0}}\: M \overset{\tau}{\lto} A_1 \:{_{\eta_R}\!\otimes _{A_0}}\: M.
\]
Let us see that $\psi$ is counitary. Let $x\in M$, with notation as in \ref{cuentas} we have
\[
(\epsilon \otimes \id)\,\psi(x)=\sum_{i=1}^n \epsilon(c_i)\,x_i.
\]
To see that $\sum_{i=1}^n \epsilon(c_i)\,x_i=x$, we apply the homomorphism $\epsilon \otimes \id \otimes \id$ to both sides of the identity (\ref{cuentas}) for $a = b = 1$ to get
\[
\sum_{i=1}^n \tau(\eta_L \, \epsilon(c_i)\otimes x_i)=\sum_{i=1}^{n}c_i \otimes x_i,
\]
and since $\tau$ is an isomorphism we deduce that
\[
\sum_{i=1}^n \eta_L \epsilon(c_i) \otimes x_i = 1\otimes x .
\]
Thus
\[
\sum_{i=1}^n \epsilon(c_i) x_i = 
(\epsilon \otimes \id) \sum_{i=1}^n \eta_L \epsilon(c_i) \otimes x_i =
(\epsilon \otimes \id) (1\otimes x) = x .
\]
The coassociativity of  $\psi$ follows from the identity (\ref{cuentas}) for $a = b = 1$.
Now, if $f \colon (M_1,\tau_1) \to (M_2,\tau_2)$ is a morphim in $\des(A_\bullet)$, then
\[
(\id \otimes f)  \psi_1 =(\id\otimes f) \tau_1 (\eta_L\otimes \id) =
\tau_2 (\id\otimes f) (\eta_L \otimes \id) =
\tau_2 (\eta_L\otimes \id) f =
\psi_2 f
\]
so $f \colon (M_1, \psi_1) \to (M_2, \psi_2)$ is a morphism of $A_\bullet$-comodules. It follows that this defines a functor
$\SC \colon \des(A_\bullet) \to A_\bullet\com$.

Conversely, let $(M, \psi)$ be an $A_\bullet$-comodule. We define 
\[
\tau: A_1\: {_{\eta_L} \otimes _{A_0}}M \to  A_1\: {_{\eta_R} \otimes _{A_0}}M\]
as the composition
\[
A_1\:{_{\eta_L}
\otimes _{A_0}}M \xto{\id\otimes \psi} A_1\: {_{\eta_L} \otimes _{\eta_L}}\:
 A_1\: {_{\eta_R} \otimes _{A_0}}\: M \xto{\mu \otimes \id}
A_1\:{_{\eta_R} \otimes _{A_0}}M.
\]
It easy to prove that $\tau$ is a homomorphism of $A_1$-modules. We show that $\tau$ is an isomorphism by constructing its inverse. We define $\tau'$ as the following composition
\begin{small}
\[
A_1\:{_{\eta_R} \otimes _{A_0}}M 
\xto{\id\otimes \psi} 
A_1\: {_{\eta_R} \otimes _{\eta_L}}\: A_1\: {_{\eta_R} \otimes _{A_0}}\: M  \xto{\id\otimes \kappa \otimes \id} 
A_1\: {_{\eta_R} \otimes _{\eta_R}}\: A_1\: {_{\eta_L} \otimes _{A_0}}\: M  \xto{\mu \otimes \id}
A_1\:{_{\eta_L} \otimes _{A_0}}M.
\]
\end{small}%
Let us see that $\tau$ and $\tau'$ are mutually inverse. Consider the
commutative diagram in which all tensor products are taken over $A_0$ (with most subindices omitted)
\begin{small}
\begin{center}
\begin{tikzpicture}
\matrix(m)[matrix of math nodes, row sep=2.6em, column sep=2.8em,
text height=1.5ex, text depth=0.25ex] {
A_1\:{_{\eta_R}\otimes} M & A_1 \otimes A_1 \otimes M &
     A_1 \otimes A_1 \otimes M  & A_1\:{_{\eta_L}\otimes} M \\
A_1 \otimes A_1 \otimes M & A_1 \otimes A_1 \otimes A_1 \otimes M &
     A_1 \otimes A_1 \otimes A_1 \otimes M & A_1 \otimes A_1 \otimes M \\
&& A_1\otimes A_1 \otimes M  & A_1\:{_{\eta_R}\otimes} M \\};
\path[->,font=\tiny,>=angle 90]
(m-1-1) edge node[auto] {$\id\otimes \psi$} (m-1-2)
(m-1-2) edge node[auto] {$\id\otimes \kappa \otimes \id$} (m-1-3)
(m-1-3) edge node[auto] {$\mu \otimes \id$} (m-1-4)
(m-2-1) edge node[auto] {$\id \otimes \nabla \otimes \id$} (m-2-2)
(m-2-2) edge node[auto] {$\id \otimes \kappa \otimes \id$}(m-2-3)
(m-2-3) edge node[auto] {$\mu \otimes \id\otimes \id$} (m-2-4)
(m-3-3) edge node[auto] {$\mu \otimes \id$} (m-3-4)
(m-1-1) edge node[auto] {$\id \otimes \psi$} (m-2-1)
(m-1-2) edge node[auto] {$\id\otimes \id \otimes \psi$} (m-2-2)
(m-1-3) edge node[auto] {$\id\otimes \id \otimes \psi$} (m-2-3)
(m-1-4) edge node[auto] {$\id\otimes \psi$} (m-2-4)
(m-2-3) edge node[auto] {$\id \otimes \mu \otimes \id$}(m-3-3)
(m-2-4) edge node[auto] {$\mu \otimes \id$}(m-3-4);
\end{tikzpicture}
\end{center}
\end{small}
then we have
\begin{align*}
\tau \tau' 
&= (\mu \otimes \id) (\id \otimes \mu \otimes \id) (\id \otimes \kappa \otimes \id) (\id \otimes \nabla \otimes \id) (\id \otimes \psi) \\
&= (\mu \otimes \id) (\id\otimes \mu(\kappa\otimes \id) \nabla\otimes \id) (\id\otimes \psi) \\
&=(\mu \otimes\id) (\id\otimes \eta_R \epsilon \otimes \id) (\id\otimes \psi)
= \id . 
\end{align*}
A similar computation shows that $\tau' \tau=\id$.

Finally, it remains to prove that $\tau$ verifies ({\ref{cuentas}}). Since $\psi$ is coassociative, we have for each $x \in M$
\begin{equation}\label{cuentascoa}
\sum_{i=1}^{n}c_i \otimes \psi( x_i)=\sum_{i=1}^{n}\nabla(c_i)\otimes x_i
\end{equation}
where $\psi(x)=\sum_{i=1}^{n}c_i \otimes x_i$. But $\tau(1\otimes x)=\psi(x)$, for each $x\in M$, and the result follows multiplying by $a \otimes b \in A_2$ in both sides of (\ref{cuentascoa}).

Now, if $f \colon (M_1, \psi_1) \to (M_2, \psi_2)$ is a morphism of $A_{\bullet}$-comodules, and $\tau_1$ and $\tau_2$ are their corresponding descent data, then
\[
(\id\otimes f) \tau_1 = 
(\mu \otimes \id) (\id \otimes \id \otimes f) (\id \otimes \psi_1) =
\tau_2 (\id\otimes f)
\]
and so $f:(M_1, \tau_1)\to (M_2, \tau_2)$ is a morphism in $\des(A_{\bullet})$. It follows that this defines a functor $\SD' \colon A_\bullet\com \to \des(A_\bullet)$.

Note that the functors $\SC$ and $\SD'$ are inverse to each other. In fact, let $(M,\tau)\in \des(A_\bullet)$ and $\psi \colon M \to A_{1\,\eta_R} \otimes _{A_0}M\,$ be its associated $A_\bullet$-comodule structure.  The descent data associated to the comodule $(M, \psi)$ is the composition $(\mu \otimes \id) (\id\otimes \psi) = \tau$.
Conversely, let $(M, \psi)$ be an $A_\bullet$-comodule. If $\tau$ is the
corresponding descent data on $M$, the structure map of the $A_\bullet$-comodule associated to $(M,\tau)$ is the morphism $\tau (\eta_L \otimes 1) = \psi$. 
\end{proof}

\begin{cosa}
Let $X = \spec(A)$ be an affine scheme. Recall that by Theorem \ref{agree} the categories $\qco(X)$ and $\cart(X)$ are the same. Correspondingly, the functors ${\widetilde{(-)}}$ and $(-)_{\car}$ are identified, too. We will stick to the former notation according to the usage of \cite[(1.3)]{ega1}.
\end{cosa}

\begin{thm} \label{crt11}
Let $\SX$ be a geometric stack and $p \colon X \to \SX$ a presentation. Take $A_\bullet$ its corresponding Hopf algebroid. The categories $\des_{\qc}(X/\SX)$ and $A_\bullet\com$ are equivalent.
\end{thm}

\begin{proof}
In this setting $X = \spec(A_0)$ and $\SX = \stack(A_\bullet)$. Notice that the equivalence of categories given by the couple of functors 
\[
\Gamma(X,-) \colon \qco(X) \to A_0\md \enskip\text{ and }\enskip
{\widetilde{(-)}} \colon  A_0 \md \to \qco(X)
\] 
induces an equivalence between the categories $\des_{\qc}(X/\SX)$ and  $\des(A_\bullet)$, and this last category is equivalent to $A_\bullet\com$ by Proposition \ref{comod}.
\end{proof}

\begin{rem}
This result is similar to \cite[Theorem 2.2]{hmh}. Notice however that Hovey's quasi-coherent sheaves over $\SX$ are sheaves $\CF$ for the (big) \emph{discrete} topology while we work in the $\ff$ topology. The ultimate reason of the agreement of both approaches lies in Lemma \ref{eshaz} that expresses the property of descent of Cartesian presheaves for this topology which makes them automatically sheaves. Notice also that the present proof is  simpler once we realize that all the data involved refer to modules on an affine scheme and, therefore, it has an essentially algebraic nature.
\end{rem}

\begin{cor} \label{descrt}
With $\SX$ and $A_\bullet$ as in the previous theorem, the categories $\qco(\SX)$ and $A_\bullet\com$ are equivalent.
\end{cor}

\begin{proof}
Combine Corollary \ref{c53} and Theorem \ref{crt11}.
\end{proof}

\begin{cor} \label{Grothqco}
Let $\SX$ be a geometric stack, $\qco(\SX)$ is a Grothendieck category.
\end{cor}

\begin{proof}
Combine Theorem \ref{grothencomod} with Corollary \ref{descrt}.
\end{proof}

\begin{cosa}\textbf{Example:} \emph{Equivariant quasi-coherent sheaves}.

Go back to the setting of \ref{exaction}. A quasi-coherent sheaf on $[X/G]$ may be described as a descent datum associated to the presentation $p \colon X \to [X/G]$. Let us explain this.

Let $M$ be a $B$-module. According to \cite[Ch 1. \S 3. Definition 1.6.]{git} the sheaf $\CF :=\widetilde{M}$ on $X$ is equivariant\footnote{Mumford's original terminology in \cite{git} is that ``$s$ is a $G$-linearization of $\CF$''.} by the action of $G$ if there is given an isomorphism of sheaves on $G \times X$
\[
s \colon a^*\CF \liso p_2^*\CF
\]
such that the following cocycle condition of $t$ on $\CF$ holds
\[
p_{2 3}^* s \circ (\id_G \times a)^* s = (m \times \id_X)^*s
\]
where the maps are given by the following cartesian square
\begin{center}
\begin{tikzpicture}
\matrix(m)[matrix of math nodes, row sep= 2.6em, column sep=2.8em,
text height=1.5ex, text depth=0.25ex]
{G \times G \times X        & G \times X \\
 G \times X                 & X   \\};
\path[->,font=\scriptsize,>=angle 90]
(m-1-1) edge node[auto] {$p_{2 3}$} 
(m-1-2) edge node[left] {$\id_G \times a$} (m-2-1)
(m-1-2) edge node[auto] {$a$} (m-2-2) 
(m-2-1) edge node[auto] {$p_2$} (m-2-2);
\end{tikzpicture}
\end{center}
Notice that the pair $(\CF, s^{-1})$ is a descent datum on $\CF$ with respect to the presentation $p \colon X \to [X/G]$, therefore, by Corollary \ref{c53} it corresponds to a quasi-coherent sheaf on the quotient stack $[X/G]$.

Taking global sections these data are equivalent to give an isomorphism of $H \otimes_{R} B$-modules
\[
\tau \colon (H \otimes_{R} B)\: {_{i_2} \otimes_{B}} M \liso 
(H \otimes_{R} B)\: {_{\alpha} \otimes_{B}} M
\]
satisfying the commutativity of diagram (\ref{equdes}), with $A_1 = H \otimes_{R} B$ and $A_0 = B$. We see that $(M, \tau)$ and $(\CF, s^{-1})$ correspond to each other by the equivalence between the categories $\des_{\qc}(X/[X/G])$ and  $\des(B, H \otimes_{R} B)$.
\end{cosa}

We denote by $\GGamma_p^{\SX} \colon \qco(\SX) \to  A_\bullet\com$ the equivalence of categories defined as $\GGamma_p^{\SX}:= \SC \circ \Gamma(X,-)\circ \SD$, with $\SD$ as in \ref{fund} and $\SC$ as in Proposition \ref{comod}. A quasi-inverse is $\Theta_p^{\SX} = \SG \circ \widetilde{(-)} \circ \SD'$.

\begin{prop}\label{modcomod}
We have
\begin{enumerate}
\item If $(M, \psi)$ is an $A_\bullet$-comodule and $\CF$ is its associated $\CO_{\SX}$-module, then the $\widetilde{A_0}$-module $p^*\!\CF$ is $\widetilde{M}$.
\item If $\CG\in \qco(X)$ then the comodule $\GGamma_p^{\SX}(p_*\!\CG)$ is isomorphic to the extended comodule $A_1\: {_{\eta_R} \otimes_{A_0}\CG(X)}$.
\end{enumerate}
\end{prop}

\begin{proof}
By Theorem~\ref{crt11} $(M, \psi)$ corresponds to $(\widetilde{M},t) \in
\des_{\qc}(X/\SX)$. By Corollary \ref{descrt}
$p^*\!\CF = p^*\!\SG(\widetilde{M},t) \cong \widetilde{M}$. This gives (\emph{i}).

For (\emph{ii}), we have by {\ref{cartaff}} that $\widetilde{\CG(X)}$ and $\CG$ are isomorphic $\widetilde{A_0}$-modules and then
 the comodules $\GGamma_p^{\SX}(p_*{\widetilde{\CG(X})})= A_1\: {_{\eta_R} \otimes _{A_0}} \CG(X)$ and $\GGamma_p^{\SX}(p_*\CG)=\CG(X_1,p_2)$ are isomorphic via $\wadj{\CG(p_2)}\colon A_1\: {_{\eta_R} \otimes_{A_0}} \CG(X) \to \CG(X_1,p_2)$, notation as in \ref{superadj}. To prove (\emph{ii}) is sufficient to show that the comodule structure on  $A_1\: {_{\eta_R} \otimes _{A_0}} \CG(X)$ given by \ref{fund} and Theorem \ref{crt11}, is $\nabla \otimes \id$.
Set $M=\CG(X)$. The descent datum $\tau$ on  $A_1\: {_{\eta_R} \otimes _{A_0}} M$ is defined by
\[
\tau =(\wadj{p_*{\widetilde{M}}(p_2,\iid)})^{-1} p_*{\widetilde{M}}(\id,\phi^{-1})\, \wadj{(p_*{\widetilde{M}}(p_1,\iid))}.
\]
 The structure map $\psi$ is the composition
\[
 A_1\: {_{\eta_R} \otimes _{A_0}}M \xto{\eta_L\otimes \id}
 A_1\: {_{\eta_L} \otimes _{\eta_L}} A_1\: {_{\eta_R} \otimes _{A_0}} M\overset{\tau}\lto 
 A_1\: {_{\eta_R} \otimes _{\eta_L}} A_1\: {_{\eta_R} \otimes _{A_0}} M.
\]
Let $X_2 = \spec(A_1\: {_{\eta_R} \otimes _{\eta_L}} A_1)$. Consider the isomorphism 
\begin{equation}\label{uvedoble}
 w \colon X_2 \liso X\times_{\SX}X\times_{\SX}X
\end{equation}
given by $p_{12} w=q_2$, $p_3 w=p_2 q_1$ and $\phi' w=\phi  q_1$ with $\phi'$ as in \ref{phiprima} where $q_i \colon X_2 \to X\times_{\SX}X$ ($i\in\{1,2\}$) are the canonical projections. We obtain
\begin{align*}
\psi&=(\wadj{p_*{\widetilde{M}}(p_2,\iid)})^{-1} p_*{\widetilde{M}}(\id,\phi^{-1}) {p_*{\widetilde{M}}(p_1,\iid)}\\
&= (\wadj{p_*{\widetilde{M}}(p_2,\iid)})^{-1} p_*{\widetilde M}(p_1,\phi^{-1})\\
&=(\wadj{p_*{\widetilde{M}}(p_2,\iid)})^{-1}\widetilde{M}(w)^{-1}\widetilde{M}(w)\widetilde{M}(p_{13})\\
&= \widetilde{M}(m) = \nabla \otimes\id.
\end{align*}
The last equalities hold because $\widetilde{M}(w)\wadj{p_*{\widetilde{M}}(p_2,\iid)} = \id$ and $m = \spec(\nabla)$.
\end{proof}

\section{Properties and functoriality of quasi-coherent sheaves}\label{sec6}

We have reached our first main objective: representing quasi-coherent sheaves on a geometric stack. It is time to move forward to the next one, namely, describing a nice functoriality formalism for these kind of sheaves for an \emph{arbitrary} 1-morphism of stacks. Let then $f \colon \SX \to \SY$ be a 1-morphism of geometric stacks. 

Unfortunately, $f$ does not induce a (continuous) functor from $\aff_{\ff}/\SY$ to $\aff_{\ff}/\SX$, so we can not apply directly the methods from \ref{adjtop} and \ref{adjring}. Of course, it could be done if $f$ is an affine 1-morphism, but this condition is clearly too restrictive.

Besides, there is no hope to define a topos morphism from  $\SX_{\ff}$ to $\SY_{\ff}$ because the topology is too fine, however we will construct an adjunction that will induce the corresponding pair of adjoint functors between the categories of quasi-coherent sheaves. The problem arises already in the \emph{lisse-\'etale} topology, see the remark after Corollary \ref{adjhaz6}. However, the construction will follow the usual procedure (as in \emph{e.g.} \cite{lmb}). We will treat carefully the existence of an adjunction between the category of sheaves and how it extends to the category of \emph{quasi-coherent} sheaves.

\begin{cosa}\label{di}\textbf{Direct image}.
For fixed $(V,v)\in \aff_{\ff}/\SY$ we
denote by $\mathbf{J}(v,f)$ the category whose objects are the
$2$-commutative squares
\begin{center}
\begin{equation}\label{abiertos}
\begin{tikzpicture}[baseline=(current  bounding  box.center)]
\matrix(m)[matrix of math nodes, row sep=2.6em, column sep=2.8em,
text height=1.5ex, text depth=0.25ex]
{U  &   V \\
\SX & \SY \\}; 
\path[->,font=\scriptsize,>=angle 90]
(m-1-1) edge node[auto] {$h$} (m-1-2)
        edge node[left] {$u$} (m-2-1)
(m-1-2) edge node[auto] {$v$} (m-2-2) 
(m-2-1) edge node[auto] {$f$} (m-2-2);
\draw [shorten >=0.2cm,shorten <=0.2cm,->,double] (m-2-1) -- (m-1-2) node[auto, midway,font=\scriptsize]{$\gamma$};
\end{tikzpicture}
\end{equation}
\end{center}
where $(U,u) \in \aff_{\ff}/\SX$. We denote this object by
$(U,u,h,\gamma)$. A morphism from $(U,u,h,\gamma)$ to
$(U',u',h',\gamma')$ is just a morphism
$(l,\alpha)\colon (U,u)\to (U',u')$ in
$\aff_{\ff}/\SX$ satisfying $h'l = h$ and $\gamma =
\gamma' l \circ f \alpha$. If $\CG \in \pre(\aff_{\ff}/\SX)$
we define the presheaf $f_*\CG$ by:
\[
(f_*\CG)(V,v) = \invlim{\mathbf{J}{(v, f)}} \CG(U,u).
\]
Let $(V_1,v_1)$, $(V_2,v_2)\in \aff_{\ff}/\SY$ and $(j, \beta) \colon (V_1, v_1) \to (V_2, v_2)$. In this situation we define a functor
\[
(j, \beta)_* \colon \mathbf{J}{(v_1, f)} \lto \mathbf{J}{(v_2, f)}
\]
by $(j, \beta)_*(U,u,h,\gamma) := (U,u,jh, \beta h \circ \gamma)$. This induces the restriction morphism
\[
(f_*\CG)(V_2,v_2) \lto (f_*\CG)(V_1,v_1).
\]

In particular, if $f$ is an affine morphism, then
$f_*\CG(V,v)=\CG(\SX \times_{\SY} V, p_1)$
because $(\SX \times_{\SY} V,p_1,p_2,\phi_0)$, with $\phi_0$ the canonical 2-cell in the 2-pullback, is a final object of $\mathbf{J}{(v, f)}$.
\end{cosa}

\begin{cosa}\label{ii}\textbf{Inverse image}.
Let $(U,u)\in \aff_{\ff}/\SX$, we denote by $\mathbf{I}(u,f)$ the category whose objects are the $2$-commutative squares as (\ref{abiertos}) with $(V,v) \in \aff_{\ff}/\SY$. We denote these objects now as $(V,v,h,\gamma)$. A morphism from $(V,v,h,\gamma)$ to $(V',v',h',\gamma')$ is a  morphism
$(j,\beta)\colon (V,v)\to (V', v')$ in
$\aff_{\ff}/\SY$ satisfying $j  h = h'$ and $\gamma' =
\beta h \circ \gamma$. For $\CF \in \pre(\aff_{\ff}/\SY)$, the presheaf
$f^{\pin}\CF$ is given by
\[
(f^{\pin} \CF)(U,u) = \dirlim{\mathbf{I}{(u, f)}} \CF(V,v)
\]
for each $(U,u) \in \aff_{\ff}/\SX$. Consider $(U_1,u_1)$, $(U_2,u_2)\in \aff_{\ff}/\SX$ together with $(l, \alpha) \colon (U_1,u_1) \to (U_2,u_2)$. In this situation we define a functor
\[
(l, \alpha)^* \colon \mathbf{I}{(u_2, f)} \lto \mathbf{I}{(u_1, f)}
\]
by $(l, \alpha)^*(V,v,h,\gamma) := (V,v,hl,\gamma l \circ f \alpha)$. This induces the restriction morphism
\[
(f^{\pin}\CF)(U_2,u_2) \lto (f^{\pin}\CF)(U_1,u_1).
\]

In particular, if $f$ is a flat finitely presented morphism, then $f^{\pin}\CF(U,u)=\CF(U,fu)$ because $(U,fu,\id_U,\iid)$ is an initial object of $\mathbf{I}{(u, f)}$.
\end{cosa}

\begin{lem}
Let $f_1, f_2 \colon \SX \to \SY$ be 1-morphisms of geometric stacks. In the previous setting, any 2-morphism $\zeta \colon f_1 \imp f_2$ induces equivalences of categories 
\begin{enumerate}
\item $\zeta_{\mathbf{J}} \colon \mathbf{J}(v,f_1) \liso \mathbf{J}(v,f_2)$,
\item $\zeta^{\mathbf{I}} \colon \mathbf{I}(u,f_2) \liso \mathbf{I}(u,f_1)$.
\end{enumerate}
\end{lem}

\begin{proof}
For (\emph{i}), define, using the notation as in \ref{di}, 
\[
\zeta_{\mathbf{J}}(U,u,h,\gamma) = (U,u,h,\gamma \circ \zeta^{-1} u)
\]
and extend in an obvious way to maps. We see that these data defines a functor and a quasi-inverse is given by $(\zeta^{-1})_{\mathbf{J}}$.

Analogously, for (\emph{ii}), define, using the notation as in \ref{ii},
\[
 \zeta^{\mathbf{I}}(V,v,h,\gamma) = (V,v,h,\gamma \circ \zeta u)
\]
As before, this defines a functor whose quasi-inverse is $(\zeta^{-1})^{\mathbf{I}}$.
\end{proof}

\begin{prop}\label{2fib}
Let $f_1, f_2 \colon \SX \to \SY$ be 1-morphisms of geometric stacks. With the previous notation, given a 2-morphism $\zeta \colon f_1 \imp f_2$ we have isomorphisms
\begin{enumerate}
\item $\zeta_* \colon f_{1 *} \liso f_{2 *}$,
\item $\zeta^{\pin} \colon f_{2}^{\pin} \liso f_{1}^{\pin}$.
\end{enumerate}
\end{prop}

\begin{proof}
To prove $\zeta_*$ is an isomorphism, let us apply the equivalence of small categories 
$\zeta_{\mathbf{J}} \colon \mathbf{J}(v,f_1) \iso \mathbf{J}(v,f_2)$ to the sets $\CG(U,u)$ with $(U,u) \in \mathbf{J}(v,f_1)$, this gives a map
\[
\invlim{\mathbf{J}{(v, f_1)}} \CG(U,u) \liso
\invlim{\mathbf{J}{(v, f_2)}} \CG(U,u).
\]
This isomorphism is natural, therefore induces a natural isomorphism
\[
f_{1 *}\CG \liso f_{2 *}\CG.
\]
For $\zeta^{\pin}$ use a similar argument with $\zeta^{\mathbf{I}}$ in place of $\zeta_{\mathbf{J}}$.
\end{proof}

\begin{prop}\label{adjpre6}
The previous constructions define a pair of adjoint functors
\[
\pre(\aff_{\ff}/\SX)
\underset{f_*}{\overset{f^{\pin}}{\longleftrightarrows}}
\pre(\aff_{\ff}/\SY).
\]
\end{prop}

\begin{proof}
We have to prove, for $\CF \in \pre(\aff_{\ff}/\SY)$ and $\CG \in \pre(\aff_{\ff}/\SX)$ that there is an isomorphism
\[
\Hom_{\SX}(f^{\pin}\CF, \CG) \liso \Hom_{\SY}(\CF,f_*\CG)
\]
A map $f^{\pin}\CF \to \CG$ induces for every $(U,u) \in \aff_{\ff}/\SX$ and $(V,v) \in \mathbf{I}(u,f)$ morphisms
\[
\CF(V,v) \lto \dirlim{\mathbf{I}{(u, f)}} \CF(V,v) \lto \CG(U,u)
\]
Now fixing $(V,v)$ and varying $(U,u)$ we get a map
\[
\CF(V,v) \lto \invlim{\mathbf{J}{(v, f)}} \CG(U,u).
\]
This construction provides the adjunction map. By a similar procedure, we obtain a map in the opposite direction. It is not difficult to check that both maps are mutually inverse.
\end{proof}

\begin{cosa}
The \emph{continuity} of our construction is expressed by the relation between the coverings. Let us spell it out.

Take a covering $\{(g_i, \beta_i) \colon (V_i,v_i) \to (V,v)\}_{i \in I}$ in the site $\aff_{\ff}/\SY$. Given $(U,u,h,\gamma) \in \mathbf{J}(v,f)$, consider for each $i \in I$ the following Cartesian square of affine schemes
\begin{center}
\begin{tikzpicture}
\matrix(m)[matrix of math nodes, row sep=2.6em, column sep=2.8em,
text height=1.5ex, text depth=0.25ex]{
U \times_V V_i & V_i \\
U              & V   \\}; 
\path[->,font=\scriptsize,>=angle 90]
(m-1-1) edge node[auto] {$h_i$}  (m-1-2)
        edge node[left] {$g'_i$} (m-2-1)
(m-1-2) edge node[auto] {$g_i$}  (m-2-2) 
(m-2-1) edge node[auto] {$h$}    (m-2-2);
\end{tikzpicture}
\end{center}
Set $U_i := U \times_V V_i$ and $\gamma_i \colon f u g'_i \imp v_i h_i$, given by $\gamma_i := \beta_i^{-1}h_i \circ \gamma g'_i$. This defines a functor
\begin{equation}\label{cont}
 \overline{g_i} \colon \mathbf{J}(v,f) \lto \mathbf{J}(v_i,f)
\end{equation}
by $\overline{g_i}(U,u,h,\gamma) := (U_i,ug'_i,h_i,\gamma_i)$.

A consequence of this observation is the following important statement.
\end{cosa}

\begin{prop} \label{p61}
If $\CG$ is a sheaf on $\aff_{\ff}/\SX$, then $f_*\CG$ is a sheaf.
\end{prop}

\begin{proof}
Let $\{(g_i, \beta_i) \colon (V_i,v_i) \to (V,v)\}_{i \in I}$ be  a covering in the site $\aff_{\ff}/\SY$. With the previous notation, given $(U,u,h,\gamma) \in \mathbf{J}(v,f)$, the family
\[
\{(g'_i, \iid) \colon (U_i, ug'_i) \to (U,u)\}_{i \in I}
\]
is a covering in the site $\aff_{\ff}/\SX$. Associated to the sheaf $\CG$ we have an equalizer diagram
\[
\begin{tikzpicture}
\matrix (m) [matrix of math nodes, row sep=3em, column sep=3em]{
\CG(U) & \prod_{i \in I}  \CG(U_i) & 
\prod_{i, j \in I}  \CG(U_i \times_U U_j).\\
}; 
\draw [transform canvas={yshift= 0.3ex},font=\scriptsize,->]
(m-1-2) -- node[above]{$\rho_2$}(m-1-3);
\draw [transform canvas={yshift= -0.3ex},font=\scriptsize,->]
(m-1-2) -- node[below]{$\rho_1$}(m-1-3);
\path[->,font=\scriptsize,>=angle 90]
(m-1-1) edge node[auto] {} (m-1-2);
\end{tikzpicture}
\]
Taking limits, we obtain another equalizer diagram (in the following we will keep the notation $\rho$ on the maps to avoid the clutter)
\[
\begin{tikzpicture}
\matrix (m) [matrix of math nodes, row sep=3em, column sep=3em]{
\invlim{\mathbf{J}{(v, f)}} \CG(U)
&  \invlim{\mathbf{J}{(v, f)}}   \prod_{i \in I}     \CG(U_i)
&  \invlim{\mathbf{J}{(v, f)}} \prod_{i, j \in I}  \CG(U_i \times_U U_j).\\
}; 
\draw [transform canvas={yshift= 1.6ex},font=\scriptsize,->]
(m-1-2) -- node[above]{$\rho_2$}(m-1-3);
\draw [transform canvas={yshift= 1ex},font=\scriptsize,->]
(m-1-2) -- node[below]{$\rho_1$}(m-1-3);
\path[transform canvas={yshift= 1.3ex},->,font=\scriptsize,>=angle 90]
    (m-1-1) edge node[auto] {} (m-1-2);
\end{tikzpicture}
\]
By commutation of limits \cite[Prop. 2.12.1]{bor}, this sequence becomes
\[
\begin{tikzpicture}
\matrix (m) [matrix of math nodes, row sep=3em, column sep=3em]{
\invlim{\mathbf{J}{(v, f)}} \CG(U)
&  \prod_{i \in I} \invlim{\mathbf{J}{(v, f)}}    \CG(U_i)
&  \prod_{i, j \in I} \invlim{\mathbf{J}{(v, f)}} \CG(U_i \times_U U_j).\\
}; 
\draw [transform canvas={yshift= 1.6ex},font=\scriptsize,->]
(m-1-2) -- node[above]{$\rho_2$}(m-1-3);
\draw [transform canvas={yshift= 1ex},font=\scriptsize,->]
(m-1-2) -- node[below]{$\rho_1$}(m-1-3);
\path[transform canvas={yshift= 1.3ex},->,font=\scriptsize,>=angle 90]
(m-1-1) edge node[auto] {} (m-1-2);
\end{tikzpicture}
\]
Let us denote by $\mathbf{J}_i{(v, f)}$ the essential image of $\mathbf{J}{(v, f)}$ in $\mathbf{J}{(v_i, f)}$ through the functor $\overline{g_i}$ \eqref{cont}. Note that 
\[
\invlim{\mathbf{J}{(v, f)}}    \CG(U_i) \quad=
\invlim{\mathbf{J}_i{(v, f)}}  \CG(U_i)
\]
All of this applies verbatim to $\mathbf{J}_{ij}{(v, f)}$.

Consider now the following commutative diagram
\begin{center}
\begin{tikzpicture}[baseline=(current  bounding  box.center)]
\matrix (m) [matrix of math nodes, row sep=3em, column sep=3em]{
\invlim{\mathbf{J}{(v, f)}} \CG(U)
&  \prod_{i \in I} \invlim{\mathbf{J}_i{(v, f)}}    \CG(U_i)
&  \prod_{i,j \in I} \invlim{\mathbf{J}_{ij}{(v, f)}} \CG(U_i \times_U U_j)\\
f_*\CG(V)    &   \prod_{i \in I}    f_*\CG(V_i)
&   \prod_{i, j \in I} f_*\CG(V_i \times_V V_j)\\
};
\draw [-, double distance=2pt] (m-1-1) -- (m-2-1);
\draw [transform canvas={yshift= 1.6ex},font=\scriptsize,->]
(m-1-2) -- node[above]{$\rho_2$}(m-1-3);
\draw [transform canvas={yshift= 1ex},font=\scriptsize,->]
(m-1-2) -- node[below]{$\rho_1$}(m-1-3);
\path[transform canvas={yshift= 1.3ex},->,font=\scriptsize,>=angle 90]
(m-1-1) edge node[auto] {} (m-1-2);
\draw [transform canvas={yshift= 0.3ex},font=\scriptsize,->]
(m-2-2) -- node[above]{$\rho_2$}(m-2-3); 
\draw[transform canvas={yshift=-0.3ex},font=\scriptsize,->] 
(m-2-2) -- node[below]{$\rho_1$}(m-2-3);
\path[->,font=\scriptsize,>=angle 90] (m-2-1) edge node[auto]{} (m-2-2);
\draw[->] (m-2-2) -- (m-1-2) node[midway, right, scale=0.75]{can}; 
\draw[->] (m-2-3) -- (m-1-3) node[midway, right, scale=0.75]{can};
\end{tikzpicture}
\end{center}
In a natural way, for every $\overline{U}=(U,u,h,\gamma) \in \mathbf{J}{(v_i, f)}$  we have $(U,u,g_1 h, \beta h \circ \gamma) \in \mathbf{J}{(v, f)}$. Then we get
\[
\overline{U}_i = \overline{g_i}(U,u,g_1 h, \beta h \circ \gamma) \in \mathbf{J}_i{(v, f)}
\]
jointly with maps in $\mathbf{J}{(v_i, f)}$, $\overline{U} \to \overline{U}_i$ and  $\overline{U}_i\to \overline{U}$ whose composition is the identity. This implies that the canonical map 
\[
f_*\CG(V_i) \quad \xto{\text{ can }} \invlim{\mathbf{J}_i{(v, f)}}    \CG(U_i)
\]
is a monomorphism. Similarly the rightmost vertical map is a monomorphism, too. The top row is an equalizer and this implies that the bottom row is. As a consequence, $f_*\CG$ a sheaf.
\end{proof}

\begin{cor}\label{adjhaz6}
 Let $\CF \in \SY_{\ff}$. We denote by $f^{-1}\CF$ the sheaf
associated to the presheaf $f^{\pin}\CF$. There is a pair of adjoint functors
\[
\SX_{\ff} \,
\underset{f_*}{\overset{f^{-1}}{\longleftrightarrows}} \,
\SY_{\ff}.
\]
\end{cor}
\begin{proof}
 Combine the previous Propositions \ref{adjpre6} and \ref{p61}, having in mind the adjoint property of sheafification.
\end{proof}

\begin{rem} In the case in which $f$ is a flat finitely presented morphism, then $f^{-1}$ is exact.

Notice that, in general,  the pair $(f^{-1},f_*)$ \emph{does
not} define a morphism of topoi, since $f^{-1}$ need not be
left exact. For an explicit example for the \emph{lisse-\'etale} topology see \cite[5.3.12]{beh}. The example remains valid for the $\ff$ topology because it is finer than the \emph{lisse-\'etale} one.
\end{rem}

\begin{prop}\label{l62}
 Let $f \colon \SX \to \SY$ and $g \colon \SY \to \SZ$ be 1-morphisms of geometric stacks. It holds that $(g f)_* \cong g_* f_*$
\end{prop}

\begin{proof}
Let $\CF \in \SX_{\ff}$. For each $(W,w) \in \aff_{\ff}/\SZ$, we have that
\[
(gf)_*\CF(W,w) \quad = \invlim{\mathbf{J}{(w, gf)}} \CF(U,u)
\]
where $\mathbf{J}(w,gf)$ is as in \ref{di}. On the other hand
\[
g_*f_*\CF(W,w) \quad = \invlim{\mathbf{J}{(w, g)}} f_* \CF (V, v).
\]
We have a family of maps compatible with morphisms in $\mathbf{J}{(w, g)}$
\[
(gf)_*\CF(W,w) \lto f_* \CF (V, v)
\]
defining a natural map
\begin{equation}\label{fl1}
(gf)_*\CF(W,w) \lto g_*f_*\CF(W,w).
\end{equation}
Let us show that this map is an isomorphism.

Next, for each $(U,u,h,\gamma) \in \mathbf{J}{(w, gf)}$ we will construct a map
\[
g_*f_*\CF(W,w) \lto \CF(U,u)
\]
compatible with morphisms in $\mathbf{J}{(w, gf)}$. Consider the 2-fibered product stack $\SY' := \SY \times_{\SZ} W$, its canonical projections $q_1$ and $q_2$, and $\gamma_0 \colon gq_1 \imp wq_2$ its associated 2-morphism. $\SY'$ is a geometric stack by Proposition \ref{2cartsq}. There is a morphism  $f' \colon U \to \SY'$ given by the 2-isomorphism $\gamma \colon g f u \imp w h$, such that $q_1 f' = f u$, $h = q_2 f'$ and $\gamma_0 f' = \gamma$.
Let $q \colon V' \to \SY'$ be a presentation. Consider the 2-fiber squares
\begin{center}
\begin{tikzpicture}
\matrix(m)[matrix of math nodes, row sep=2.6em, column sep=2.8em,
text height=1.5ex, text depth=0.25ex]{
U' & V'   \\
U  & \SY' \\}; 
\path[->,font=\scriptsize,>=angle 90]
(m-1-1) edge node[auto] {$f''$} (m-1-2)
        edge node[left] {$q'$} (m-2-1)
(m-1-2) edge node[auto] {$q$} (m-2-2)
(m-2-1) edge node[auto] {$f'$} (m-2-2);
\draw [shorten >=0.2cm,shorten <=0.2cm,->,double] (m-2-1) -- (m-1-2) node[auto, midway,font=\scriptsize]{$\gamma''$};
\end{tikzpicture}
 \hspace{6mm}
\begin{tikzpicture}
\matrix(m)[matrix of math nodes, row sep=2.6em, column sep=2.8em,
text height=1.5ex, text depth=0.25ex]{
V'' & V'                 \\
V'  & \SY' \\}; 
\path[->,font=\scriptsize,>=angle 90]
(m-1-1) edge node[auto] {$p'_2$} (m-1-2)
        edge node[left] {$p'_1$} (m-2-1)
(m-1-2) edge node[auto] {$q$} (m-2-2)
(m-2-1) edge node[auto] {$q$} (m-2-2);
\draw [shorten >=0.2cm,shorten <=0.2cm,->,double] (m-2-1) -- (m-1-2) node[auto, midway,font=\scriptsize]{$\gamma'''$};
\end{tikzpicture}
\end{center}
Denote by $v' \colon V' \to \SY$ the map $v' = q_1q$. It is clear that $(V',v', q_2q, \gamma_0 q) \in \mathbf{J}{(w, g)}$. Take $v'' := v' p'_1$. In $\mathbf{J}{(w, g)}$, we have the projections $(p'_1, \iid)$ and $(p'_2, q_1\gamma''')$ from $(V'',v'', q_2qp'_1, \gamma_0 qp'_1)$ to $(V',v', q_2q, \gamma_0 q)$. We have a diagram
\[
\begin{tikzpicture}
\matrix (m) [matrix of math nodes, row sep=3em, column sep=4.5em]{
g_*f_*\CF(W,w) &  f_* \CF (V', v') &  f_* \CF (V'', v'') \\
\CF(U,u)       &  \CF(U', u')      &  \CF(U' \times_U U', u'') \\
}; 
\draw [transform canvas={yshift= 0.3ex},font=\scriptsize,->]
(m-1-2) -- node[above]{$f_*\CF(p'_1,\iid)$}(m-1-3);
\draw [transform canvas={yshift= -0.3ex},font=\scriptsize,->]
(m-1-2) -- node[below]{$f_*\CF(p'_2,q_1 \gamma''')$}(m-1-3);
\draw[->,font=\scriptsize,>=angle 90]
(m-1-1) edge node[auto]{$\lambda$} (m-1-2);
\draw [transform canvas={yshift= 0.3ex},font=\scriptsize,->]
(m-2-2) -- node[above]{$\CF(p_1,\iid)$}(m-2-3);
\draw [transform canvas={yshift= -0.3ex},font=\scriptsize,->]
(m-2-2) -- node[below]{$\CF(p_2,\iid)$}(m-2-3);
\path[->,font=\scriptsize,>=angle 90]
(m-2-1) edge node[auto]{$\CF(q',\iid)$} (m-2-2);
\draw[->,font=\scriptsize, dashed] (m-1-1) -- (m-2-1) ;
\draw[->,font=\scriptsize] (m-1-2) -- node[left]{can}(m-2-2); 
\draw[->,font=\scriptsize] (m-1-3) -- node[right]{can}(m-2-3);
\end{tikzpicture}
\]
where $p_1, p_2 \colon U' \times_U U' \to U'$ are the canonical projections, $u' = uq'$, $u'' = u' p_1$ and the vertical maps labelled ``can'' are induced by the limit diagrams indexed by $\mathbf{J}{(v', f)}$ and $\mathbf{J}{(v'', f)}$. The map $\lambda$ is the canonical map corresponding to the limit of the diagram indexed by $\mathbf{J}{(w, g)}$.

Notice that $q$ is smooth and surjective and the same holds for $q'$ therefore, $\{q' \colon U' \to U\}$ is a (flat) covering in $\aff_{\ff}/\SX$. As $\CF$ is a sheaf, the lower row is an equalizer diagram. We have that $f_*\CF(p'_1,\iid) \lambda = f_*\CF(p'_2,q_1 \gamma''') \lambda$, therefore there is a canonical vertical dashed map. Passing to the limit along $\mathbf{J}{(w, gf)}$ we obtain a natural map
\begin{equation*}
g_*f_*\CF(W,w) \lto 
(gf)_*\CF(W,w)
\end{equation*}
which is inverse to (\ref{fl1}).
\end{proof}

\begin{rem}
 The long proof of the previous result is due to the lack of an underlying map of sites, which, in turn, forces the awkward definition in \ref{di}. Of course, if $f$ is representable, by using an appropriate site of schemes or algebraic spaces, the preceding proof becomes the usual short one.
\end{rem}

\begin{cor}
In the previous situation $(g  f)^{-1} \liso f^{-1} g^{-1}$.
\end{cor}

\begin{proof}
 It is a consequence of the previous proposition by adjunction.
\end{proof}

\begin{cosa}
In general, for $f\colon \SX \to \SY$, there is no map between ringed sites, however, we consider the morphism of sheaves of rings $f^\#\colon \CO_{\SY} \to f_*\CO_{\SX}$, given for each
 $(V,v)\in \aff_{\ff}/\SY$ by the induced ring homomorphism
\[
f^\#_{(V,v)}\colon \CO_{\SY}(V,v) \lto
\invlim{\mathbf{J}{(v,f)}} \CO_{\SX}(U,u).
\]

We want to extend the previous adjunction to the categories of sheaves of modules. As we do not have neither a topos morphism neither a continuous map of sites, this extension is not straightforward.

The functor $f_*$ takes $\CO_{\SX}$-Modules to $f_*\CO_{\SX}$-Modules and, by the forgetful functor associated to $f^\#$, to $\CO_{\SY}$-Modules. Note that the formula in Proposition \ref{l62} keeps holding by the transitivity of the forgetful functors.

We will apply Lemma~\ref{Olss} to guarantee that the functor $f^{\pin}$ commutes with finite products and so does $f^{-1}$. It will follow that if $\CF$ is an $\CO_{\SY}$-module, $f^{-1}\CF$ will be an $f^{-1}\CO_{\SY}$-module. Therefore we have to check the hypothesis of Lemma~\ref{Olss}, so we are reduced to prove the following.
\end{cosa}

\begin{lem}\label{fprod}
 The category $\mathbf{I}(u,f)$ has finite products. 
\end{lem}

\begin{proof}
We only need to prove that any two objects in $\mathbf{I}(u,f)$ admit a product. For that, let $(V_1,v_1,h_1,\gamma_1)$ and $(V_2,v_2,h_2,\gamma_2)$ be objects in $\mathbf{I}(u,f)$. For $i \in \{1,2\}$, let $p_i$ be the projection from $V_1\times_{\SY}V_2$ to $V_i$, and $\beta \colon v_1 p_1 \imp v_2 p_2$ be the canonical 2-morphism. If $h \colon U \to V_1\times_{\SY}V_2$ is the morphism satisfying $p_1 \, h=h_1$, $p_2 \, h=h_2$ and $\beta \,h=\gamma_2 \circ \gamma_1^{-1}$, then the object $(V_1\times_{\SY}V_2,v_1\, p_1,h,\gamma_1)$ together with the morphisms $(p_1,\iid)$ and $(p_2,\beta)$ is a product
of $(V_1,v_1,h_1,\gamma_1)$ and $(V_2,v_2,h_2,\gamma_2)$ in $\mathbf{I}(u,f)$.
\end{proof}

\begin{cor}\label{adjmod}
 The functor $f^*\colon \CO_{\SY}\md \to \CO_{\SX}\md$ given by
\[
f^*\CF := \CO_{\SX}\otimes_{f^{-1}\CO_{\SY}}f^{-1}\CF
\]
is left adjoint to the functor $f_*\colon \CO_{\SX}\md \to \CO_{\SY}\md$.
\end{cor}

\begin{proof}
 Combine Corollary \ref{adjhaz6} with the previous discussion.
\end{proof}

\begin{cor}
 Let $f \colon \SX \to \SY$ and $g \colon \SY \to \SZ$ be 1-morphisms of geometric stacks. It holds that $(g f)^* \cong f^* g^*$.
\end{cor}

\begin{proof} 
It is a consequence of Proposition \ref{l62} by adjunction.
\end{proof}

\begin{cor}
 Let $f_1, f_2 \colon \SX \to \SY$ be 1-morphisms of geometric stacks. With the previous notation, given a 2-morphism $\zeta \colon f_1 \imp f_2$ we have isomorphisms
\begin{enumerate}
\item $\zeta_* \colon f_{1 *} \liso f_{2 *}$
\item $\zeta^* \colon f_{2}^* \liso f_{1}^*$
\end{enumerate}
as functors defined on $\CO_{\SX}\md$ and $\CO_{\SY}\md$, respectively.
\end{cor}

\begin{proof}
 Immediate from \ref{2fib} and the previous discussion.
\end{proof}

\begin{rem}
Notice that the isomorphism $\phi^*$ in \ref{fund} is a particular case of  (\emph{ii}).
\end{rem}

\begin{prop} \label{qcoinv}
If $\CF \in \qco(\SY)$ then $f^*\CF \in \qco(\SX)$.
\end{prop}
\begin{proof}
Let $p \colon X \to \SX$ and $q:Y \to \SY$ be affine presentations such that we have a $2$-commutative square
\begin{center}
\begin{tikzpicture}
\matrix(m)[matrix of math nodes, row sep=2.6em, column sep=2.8em,
text height=1.5ex, text depth=0.25ex]
{X & Y\\
\SX&\SY\\}; \path[->,font=\scriptsize,>=angle 90]
(m-1-1) edge node[auto] {$f_0$} (m-1-2)
        edge node[left] {$p$} (m-2-1)
(m-1-2) edge node[auto] {$q$} (m-2-2)
(m-2-1) edge node[auto] {$f$} (m-2-2);
\draw [shorten >=0.2cm,shorten <=0.2cm,->,double] (m-2-1) -- (m-1-2) node[auto, midway,font=\scriptsize]{$\gamma$};
\end{tikzpicture}
\end{center}
It follows from the previous Corollary that $p^* f^* \CF \cong f_0^* q^*\CF$, via $\gamma^*$ using pseudofunctoriality.
By  Proposition \ref{desc1} $q^*\CF \in \cart(Y)$ and by Proposition \ref{adjaf}, $f_0^* q^* \CF \in \cart(X)$. We conclude applying Lemma \ref{descart} and Theorem \ref{agree}.
\end{proof}

\begin{prop}\label{qcodir}
If $\CG \in \qco(\SX)$ then $f_*\CG \in \qco(\SY)$.
\end{prop}
\begin{proof}
If $f$ is an affine morphism the proof is similar to the one given in
Proposition \ref{adjaf}. If $f$ is not affine, take an affine
presentation $p \colon X \to \SX$. The morphism $f p$ is affine.
Let $\SST := p_*p^*$ as in Lemma \ref{descart} and consider the following equalizer of quasi-coherent sheaves
\[
\CG {\xto{\eta_{\CG}}}
\SST \CG \,
\underset{\eta_{\SST\CG}}{\overset{\SST\eta_{\CG}}{\longrightrightarrows}}\,
\SST^2 \CG.
\]
Since $f_*$ is left exact we have the equalizer of
$\CO_{\SY}$-modules
\[
f_*\CG \,\,
{\xto{f_*\eta_{\CG}}}
f_*\SST  \CG \,\,
\underset{f_*\eta_{\SST\CG}}{\overset{f_*\SST\eta_{\CG}}{\longrightrightarrows}} \, \, f_*\SST^2 \CG.
\]
From Proposition \ref{desc1} and Theorem \ref{agree} by using Proposition \ref{l62} it follows that $f_*\SST\CG$ and $f_*\SST^2\CG$ are quasi-coherent sheaves. Then, $f_*\CG$ is a quasi-coherent
Module since it is the kernel of a morphism of quasi-coherent Modules.
\end{proof}

\begin{cor}
There is a pair of adjoint functors
\[
\qco(\SX) \,
\underset{f_*}{\overset{f^*}{\longleftrightarrows}} \,
\qco(\SY).
\]
\end{cor}

\begin{proof}
Combine the adjunction of Corollary \ref{adjmod} with Propositions \ref{qcoinv} and \ref{qcodir}.
\end{proof}

\section{Describing functoriality via comodules}\label{sec7}

  Let $\varphi \colon A_\bullet \to B_\bullet$ be a  homomorphism of Hopf algebroids. Following \cite{hov}, we describe a pair of adjoint functors
\[
\begin{tikzpicture}
\matrix (m) [matrix of math nodes, row sep=3em, column sep=3em]{
B_\bullet\com &  A_\bullet\com. \\
}; 
\draw [transform canvas={yshift= 0.3ex},font=\scriptsize,<-]
(m-1-1) -- node[above]{$B_0\otimes_{A_0}-$}(m-1-2);
\draw [transform canvas={yshift= -0.3ex},font=\scriptsize,->]
(m-1-1) -- node[below]{$\SU^\varphi$}(m-1-2);
\end{tikzpicture}
\]
We will see that for the corresponding 1-morphism of stacks $f := \stack(\varphi)$
\[
f \colon \stack(B_\bullet) \lto \stack(A_\bullet),
\]
this adjunction corresponds to the previously considered $f^* \dashv f_*$ between sheaves of quasi-coherent modules on the small flat site.

\begin{cosa}\label{homHopf}
 A homomorphism of Hopf algebroids $\varphi \colon A_\bullet \to B_\bullet$
is a pair of ring homomorphisms $\varphi_i \colon A_i \to B_i$ with
$i \in \{0,1\}$ that respects the structure of a Hopf algebroid, namely, the following equalities hold
\begin{align*}
 \varphi_1 \eta_L   &= \eta_L \varphi_0   & 
 \epsilon \varphi_1 &= \varphi_0 \epsilon        \\
 \varphi_1 \eta_R   &= \eta_R \varphi_0   &
 \kappa \varphi_1   &= \varphi_1 \kappa   & 
 (\varphi_1 \otimes \varphi_1)\nabla      &= \nabla  \varphi_1.
\end{align*}

\end{cosa}

\begin{cosa}
 The homomorphism $\varphi \colon A_\bullet \to B_\bullet$ induces a functor
\[
B_0\otimes_{A_0}- \colon A_\bullet\com \lto B_\bullet\com,
\]
sending an $A_\bullet$-comodule $(M, \psi)$ to the $B_\bullet$-comodule $(B_0\otimes_{A_0}M, \psi')$. Let 
\[
\overline{\varphi}(b_0 \otimes a_1) := \eta_L(b_0)\,\varphi_1(a_1), \text{ for } a_1 \in A_1 \text{ and } b_0 \in B_0. 
\]
The structure map $\psi'$ is given by the composition
\[
B_0 \otimes_{A_0} M \xto{\id\otimes \psi}
B_0 \otimes_{\eta_L} A_{1 \,\eta_R} \otimes_{A_0} M \xto {\overline{\varphi} \otimes \id}
B_1\: {_{\eta_R}\otimes_{A_0}} M \liso
B_1\: {_{\eta_R}\otimes_{B_0}} (B_0\otimes_{A_0} M).
\]
\end{cosa}

\begin{cosa}\label{olv}
From \cite[Proposition 1.2.3]{hov}, we will describe
\[
\SU^\varphi \colon B_\bullet\com \lto A_\bullet\com,
\]
a right adjoint for the functor $B_0\otimes_{A_0}-$. This will be done in two steps. 

Consider first an extended comodule, that is, given $N \in B_0\md$, let us consider the comodule $B_1 \otimes_{B_0}N$ with the induced structure. We define
\[
\SU^\varphi(B_1 \otimes_{B_0}N) := A_1 \otimes_{A_0}N
\] 
with the induced structure. For maps $\lambda \colon B_1 \otimes_{B_0}N \to B_1 \otimes_{B_0}N'$ with $N, N' \in B_0\md$, then $\SU^\varphi(\lambda)$ is defined as the following composition
\begin{align*}
A_1\: {_{\eta_R} \otimes_{A_0}} N &\xto{\nabla \otimes \id}
    A_1\: {_{\eta_R}\otimes_{\eta_L}} A_1\: {_{\eta_R}\otimes_{A_0}} N \\
&\xto{\id\otimes \varphi_1 \otimes \id}
A_1\: {_{ \eta_R} \otimes _{\eta_L \varphi_0}}\: B_1\: {_{\eta_R} \otimes _{B_0}}\: N\\
&\xto{\id\otimes \lambda} A_1\: {_{\eta_R} \otimes_{\eta_L \varphi_0}}\: B_1\: {_{\eta_R} \otimes_{B_0}}\: N'\\
&\xto{\id\otimes \epsilon \otimes \id} A_1\: {_{\eta_R} \otimes _{A_0}}N'.
\end{align*}

Now, let $(N, \psi)$ be an arbitrary $B_\bullet$-comodule. Let $\lambda' \colon B_1 \otimes_{B_0} N \to N'$ be the cokernel of $\psi$ as $B_0$-modules. Let $\psi'$ be the induced comodule structure on $N'$ and consider the following $A_\bullet$-comodule homomorphism:
\[
A_1\: {_{\eta_R} \otimes _{A_0}}N  \xto{\SU^\varphi(\psi' \lambda')}
A_1\: {_{\eta_R} \otimes _{A_0}}N'.
\]
Define
\[
\SU^\varphi(N) := \ker(\SU^\varphi(\psi' \lambda'))
\]
The extension to morphisms is straightforward by the naturality of kernels.
\end{cosa}

The following result is due to Hovey \cite{hov}. We have included a proof for the reader's convenience.

\begin{prop}
Given a homomorphism $\varphi \colon A_\bullet \to B_\bullet$ of Hopf algebroids, there is an associated pair of adjoint functors
\[
\begin{tikzpicture}
\matrix (m) [matrix of math nodes, row sep=3em, column sep=3em]{
B_\bullet\com &  A_\bullet\com. \\
}; 
\draw [transform canvas={yshift= 0.3ex},font=\scriptsize,<-]
(m-1-1) -- node[above]{$B_0\otimes_{A_0}-$}(m-1-2);
\draw [transform canvas={yshift= -0.3ex},font=\scriptsize,->]
(m-1-1) -- node[below]{$\SU^\varphi$}(m-1-2);
\end{tikzpicture}
\]
\end{prop}

\begin{proof}
Let $(P, \psi) \in B_\bullet\com$ and $(M, \psi') \in A_\bullet\com$. We have to construct an isomorphism
\[
\Hom_{B_\bullet\com}(B_0 \otimes_{A_0} M, P) \liso 
\Hom_{A_\bullet\com}(M, \SU^\varphi(P))
\]
First, if $P$ is extended, \ie{}\! $P = B_1 \otimes_{B_0} N$ for a $B_0$-module $N$, then we have
\begin{align*}
\Hom_{B_\bullet\com}(B_0 \otimes_{A_0} M, P)
&= \Hom_{B_\bullet\com}(B_0 \otimes_{A_0} M, B_1 \otimes_{B_0} N) \\
&\cong \Hom_{B_0\md}(B_0 \otimes_{A_0} M, N) \\
&\cong \Hom_{A_0\md}(M, N)\\
&\cong \Hom_{A_\bullet\com}(M, A_1 \otimes_{A_0} N)\\
&= \Hom_{A_\bullet\com}(M, \SU^\varphi(P))
\end{align*}

In the general case, $P$ may be obtained as a kernel of a morphism between extended $B_\bullet$-comodules, in other words, we have an exact sequence
\[
0 \lto P \lto P' \lto P''
\]
with $P'$ and $P''$ extended $B_\bullet$-comodules. Now we have a diagram
\[
\begin{tikzpicture}
\matrix (m) [matrix of math nodes, row sep=2.5em, column sep=2em]{
\Hom_{B_\bullet}(B_0 \otimes_{A_0} M, P) & \Hom_{B_\bullet}(B_0 \otimes_{A_0} M, P') & \Hom_{B_\bullet}(B_0 \otimes_{A_0} M, P'') \\
\Hom_{A_\bullet}(M, \SU^\varphi(P)) & \Hom_{A_\bullet}(M, \SU^\varphi(P')) & \Hom_{A_\bullet}(M, \SU^\varphi(P''))  \\
};
\draw [->] (m-1-2) -- (m-1-3);
\draw [->] (m-2-2) -- (m-2-3);
\draw[->] (m-1-2) -- node[auto] {$\wr$} (m-2-2);
\draw[->] (m-1-3) -- node[auto] {$\wr$} (m-2-3);
\draw[dashed, ->] (m-1-1) -- (m-2-1);
\path[->] (m-1-1) edge (m-1-2)
          (m-2-1) edge (m-2-2);
\end{tikzpicture}
\]
in which the last two vertical maps are isomorphisms because $P'$ and $P''$ are extended comodules, therefore they induce the dashed vertical map which is also an isomorphism as desired.
\end{proof}


\begin{cosa}
Our next task is to identify the functor $\SU^\varphi$ with the direct image of sheaves on the geometric stacks corresponding to the Hopf algebroids.

Let $f \colon \SX \to \SY$ be a morphism of geometric stacks, $p
\colon X \to \SX$ and $q \colon Y \to \SY$ affine presentations
and $f_0 \colon X \to Y$ be a morphism such that the following
square
\begin{center}
\begin{tikzpicture}
\matrix(m)[matrix of math nodes, row sep=2.6em, column sep=2.8em,
text height=1.5ex, text depth=0.25ex]
{X & Y\\
\SX&\SY\\}; \path[->,font=\scriptsize,>=angle 90]
(m-1-1) edge node[auto] {$f_0$} (m-1-2)
        edge node[left] {$p$} (m-2-1)
(m-1-2) edge node[auto] {$q$} (m-2-2)
(m-2-1) edge node[auto] {$f$} (m-2-2);
\draw [shorten >=0.2cm,shorten <=0.2cm,->,double] (m-2-1) -- (m-1-2) node[auto, midway,font=\scriptsize]{$\gamma$};
\end{tikzpicture}
\end{center}
is 2-commutative.
Set $\SX = \stack(B_\bullet)$ and $\SY = \stack(A_\bullet)$.
Denote the ring homomorphisms associated to $f$ as $\varphi_0\colon A_0
\to B_0$ and $\varphi_1\colon A_1 \to B_1$. We remind the reader that
$\varphi=(\varphi_0,\varphi_1)$ is a homomorphism of Hopf algebroids (see \ref{homHopf}).
\end{cosa}

\begin{prop}\label{agreement}
There is a natural isomorphism of functors from $\qco(\SX)$ to $A_\bullet\com$
\[
\Upsilon \colon 
\GGamma^{\SY}_q f_*
\, \liso \,
\SU^\varphi \GGamma^{\SX}_p
\]
\end{prop}

\begin{proof}
We will start considering an object in the source that lies in the essential image of the functor $\SST = p_*p^*$. Let $\CF \in \qco(\SX)$ with $p^*\!\CF = \widetilde{M}$, for $M \in B_0\md$. By Proposition \ref{modcomod}(\emph{ii}), $\GGamma_p^{\SX}\SST\CF = (B_1 \otimes_{B_0} M, \nabla \otimes \id)$. By \ref{olv}
\[
\SU^\varphi \GGamma_p^{\SX} \SST \CF = 
(A_1 \otimes_{A_0} M, \nabla \otimes \id)
\]
On the other hand, by 2-functoriality of direct images there is a canonical isomorphism via $\gamma^{-1}_*$
\[
f_*\SST\CF = f_*p_*p^*\!\CF \iso q_*f_{0 *}\widetilde{M} 
= q_* \widetilde{M_{[\varphi_0]}}.
\] 
Therefore there is an isomorphism $\GGamma_q^{\SY} f_*\SST\CF \iso \GGamma_q^{\SY} q_*\widetilde{M_{[\varphi_0]}}$ and we have
\[
\GGamma_q^{\SY} q_*\widetilde{M_{[\varphi_0]}} = 
(A_1 \otimes_{A_0} M, \nabla \otimes \id)
\] 
by Proposition \ref{modcomod}(\emph{ii}), again. This defines $\Upsilon_{\SST\CF} \colon \GGamma_q^{\SY} f_* \SST\CF \, \liso \, \SU^\varphi \GGamma_p^{\SX} \SST\CF$.
Notice that $\Upsilon_{\SST(-)}$ is a natural transformation.


Finally we extend this to all quasi-coherent sheaves on $\SX$ and all morphisms. Consider the canonical equalizer of quasi-coherent $\CO_{\SX}$-modules
\[
\CF \; \; {\overset{\eta_{\CF}}{\lto}} \; \; 
	\SST \CF\, \underset{\SST\eta_{\CF}}{\overset{\eta_{\SST\CF}}{\longrightrightarrows}}
\; \; \SST^2  \CF.
\]
We may thus present every object in $\qco(\SX)$ as a kernel of objects that lie in the image of the functor $\SST$. Consider the following commutative equalizer diagram
\[
\begin{tikzpicture}
\matrix (m) [matrix of math nodes, row sep=3em, column sep=3.5em]{
\GGamma_q^{\SY}f_*\CF        &  \GGamma_q^{\SY} f_*\SST \CF\,    
                                      & \GGamma_q^{\SY} f_* \SST^2 \CF \\
\SU^\varphi\GGamma_p^{\SX} \CF &   \SU^\varphi \GGamma_p^{\SX} \SST \CF      
                                      & \SU^\varphi\GGamma_p^{\SX} \SST^2 \CF . \\
}; 
\draw[->,font=\scriptsize, dashed] 
(m-1-1) -- node[left]{$\Upsilon_{\CF}$}(m-2-1) ;
\draw [transform canvas={yshift= 0.3ex},font=\scriptsize,->]
(m-1-2) -- node[above]{$\GGamma_q^{\SY}f_*\eta_{\SST\CF}$} (m-1-3);
\draw [transform canvas={yshift=-0.3ex},font=\scriptsize,->]
(m-1-2) -- node[below]{$\GGamma_q^{\SY}f_*\SST\eta_{\CF} $} (m-1-3);
\draw [transform canvas={yshift= 0.3ex},font=\scriptsize,->]
(m-2-2) -- node[above]{$\SU^\varphi\GGamma_p^{\SX} \eta_{\SST\CF}$} (m-2-3);
\draw [transform canvas={yshift=-0.3ex},font=\scriptsize,->]
(m-2-2) -- node[below]{$\SU^\varphi\GGamma_p^{\SX} \SST\eta_{\CF}$} (m-2-3);
\draw[->,font=\scriptsize] (m-1-2) -- node[left]{$\Upsilon_{\SST\CF}$}(m-2-2); 
\draw[->,font=\scriptsize] (m-1-3) -- node[right]{$\Upsilon_{\SST^2\CF}$}(m-2-3);
\draw[->,font=\scriptsize]
(m-1-1) edge node[auto] {$\GGamma_q^{\SY}f_*\eta_{\CF}$} (m-1-2);
\draw[->,font=\scriptsize]
(m-2-1) edge node[auto] {$\SU^\varphi\GGamma_p^{\SX}\eta_{\CF}$} (m-2-2);
\end{tikzpicture}
\]
Notice that $\SU^\varphi \GGamma_p^{\SX} \SST \eta_{\CF} \,\Upsilon_{\SST\CF} =
\Upsilon_{\SST^2\CF} \;\GGamma_q^{\SY} f_*\SST \eta_{\CF}$ being $\Upsilon_{\SST(-)}$ a natural transformation. Let us check that 
\[
\SU^\varphi \GGamma_p^{\SX} \eta_{\SST\CF}\, \Upsilon_{\SST\CF} =
\Upsilon_{\SST^2\CF}\,\GGamma_q^{\SY}f_*\eta_{\SST\CF}.
\]
Applying $\GGamma_p^{\SX}$ to the map $\eta_{\SST\CF} \colon \SST\CF \to \SST^2\CF$, we get 
\[
\CF(p_{13}, \iid) \colon 
\CF(X_1, pp_2) \lto 
\CF(X_1 \times_{\SX} X, pp_3).
\]
In other words, having in mind that $p_{13}w = m$ where $w \colon X_2 \to X_1\times_{\SX}X$ is the isomorphism in \ref{uvedoble} and $m = \spec(\nabla)$, taking $M := \CF(X, p)$, this morphism corresponds to: 
\[
\nabla \otimes \id \colon 
B_1\: {_{\eta_R} \otimes_{A_0}} M \lto 
B_{1\eta_R}\otimes_{\eta_L}B_{1\eta_R} \otimes_{A_0} M.
\] 
Now, we apply $\SU^\varphi$ and we obtain the following composition
\[
A_1\: {_{\eta_R} \otimes_{A_0}} M \xto{\nabla \otimes \id}
A_1\: {_{\eta_R} \otimes_{\eta_L}} A_1\: {_{\eta_R}\otimes_{A_0}} M
                                  \xto{\id \otimes \varphi_1 \otimes \id} 
A_1\: {_{\eta_R} \otimes_{\eta_L}} B_{1\eta_R} \otimes_{B_0} M.
\]
On the other hand, consider the 2-cartesian squares
\begin{center}
\begin{tikzpicture}
\matrix(m)[matrix of math nodes, row sep=2.6em, column sep=2.8em,
text height=1.5ex, text depth=0.25ex]{
Y\times_{\SY}X   & X  \\
Y& \SY\\}; \path[->,font=\scriptsize,>=angle 90]
(m-1-1) edge node[auto] {$\pi_2$} (m-1-2)
        edge node[left] {$\pi_1$} (m-2-1)
(m-1-2) edge node[auto] {$f p$} (m-2-2)
(m-2-1) edge node[auto] {$q$} (m-2-2); 
\draw [shorten >=0.2cm,shorten <=0.2cm,->,double] (m-2-1) -- (m-1-2) 
node[auto, midway,font=\scriptsize]{$\phi_1$};
\end{tikzpicture} 
\hspace{6mm} 
\begin{tikzpicture}
\matrix(m)[matrix of math nodes, row sep=2.6em, column sep=2.8em,
text height=1.5ex, text depth=0.25ex]{
(Y\times_{\SY}X)\times_{\SX}X   & X  \\
Y\times_{\SY}X & \SX\\}; \path[->,font=\scriptsize,>=angle 90]
(m-1-1) edge node[auto] {$\pi'_2$} (m-1-2)
        edge node[left] {$\pi'_1$} (m-2-1)
(m-1-2) edge node[auto] {$p$} (m-2-2) 
(m-2-1) edge node[auto] {$p\pi_2$} (m-2-2); 
\draw [shorten >=0.2cm,shorten <=0.2cm,->,double] (m-2-1) -- (m-1-2) 
node[auto, midway,font=\scriptsize]{$\phi_2$};
\end{tikzpicture}
\end{center}
and the morphism 
$l \colon (Y \times_{\SY} X) \times_{\SX} X \lto Y \times_{\SY} X$
satisfying that $\pi_1 l = \pi_1 \pi'_1$, $\pi_2 l = \pi'_2$ and $\phi_1 l = f \phi_2 \circ \phi_1 \pi'_1$. Applying $\GGamma_q^{\SY} f_*$ to $\eta_{\SST\CF}$, we have the following commutative diagram
\begin{center}
\begin{tikzpicture}
\matrix(m)[matrix of math nodes, row sep=2.6em, column sep=3.5em,
text height=1.5ex, text depth=0.25ex]{
\GGamma_q^{\SY} f_* \SST\CF     & \GGamma_q^{\SY} f_* \SST^2\CF  \\
\CF(Y\times_{\SY} X, p\pi_2) & \CF((Y\times_{\SY} X)\times_{\SX} X, p\pi'_2) 
\\}; 
\path[->,font=\scriptsize,>=angle 90]
(m-1-1) edge node[auto] {$\GGamma_q^{\SY} f_* \eta_{\SST\CF}$} (m-1-2)
        edge node[left] {via $\gamma_*$} (m-2-1)
(m-1-2) edge node[auto] {via $\gamma_*$} (m-2-2)
(m-2-1) edge node[auto] {$\CF(l, \iid)$} (m-2-2);
\end{tikzpicture}
\end{center}
Consider now the diagram 
\begin{center}
\begin{tikzpicture}
\matrix(m)[matrix of math nodes, row sep=2.6em, column sep=3.5em,
text height=1.5ex, text depth=0.25ex]{
(Y \times_{\SY} X) \times_{\SX} X   & Y \times_{\SY} X   \\
Y_1\: {_{p_2}\times_{f_0 p_1}} X_1  & Y_1\: {_{p_2}\times_{Y}} X \\}; 
\path[->,font=\scriptsize,>=angle 90]
(m-1-1) edge node[auto] {$l$} (m-1-2)
(m-2-1) edge node[left] {$c'$} (m-1-1)
(m-2-2) edge node[auto] {$c$} (m-1-2)
(m-2-1) edge node[auto] {$l'$} (m-2-2);
\end{tikzpicture}
\end{center}
where $c$ is the composition of the following isomorphisms
\[
Y_1\: {_{p_2}\times_Y} X = (Y \times_{\SY} Y)\: {_{p_2}\times_Y} X \liso
Y\: {_{q}\times_{qf_0}} X \liso Y \times_{\SY} X,
\]
$c'$ is the isomorphism given by
\[
Y_1\: {_{p_2}\times_{f_0 p_1}} X_1 \liso\, 
(Y_1\: {_{p_2}\times_{f_0}} X) \times_{\SX} X \xto{c \times \id}
(Y \times_{\SY} X) \times_{\SX} X 
\]
and $l' = c^{-1} l {c'}$.
A computation shows that the following composition
\begin{align*}
 Y_1\: {_{p_2}\times_{f_0 p_1}} X_1
&\liso Y_{1}\: {_{p_2}\times_{f_0 p_1}} X_1\: {_{p_2}\times_{X}} X \\
&\xto{\id \times f_1 \times \id}  Y_1\: {_{p_2}\times_{p_1}} Y_1\: {_{p_2}\times_{Y}} X\\
&\xto{s \times_Y X} (Y_1\: {_{p_1}\times_{p_2}} Y_1) {\times_{Y}} X \\
&\xto{m \times \id} Y_1\: {_{p_2}\times_{Y}} X
\end{align*}
agrees with $l'$, where $s$ is the isomorphism that interchanges the factors of the fiber product. But then, considering the isomorphisms $Y_1 \times_Y X = \spec(A_1\: {_{\eta_R} \otimes_{A_0}} B_0)$ and $Y_1 \times_Y X_1 = \spec(A_1\: {_{\eta_R}\otimes_{\eta_L}} B_1)$,  $l'$ is identified the spectrum of the composition
\[
A_1\: {_{\eta_R} \otimes_{A_0}} B_0 \xto{\nabla \otimes \id}
A_1\: {_{\eta_R} \otimes_{\eta_L}} A_1\: {_{\eta_R}\otimes_{A_0}} B_0
                                  \xto{\id \otimes \varphi_1 \otimes \id} 
A_1\: {_{\eta_R} \otimes_{\eta_L}} B_{1\eta_R} \otimes_{B_0} B_0.
\]
Finally, as the map $\CF(l', \iid)$ corresponds to the last map tensored with $M = \CF(X, p)$ this shows the agreement of both morphisms.

The fact that the rows are equalizers defines the promised dashed natural isomorphism $\Upsilon_{\CF}$.
\end{proof}

\begin{cor}\label{adjHoveyInv}
There is a natural isomorphism of functors
\[
\GGamma_p^{\SX} f^*
\liso  B_0\otimes_{A_0}\GGamma_q^{\SY}.
\]
\end{cor}

\begin{proof}
Denote by $\Theta_q^{\SY}$ a quasi-inverse of $\GGamma_q^{\SY}$. The functor $\GGamma_p^{\SX} f^* \Theta_q^{\SY}$ is left adjoint to $U^{\varphi}$, so it has to agree with $B_0\otimes_{A_0} -$.
\end{proof}

\section{Deligne-Mumford stacks and functoriality for the \'etale topology}\label{sec8}

In section 6 we have shown how to obtain an adjoint functoriality of the category of quasi-coherent sheaves on a geometric stack overcoming the difficulty of the lack of functoriality of the $\ff$ topos. But one can (and should) wonder how this construction is related to the case where a functorial topos exists. This is the case for the \'etale topos on a Deligne-Mumford stack, as follows from Theorem \ref{funtopet} below.  Notice that the equivalence between quasi-coherent sheaves on the \emph{lisse-\'etale} and the \emph{\'etale} topology is treated in \cite{lmb} for Deligne-Mumford stacks where functoriality is not a problem. In our setting, our previous construction agrees  with the existing functoriality of the \emph{\'etale} topology. Let us see how.

\begin{cosa}
Let $\SX$ be a geometric Deligne-Mumford stack and consider the category $\aff_{\et}/\SX$ whose objects are pairs $(U,u)$ with $U$ an affine scheme and $u \colon U \to \SX$  an \'etale morphism, morphisms are as in $\aff/\SX$. Note that all morphisms are \'etale.  We define a topology in this category having as coverings finite families of \'etale morphisms that are jointly surjective. We denote also by $\aff_{\et}/\SX$ this site and by $\SX_{\et}$ the associated topos of sheaves.
The category $\aff_{\et}/\SX$ is ringed by the sheaf of rings $\CO \colon \aff_{\et}/\SX \to \ring$, defined by $\CO(U,u) = B$ with $U = \spec(B)$.
We define
\[
\qco_{\et}(\SX) := \qco(\aff_{\et}/\SX, \CO).
\]
As in Theorem \ref{agree}, $\qco_{\et}(\SX)$ agrees with the category of Cartesian sheaves, and we will use this fact freely.
\end{cosa}

Let $f\colon \SX \to \SY$ be a 1-morphism of geometric
\emph{Deligne-Mumford} stacks.

\begin{cosa}\textbf{Direct image for the \'etale topology}.
Let $(V,v)\in \aff_{\et}/\SY$. Denote by $\mathbf{J}_{\et}(v,f)$ the category whose objects are the
$2$-commutative squares
\begin{center}
\begin{equation}\label{masab}
\begin{tikzpicture}[baseline=(current  bounding  box.center)]
\matrix(m)[matrix of math nodes, row sep=2.6em, column sep=2.8em,
text height=1.5ex, text depth=0.25ex]{
U   & V  \\
\SX & \SY\\};
\path[->,font=\scriptsize,>=angle 90]
(m-1-1) edge node[auto]{ $h$} (m-1-2)
        edge node[left] {$u$} (m-2-1)
(m-1-2) edge node[auto] {$v$} (m-2-2)
(m-2-1) edge node[auto] {$f$} (m-2-2);
\draw [shorten >=0.2cm,shorten <=0.2cm,->,double] (m-2-1) -- (m-1-2) node[auto, midway,font=\scriptsize]{$\gamma$};
\end{tikzpicture}
\end{equation}
\end{center}
where $(U,u) \in \aff_{\et}/\SX$. As in \ref{di} (with appropriate changes) we define  for $\CG \in \pre(\aff_{\et}/\SX)$, 
\[
(f_*^{\et}\CG)(V,v)\, = \invlim{\mathbf{J}_{\et}(v,f)} \CG(U,u).
\]
Again, if $f$ is an affine morphism, then
$(f_*^{\et}\CG)(V,v)=\CG(V\times_{\SY}\SX,p_2)$.
\end{cosa}

\begin{cosa}\textbf{Inverse image for the \'etale topology}.
 Now, fix $(U,u)\in \aff_{\et}/\SY$ and denote by $\mathbf{I}_{\et}(u,f)$ the
category whose objects are the $2$-commutative squares as in diagram (\ref{masab}), with $(V,v) \in \aff_{\et}/\SY$. By a similar procedure to \ref{ii} define for $\CF \in \pre(\aff_{\et}/\SY)$ and $(U,u) \in \aff_{\et}/\SX$
\[
(f^{\pin} \CF)(U,u)\, = \dirlim{\mathbf{I}_{\et}{(u,f)}} \CF(V,v).
\]
Notice that if $f$ is  \'etale, then
$(f^{\pin} \CF)(U,u) = \CF(U,fu)$.
\end{cosa}

\begin{prop}\label{adjpre8}
The previous constructions define a pair of adjoint functors
\[
\pre(\aff_{\et}/\SX)
\underset{f_*^{\et}}{\overset{f^{\pin}_{\et}}{\longleftrightarrows}}
\pre(\aff_{\et}/\SY).
\]
\end{prop}

\begin{proof}
 The proof is analogous to the one in Proposition \ref{adjpre6}.
\end{proof}

\begin{prop}\label{p71}
The functor $f^{\pin}:\pre(\aff_{\et}/\SY)\,\to \pre(\aff_{\et}/\SX)\,$ is \emph{exact}.
\end{prop}
\begin{proof}
The functor $f^{\pin}$ is right exact, because it has a right
adjoint. To prove that $f^{\pin}$ is left exact it is enough to show that $\mathbf{I}_{\et}{(u,f)}$ is a cofiltered category for $(U,u) \in \aff_{\et}/\SY$.

By \cite[I, 2.7]{sga41} we have to prove that the category $\mathbf{I}_{\et}{(u,f)}$  is connected and satisfies PS 1) and PS 2) in \emph{loc.~cit.} The category $\mathbf{I}_{\et}{(u,f)}$ is connected since it has finite products (analogous to the $\ff$ case, Lemma \ref{fprod}).

To prove PS 1), for $i\in\{1,2\}$ take $(j_i,\beta_i) \colon (V_i,v_i,h_i,\gamma_i) \to (V,v,h,\gamma)$ morphisms in $\mathbf{I}_{\et}{(u,f)}$, the projections $p_i \colon V_1 \times_V V_2  \to V_i$ and  $h' \colon U \to V_1\times_V V_2$ the morphism satisfying $p_i h' = h_i$. We have that $(V_1\times_V V_2,v_1 p_1,h',\gamma_1)$ is in $\mathbf{I}_{\et}{(u,f)}$ and the morphisms $(p_1,\iid)$ and $\,(p_2,\beta_2^{-1} p_2 \circ \beta_1 p_1)$ satisfy
 \[
(j_1,\beta_1) \circ (p_1,\iid) =
(j_2,\beta_2) \circ (p_2,\beta_2^{-1} p_2 \circ \beta_1 p_1).
\]

Finally, we prove PS 2): Let $(g,\alpha),(g',\alpha') \colon (V,v,h,\gamma) \to
(V',v',h',\gamma')$ be morphisms in $\mathbf{I}_{\et}{(u,f)}$ and consider the pull-back
\begin{center}
\begin{tikzpicture}
\matrix(m)[matrix of math nodes, row sep=2.6em, column sep=2.8em,
text height=1.5ex, text depth=0.25ex]{
V' \times_{\SY} V' & V'\\
V'                       &\SY\\};
\path[->,font=\scriptsize,>=angle 90]
(m-1-1) edge node[auto] {$p_2$} (m-1-2)
        edge node[left] {$p_1$} (m-2-1)
(m-1-2) edge node[auto] {$v'$} (m-2-2)
(m-2-1) edge node[auto] {$v'$} (m-2-2);
\draw [shorten >=0.2cm,shorten <=0.2cm,->,double] (m-2-1) -- (m-1-2) node[auto, midway,font=\scriptsize]{$\beta$};
\end{tikzpicture}
\end{center}
We have the morphism $r \colon V \to V' \times_{\SY} V'$ satisfying $p_1 r = g$, $p_2 r = g'$ and $\beta r = \alpha' \circ \alpha^{-1}$. Consider the Cartesian square
\begin{center}
\begin{tikzpicture}
\matrix(m)[matrix of math nodes, row sep=2.6em, column sep=2.8em,
text height=1.5ex, text depth=0.25ex]
{Z & V'\\
V&V' \times_{\SY} V'\\};
\path[->,font=\scriptsize,>=angle 90]
(m-1-1) edge node[auto] {$q_2$} (m-1-2)
        edge node[left] {$q_1$} (m-2-1)
(m-1-2) edge node[auto] {$\delta$} (m-2-2)
(m-2-1) edge node[auto] {$r$} (m-2-2);
\end{tikzpicture}
\end{center}
and the morphism $l \colon U \to Z$ such that $q_1 l=h$ and $q_2
 l=h'$. Notice that $q_1$ is \'etale because $\delta$ is an open embedding and therefore, \'etale. The object $(Z,vq_1,l,\gamma)\in {\bf
I}_{\et}(u,f)$ and  the morphism $(q_1,\id)\colon
(Z,vq_1,l,\gamma) \to (V,v,h,\gamma)$ is the equalizer of $(g,
\alpha)$ and $(g', \alpha')$.
\end{proof}      

\begin{rem}
Notice that this result, specifically PS 2), is \emph{false} for finer topologies.
\end{rem}

\begin{prop} \label{p81}
If $\CF$ is a sheaf on $\aff_{\et}/\SX$, then $f_*^{\et}\CF$ is a sheaf.
\end{prop}

\begin{proof}
The proof follows the line of Proposition \ref{p61}.
\end{proof}

\begin{thm}\label{funtopet}
Let $f\colon \SX \to \SY$ be a 1-morphism of geometric Deligne-Mumford stacks There is an associated morphism of topoi
\[
f_{\et} \colon \SX_{\et} \lto \SY_{\et}
\]
\end{thm}

\begin{proof}
Let $\CG \in \SY_{\et}$, we denote by $f^{-1}_{\et}\CG$ the sheaf associated to the presheaf $f^{\pin}\CG$. This, together with Proposition \ref{p81} provides an adjunction 
\[
\SX_{\et}
\underset{f_*^{\et}}{\overset{f^{-1}_{\et}}{\longleftrightarrows}}
\SY_{\et}
\]
The pair
$(f_*^{\et},f^{-1}_{\et})$ defines a morphism of topoi, since $f^{-1}_{\et}$
is exact because so is $f^{\pin}$ and the sheafification functor.
\end{proof}

\begin{prop}\label{l82}
 Let $f \colon \SX \to \SY$ and $g \colon \SY \to \SZ$ be 1-morphisms of geometric Deligne-Mumford stacks. It holds that $(g f)^{\et}_* \cong g^{\et}_* f^{\et}_*$
\end{prop}

\begin{proof}
The proof is \emph{mutatis mutandis} the same as Proposition \ref{l62}.
\end{proof}

\begin{cor}
In the previous situation $(g f)^{-1}_{\et} \liso f^{-1}_{\et} g^{-1}_{\et}$.
\end{cor}

\begin{proof}
 It is a consequence of the previous proposition by adjunction.
\end{proof}

\begin{cosa}\label{modet}
 Let $f \colon \SX \to \SY$ be a 1-morphism of geometric Deligne-Mumford stacks. The ring homomorphism $f^{\#}_{\et} \colon {\CO}_{\SY} \to f_*{\CO}_{\SX}$ induces a canonical morphism of ringed topoi
\[
(f_{\et},f_\#^{\et}) \colon (\SX_{\et}, {\CO}_{\SX}) \lto
(\SY_{\et}, {\CO}_{\SY})
\]
where $f_{\#}^{\et} \colon f^{-1}{\CO}_{\SY} \to {\CO}_{\SX}$ is the morphism adjoint to $f^{\#}_{\et}$.  Now we can define the functor $f^*_{\et}\colon \CO_{\SY}\md \to \CO_{\SX}\md$ by the usual formula
\[
f^*_{\et}\CF := \CO_{\SX}\otimes_{f^{-1}_{\et}\CO_{\SY}}f^{-1}_{\et}\CF.
 \]
It is left adjoint to  $f_*^{\et}\colon \CO_{\SX}\md \to \CO_{\SY}\md$. 
\end{cosa}

\begin{prop}
In the previous situation, we have
\begin{enumerate}
 \item if $\CG \in \qco_{\et}(\SY)$, then $f^*_{\et}\CG \in \qco_{\et}(\SX)$, and 
 \item if $\CF \in \qco_{\et}(\SX)$, then $f_*^{\et}\CF \in \qco_{\et}(\SY)$.
\end{enumerate}
\end{prop}

\begin{proof}
This can be shown in a similar way to Propositions \ref{qcoinv} and \ref{qcodir}.
\end{proof}

\begin{cor}
 There is a pair of adjoint functors
\[
\qco_{\et}(\SX) \,
\underset{f^{\et}_*}{\overset{f_{\et}^*}{\longleftrightarrows}} \,
\qco_{\et}(\SY).
\]
\end{cor}

\begin{proof}
Combine the previous proposition with \ref{modet}.
\end{proof}

For $\SX$ a geometric Deligne-Mumford stack we denote by $\sr_{\SX}:\SX_{\ff} \to \SX_{\et}$ the restriction functor.

\begin{prop} \label{p72}
Let $f\colon \SX \to \SY$ be a 1-morphism of geometric
Deligne-Mum\-ford stacks. There is an isomorphism \( \sr_{\SY} f_* \cong f_*^{\et} \sr_{\SX} \) of functors from $\SX_{\ff}$ to $\SY_{\et}$.
\end{prop}

\begin{proof}  As $p$ is an affine morphism, $\sr_{\SX} p_* = p_*^{\et} \sr_X$.
In general, let $\CF\in \SX_{\ff}$. Recall that we have an equalizer in $\SX_{\ff}$
\begin{equation}  \label{etale}
\CF \;\;  \xto{\eta_{\CF}} \;\;  \SST \CF\;\; 
\underset{\eta_{\SST\CF}}{\overset{\SST\eta_{\CF}}{\longrightrightarrows}}
\; \; \SST^2 \CF,
\end{equation}
with $\SST = p_*p^*$. By Proposition \ref{l62} and Proposition \ref{l82}, using also that $f p$ is an affine morphism we have that
\[
\sr_{\SY}f_*p_* \cong \sr_{\SY}(fp)_* = 
(fp)_*^{\et} \sr_X \cong f_*^{\et}p_*^{\et}\sr_X .
\]
From this it follows that $\sr_{\SY}f_*\SST\CF = f_*^{\et} \sr_{\SX}\SST\CF$ and we construct the diagram whose rows are equalizers
\[
\begin{tikzpicture}
\matrix (m) [matrix of math nodes, row sep=3em, column sep=3em]{
\sr_{\SY} f_* \CF       & \sr_{\SY}f_*\SST\CF   & \sr_{\SY}f_*\SST^2\CF \\
f_*^{\et} \sr_{\SX} \CF & f_*^{\et}\sr_{\SX} \SST\CF & f_*^{\et}\sr_{\SX} \SST^2\CF\\
};
\draw [transform canvas={yshift= 0.3ex},->] (m-1-2) -- (m-1-3);
\draw [transform canvas={yshift=-0.3ex},->] (m-1-2) -- (m-1-3);
\draw [transform canvas={yshift= 0.3ex},->] (m-2-2) -- (m-2-3);
\draw [transform canvas={yshift=-0.3ex},->] (m-2-2) -- (m-2-3);
\draw[->] (m-1-2) -- node[auto] {$\wr$} (m-2-2);
\draw[->] (m-1-3) -- node[auto] {$\wr$} (m-2-3);
\draw[dashed, ->] (m-1-1) -- node[auto] {$\wr$} (m-2-1);
\path[->] (m-1-1) edge (m-1-2)
          (m-2-1) edge (m-2-2);
\end{tikzpicture}
\]
that shows existence of the dashed isomorphism. Notice that this isomorphism is clearly compatible with morphisms.
\end{proof}

Notice that if $\CG\in \qco(\SX)$, then $\sr_{\SX}\CG\in
\qco_{\et}(\SX)$, an obvious fact if we consider the Cartesian characterization. We also denote by $\sr_{\SX}$ the induced functor between the categories of quasi-coherent sheaves.

\begin{prop}\label{p73}
The functor $\sr_{\SX}\colon \qco(\SX) \to \qco_{\et}(\SX)$ is an
equivalence of categories.
\end{prop}
\begin{proof}
Let $\des_{\qc}^{\et}(X/\SX)$ be the category whose objects are the
pairs $(\CG, t)$, where $\CG\in \qco_{\et}(X)$ and $t\colon
p_{2,\et}^*\CG \to p_{1,\et}^*\CG$ is an isomorphism satisfying the
usual cocycle condition. An argument similar to the one that proves Proposition \ref{p52} yields an equivalence of categories
\[
\SD_{\et}\colon \qco_{\et}(\SX) \lto \des_{\qc}^{\et}(X/\SX).
\]
Indeed, set $X_1 := X\times_{\SX}X$. We have natural isomomorphisms of functors
\[
w_i \colon 
p_i^*\left(\widetilde{\CG(X)}\right) \lto
\widetilde{p_i^*\!\CG(X_1)} \quad (i \in \{1, 2\}),
\]
defined, for $(U,u)\in\aff_{\ff}/X$, with $U=\spec (B)$ and $u=\spec
(\nu)$ by
\[
{w_i}{(U,u)}\colon B \otimes_{A_0} \CG(X) \liso 
B\: {_{\nu\epsilon}\otimes_{A_1}}{A_1}_{\eta_L}\otimes_{A_0}\CG(X)
\xto{\id\otimes\wadj{\CG(p_i)}}
B\: {_{\nu\epsilon}\otimes_{A_1}} \CG(X_1,p_i).
\]
The induced functor
\[
 \sr_{X/\SX}\colon \des_{\qc}(X/\SX) \lto
 \des_{\qc}^{\et}(X/\SX)
\]
 given by $\sr_{X/\SX}(\CG,t)=(\sr_X\CG, \sr_{X_1}t)$
 is also an equivalence of categories; a quasi-inverse of $\sr_{X/\SX}$ is the functor $\SW_{X/\SX} \colon \des_{\qc}^{\et}(X/\SX)
\to \des_{\qc}(X/\SX)$ given by
\[
\SW_{X/\SX}(\CG, t) =
(\widetilde{\CG(X)}, w_1^{-1} \widetilde{(t(X_1,\id_{X_1}))} w_2).
\]
The result follows from the fact that the
following diagram
\begin{center}
\begin{tikzpicture}
\matrix(m)[matrix of math nodes, row sep=3em, column sep=3.4em,
text height=1.5ex, text depth=0.25ex]{
\qco(\SX)         & \qco_{\et}(\SX) \\
\des_{\qc}(X/\SX) & \des_{\qc}^{\et}(X/\SX) \\};
\path[->,font=\scriptsize,>=angle 90]
(m-1-1) edge node[auto] {$\sr_{\SX}$} (m-1-2)
        edge node[left] {$\SD$} (m-2-1)
(m-1-2) edge node[auto] {$\SD_{\et}$} (m-2-2)
(m-2-1) edge node[auto] {$\sr_{X/\SX}$} (m-2-2);
\end{tikzpicture}
\end{center}
commutes because $\sr_X p^*\!\CG = p^*_{\et}\sr_{\SX}\CG$.
\end{proof}

\begin{rem}
 The same strategy used in Propositions \ref{p72} and \ref{p73} may be used to prove that on an ordinary scheme the quasi-coherent sheaves over the Zariski topology agree with those defined like here on the \'etale topology. We leave the details to the interested readers.
\end{rem}

\begin{cor}\label{final}
In the previous setting, there is a natural isomorphism of functors from $\qco(\SY)$ to $\qco_{\et}(\SX)$
\[
\sr_{\SX} f^* \liso f^*_{\et} \sr_{\SY}.
\]

\end{cor}
\begin{proof}
If the functor $\mathsf{s}_{\SY} \colon \qco_{\et}(\SY)\to \qco(\SY)$ is a quasi-inverse of $\sr_{\SY}$, then the functor $\sr_{\SX} f^* \mathsf{s}_{\SY}$ is left adjoint to $f_*^{\et}$.
\end{proof}

\begin{rem}
The previous discussion expresses the fact that the formalism of functoriality that we established in \S 6 agrees with the usual formalism derived from the \'etale topos map induced from a 1-morphism of Deligne-Mum\-ford stacks on the category of quasi coherent sheaves. 

Recall that, when the stacks are equivalent to schemes, the corresponding theory of quasi-coherent sheaves is equivalent to the usual one in the Zariski topology \cite[VII, 4.3]{sga42}.
\end{rem}


\end{document}